\documentclass[11pt]{amsart} 
\usepackage[bottom=1in]{geometry}
\usepackage{esint} 
\usepackage{graphicx, amssymb, hyperref }
\usepackage[mathscr]{euscript}
  \usepackage{todonotes}
\usepackage[english]{babel}

\usepackage{pdftricks}
\begin{psinputs}
   \usepackage{pstricks}
   \usepackage{multido}
\end{psinputs}
 \usepackage{epsfig}
 \usepackage{pst-grad} 
 \usepackage{pst-plot} 
 \usepackage{xcolor}
\usepackage{caption}  
 \usepackage{mathtext}

\newtheorem{theorem}{Theorem}
\newtheorem{definition}{Definition}
\newcommand{\ssize}{\text{size\,}}
\newcommand{\eenergy}{\text{energy\,}}
\newtheorem{lemma}{Lemma}
\newtheorem{corollary}{Corollary}
\newtheorem{proposition}{Proposition}
\newtheorem*{notation}{Notation:}

\newcommand{\sssize}{\widetilde{\text{size}}\,}

\newcommand{\one}{\mathbf{1}}
\newcommand{\dist}{\text{ dist }}
\newcommand{\rr}{\mathbb}
\newcommand{\ii}{\mathscr}
\newcommand{\ci}{\tilde{\chi}}

\newcommand{\ds}{\displaystyle}

\newcommand{\vBHT}{\overrightarrow{BHT}}
\newtheorem{question}{\underline{Question}}
\newtheorem*{main*}{\underline{Induction statement}}
\newcommand{\lft}{\left|}
\newcommand{\rg}{\right|}
\newcommand{\ic}{\mathcal}
\newcommand{\supp}{\text{supp\,}}

\newtheoremstyle{dotless}{}{}{\itshape}{}{\bfseries}{}{ }{}

\theoremstyle{dotless}
\newtheorem*{remark}{Remark:}

\author{Cristina Benea}
\address{Cristina Benea, Universit\'{e} de Nantes, Laboratoire Jean Leray, Nantes 44322, France}
\email{cristina.benea@univ-nantes.fr}

\author[Camil Muscalu]{Camil Muscalu*}
\thanks{$^*$The author is also a Member of the ``Simion Stoilow" Institute of Mathematics of the Romanian Academy}
\address{Camil Muscalu, Department of Mathematics, Cornell University, Ithaca, NY 14853, USA}
\email{camil@math.cornell.edu}

 \title{Multiple Vector Valued Inequalities via the Helicoidal Method
} 
\begin{document}

\begin{abstract}
We develop a new method of proving vector-valued estimates in harmonic analysis, which we like to call ``the helicoidal method". As a consequence of it, we are able to give affirmative answers to several questions that have been circulating for some time. In particular, we show that the tensor product $BHT \otimes \Pi$ between the bilinear Hilbert transform $BHT$ and a paraproduct $\Pi$ satisfies the same $L^p$ estimates as the $BHT$ itself, solving completely a problem introduced in \cite{bi-parameter_paraproducts}. Then, we prove that for ``locally $L^2$ exponents" the corresponding vector-valued $\vBHT$ satisfies (again) the same $L^p$ estimates as the $BHT$ itself. Before the present work there was not even a single example of such exponents. 

Finally, we prove a bi-parameter Leibniz rule in mixed norm $L^p$ spaces, answering a question of Kenig in nonlinear dispersive PDE.
\end{abstract}

 \maketitle 

\section{Introduction}
\label{introduction}

Vector-valued estimates for classical Calder\'{o}n-Zygmund operators are known from the work of Burkholder \cite{Bourk83},  Benedek, Calder\'{o}n and Panzone \cite{vv_convolution}, Rubio de Francia, Ruiz and Torrea \cite{vv_CZ}, to mention a few. A customary way of proving such vector-valued estimates is through weighted norm inequalities and extrapolation, as explained in \cite{RFCuervaBook}. Initially, the vector-valued approach unified the existing theory for maximal operators, square functions, and singular integrals. Later on, the setting was generalized to Banach spaces which have \emph{unconditional martingale difference} property, and it was shown by Bourgain \cite{Bou86} that this is in fact a necessary condition for this theory.

For bilinear operators however, the theory is far from being fully understood, even in the scalar case. In this paper, we study vector-valued estimates for the bilinear Hilbert transform and for paraproducts. Our initial motivation was an $AKNS$ system-related problem, which can be reduced to understanding a Rubio de Francia operator for iterated Fourier integrals. Because of the specific nature of this question, our general approach is concrete, rather than abstract. As much as possible, the present article is aiming to be self-contained.

Central to time-frequency analysis is the bilinear Hilbert transform operator, defined by 
\begin{equation*}
BHT(f, g)(x)=p.v. \int_{\rr{R}} f(x-t)g(x+t) \frac{dt}{t}.
\end{equation*}

This operator was first introduced by Calder\'{o}n, in connection with his work on the Cauchy integral on Lipschitz curves. $L^p$ estimates for $BHT$ were proved nearly thirty years later, by Lacey and Thiele, without establishing the optimality of the range.

\begin{theorem}[M. Lacey, C. Thiele \cite{initial_BHT_paper}] $BHT$ is a bounded bilinear operator from $L^p \times L^q$ into $L^s$, for any $1< p, q \leq \infty$, $0<s<\infty$, satisfying $\ds \frac{1}{p}+\frac{1}{q}=\frac{1}{s}$ and $\ds \frac{2}{3} <s < \infty$.

\psscalebox{.6 .6} 
{
\begin{pspicture}(-3,-5.2)(11.520909,3)
\rput[bl](4,-2){ \Large $\left(1, 0, 0 \right)$}
\rput[bl](4.55,-3.73205080756888){\Large $\left( 1, \frac{1}{2}, -\frac{1}{2}\right)$}
\rput[b](8,1.5){\Large $\left( 0, 0, 1\right)$}
\rput[br](11.8,-2){ \Large $\left( 0, 1, 0 \right)$}
\rput[bl](9.4,-3.73205080756888){\Large  $\left( \frac{1}{2}, 1, -\frac{1}{2}\right)$}
\pscustom[linecolor=black, linewidth=0.04]
{
\newpath
\moveto(1.2857143,-5.7638702)
}
\pscustom[linecolor=black, linewidth=0.04]
{
\newpath
\moveto(5.0,-5.7638702)
}
\psdiamond[linecolor=black, linewidth=0.04, dimen=outer](8,-2)(2,3.46410161513775)
\pspolygon[linecolor=black, linewidth=0.04, fillstyle=solid,fillcolor=lightgray](8,1.46410161513775)(10,-2)(9,-3.7)(7,-3.7)(6,-2)
\psline[linecolor=black, linewidth=0.04](6,-2)(10,-2)
\end{pspicture}
}

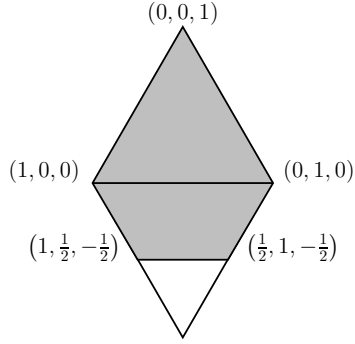
\captionof{figure}{Range for $BHT$ operator}

\end{theorem}
The range of the operator $\ds Range(BHT)$ consists of the set of triples $(p, q, s)$  satisfying the conditions above. The question that remains open is whether the bilinear Hilbert transform is bounded also for $s \in \left(\frac{1}{2}, \frac{2}{3} \right]$. The H\"{o}lder-type condition $\ds \frac{1}{p}+\frac{1}{q}=\frac{1}{s}$ reflects the scaling invariance of the operator, and it can be reformulated as $\ds \frac{1}{p}+\frac{1}{q}+\frac{1}{s'}=1$, where $s'$ is the conjugate exponent of $s$. Thus a triple $(p, q, s) \in Range(BHT)$ if $\ds \left(\frac{1}{p}, \frac{1}{q}, \frac{1}{s'} \right)$ lies in the plane $\ds \lbrace \left(x, y, z \right) \in \rr{R}^3 \vert x+y+z=1 \rbrace$, and is contained inside the convex hull of the points
\[
 \left(0, 0, 1 \right), \left(1, 0, 0 \right),  \left(1,\frac{1}{2}, -\frac{1}{2} \right), \left(\frac{1}{2}, 1,  -\frac{1}{2} \right), \left(0, 1, 0 \right).
\]
Regarded as a bilinear multiplier operator, $BHT$ becomes equivalent to 
\begin{equation}
\label{eq: BHT multiplier}
(f, g) \mapsto \int_{\xi < \eta} \hat{f}(\xi) \hat{g}(\eta) e^{2 \pi i x(\xi +\eta)} d \xi d \eta.
\end{equation}
The method of the proof, which breaks down when $\ds \frac{1}{p}+\frac{1}{q} \geq \frac{3}{2}$, consists in approximating $BHT$ by a \emph{model operator} obtained through a Whitney decomposition of the frequency region $\lbrace \xi< \eta   \rbrace$. Morally speaking, this model operator is a superposition of ``almost orthogonal" objects of a lower complexity, called \emph{discretized paraproducts}.

Paraproducts play an important role on their own, especially in the analysis of PDE. A \emph{paraproduct} is an expression of the form 
\begin{equation}
\label{def paraproducts}
(f, g) \mapsto \int_{\rr{R}}\int_{\rr{R}}f(x-t) g(x-s) k(s, t) ds dt,
\end{equation}
where $\ds k(s,t)$ is a Calder\'{o}n-Zygmund kernel in the plane $\ds \rr{R}^2$. Alternatively, a paraproduct can be regarded as a bilinear multiplier operator
\begin{equation*}
\left(f, g \right) \mapsto \int_{\rr{R}^2} m(\xi, \eta) \hat{f}(\xi) \hat{g}(\eta) e^{2 \pi i x \left( \xi +\eta  \right)} d \xi d \eta,
\end{equation*}
where $m$ is a classical Marcinkiewicz-Mikhlin-H\"{o}rmander multiplier in two variables, sufficiently smooth away from the origin. The singularity of the multiplier $m$ consists of one point: $\left( \xi, \eta \right)=\left( 0, 0\right)$. On the other hand, we can see from \eqref{eq: BHT multiplier} that the $BHT$ multiplier is singular along the line $\xi= \eta$.

We have the following result on paraproducts: 

\begin{theorem}[Coifman, Meyer \cite{CoifMeyer-ondelettes}]
\label{CoifMeyer-paraproducts}
Any bilinear multiplier operator associated to a symbol $m(\xi, \eta)$ satisfying $\ds \lft \partial^\alpha m (\xi, \eta)\rg \lesssim \lft \left( \xi, \eta\right) \rg^{-\alpha}$ for sufficiently many multi-indices $\alpha$, maps $L^p(\rr{R}) \times L^q(\rr{R})$ into $L^s(\rr{R})$ provided that $1< p, q \leq \infty $, $\dfrac{1}{2} < s < \infty$, and $\ds \frac{1}{p}+\frac{1}{q}=\frac{1}{s}$.
\end{theorem}
 Following the presentation in \cite{multilinear_harmonic}, any bilinear operator of this form can be essentially written as a finite sum of paraproducts of the form
\begin{equation}
\tag{I}
\label{eqn:paraprod 1}
(f,g) \mapsto \sum_{k } \left(\left(f \ast \psi_k   \right)   \cdot \left(g \ast \psi_k  \right) \right) \ast \varphi_k(x)= \sum_{k}P_k(Q_kf \cdot Q_k g)
\end{equation}
\begin{equation}
\tag{II}
\label{eqn:paraprod 2}
(f,g) \mapsto \sum_{k } \left(\left(f \ast \varphi_k   \right)   \cdot \left(g \ast \psi_k  \right) \right) \ast \psi_k(x)= \sum_{k} Q_k(P_kf \cdot Q_k g).
\end{equation}
\begin{equation}
\tag{III}
\label{eqn:paraprod 3}
(f,g) \mapsto \sum_{k } \left(\left(f \ast \psi_k   \right)   \cdot \left(g \ast \varphi_k  \right) \right) \ast \psi_k(x)= \sum_{k}Q_k(Q_kf \cdot P_k g).
\end{equation}
From now on, a \emph{paraproduct} will designate any of the expressions $(I), (II)$ or $(III)$, and will be denoted $\Pi(f,g)$. Here $\psi_k(x)=2^k \psi(2^k x)$, $\varphi_k(x)=2^k \varphi(2^k x)$,  $\hat{\varphi}(\xi) \equiv 1$ on $\ds \left[-1/2, 1/2 \right]$, is supported on $[-1,1]$  and $\hat{\psi}(\xi)=\hat{\varphi}(\xi/2)-\hat{\varphi}(\xi) $. The $\ds \lbrace Q_k \rbrace_k$ represent Littlewood-Paley projections onto the frequency $|\xi| \sim 2^k$, while $\ds \lbrace P_k \rbrace_k$ are convolution operators associated with dyadic dilations of a nice bump function of integral 1.

A classical application of Theorem \ref{CoifMeyer-paraproducts} is the following Leibniz rule:
\begin{equation}
\label{Leibniz rule R^1}
\big \|  D^\alpha(f \cdot g)   \big \|_s \lesssim \big \|  D^\alpha f \big \|_{p_1} \big \| g \big \|_{q_1}+\big \| f \big \|_{p_2}\big \|  D^\alpha g \big\|_{q_2}, 
\end{equation}
which holds for any $\alpha>0$, as long as $\ds \frac{1}{p_i}+\frac{1}{q_i}=\frac{1}{s}$, $\ds 1< p_i, q_i \leq \infty$, and $\ds \frac{1}{1+\alpha} < s < \infty$. In particular, if $s \geq 1$, which is the case in most applications, the Leibniz rule holds for any $\alpha >0$.

For functions on $\rr{R}^2$, with (fractional) partial derivatives in both variables, a corresponding Leibniz rule is
\begin{align}
\label{Leibniz rule R^2}
\| D_1^\alpha D_2^\beta(f \cdot g) \|_s & \lesssim \|  D_1^\alpha D_2^\beta f \|_{p_1} \| g \|_{q_1}+\| f \|_{p_2}\|  D_1^\alpha D_2^\beta g \|_{q_2}\\
&+\|  D_1^\alpha f \|_{p_3}\|D_2^\beta g \|_{q_3} +\| D_2^\beta f \|_{p_4} \| D_1^\alpha g \|_{q_4} \nonumber.
\end{align}
The proof of the above inequality relies on discrete biparameter paraproducts $\Pi \otimes \Pi$, which are expressions of the form
\begin{equation}
\sum_{k, l} \left(  \left( f \ast \left(\varphi_k \otimes \psi_l \right) \right) \cdot \left( g \ast \left(\psi_k \otimes \varphi_l  \right) \right)  \right) \ast \psi_k \otimes \psi_l(x,y).
\end{equation}
In \cite{bi-parameter_paraproducts}, the following theorem was proved:
\begin{theorem}[Muscalu, Pipher, Thiele, Tao \cite{bi-parameter_paraproducts}]
\label{thm: biparam paraprod} $\Pi \otimes \Pi$ is a bounded operator from $L^p(\rr{R}^2) \times L^q(\rr{R}^2)$ into $L^s(\rr{R}^2)$ provided that $\ds 1< p, q \leq \infty$, $\ds \frac{1}{p}+\frac{1}{q}=\frac{1}{s}$, and $0< s< \infty$.
\end{theorem}

This further implies that \eqref{Leibniz rule R^2} is true whenever $\ds \frac{1}{p_i}+\frac{1}{q_i}=\frac{1}{s}, 1<p_i, q_i \leq \infty$, and $ \max\left( \frac{1}{1+\alpha}, \frac{1}{1+\beta} \right) < r < \infty$. If $r \geq 1$ the last condition is redundant, so \eqref{Leibniz rule R^2} holds for any $\alpha, \beta >0$.

Related to this, Carlos Kenig asked the following question, that has been circulating for some time:
\begin{question}
\label{Q1}
Assuming that $\ds 1 \leq s_1, s_2 < \infty$, and $\alpha, \beta >0$, is there a Leibniz rule for mixed norm $L^p$ spaces of the form
\begin{align*}
\Big \| D_1^\alpha D^\beta_2 (f \cdot g)  \Big \|_{L^{s_1}_xL^{s_2}_y} &\lesssim \Big \| D_1^\alpha D^\beta_2 f  \Big \|_{L^{p_1}_xL^{p_2}_y} \cdot \|g\|_{L^{q_1}_xL^{q_2}_y} +\|f\|_{L^{p_3}_xL^{p_4}_y} \cdot \Big \| D_1^\alpha D^\beta_2 g  \Big \|_{L^{q_3}_xL^{q_4}_y} \\
&+\Big \| D_1^\alpha f  \Big \|_{L^{p_5}_xL^{p_6}_y} \cdot \Big \| D^\beta_2 g  \Big \|_{L^{q_5}_xL^{q_6}_y} +\Big \| D_2^\beta f  \Big \|_{L^{p_7}_xL^{p_8}_y} \cdot \Big \| D^\alpha_1 g  \Big \|_{L^{q_7}_xL^{q_8}_y} ?
\end{align*}
\end{question}
Here the mixed norms are defined by 
\begin{equation}
\label{def: mixed norms}
\big \|  f  \big \|_{L_x^p L_y^q}:=  \big \| \big \|  f \big \|_{L_y^q}  \big \|_{L_x^p}:=  \left(   \int_{\rr{R}}   \left(   \int_{\rr{R}} \big \vert f(x,y)  \big \vert^q  dy  \right)^{p/q}     dx \right)^{1/p}.
\end{equation}

A result of a similar type appeared in \cite{kenig1993_vvLeibniz}, as an important tool in establishing local well-posedness for the generalized Korteweg-de Vries equation. This is a dispersive, nonlinear equation given by
 \begin{equation}
 \label{eq: gen KdV}
      \begin{cases}
               \frac{\partial u}{\partial t}+\frac{\partial^3 u}{\partial x^3}+u^k \frac{\partial u}{\partial x}=0, \quad \quad t, x \in \rr{R}, k \in \rr{Z}^+\\
                u(x, 0)=u_0(x)
                \end{cases}
 \end{equation}
In order to prove existence, the authors use the contraction principle, but to be able to do so, they need to construct a suitable Banach space. The norm of the Banach space involves mixed $L^p$ norms of fractional derivatives in the first variable $\ds D_1^\alpha$, and the Leibniz rule employed in this paper is 
\begin{equation}
\label{eq: leibniz rule Kenig}
\big \| D_1^\alpha(f \cdot g) -f \cdot D_1^\alpha g -D_1^\alpha f \cdot g   \big \|_{L_x^p L_t^q} \lesssim C \big \|  D_1^{\alpha_1} f \big \|_{L_x^{p_1}L_t^{q_1}}  \big \|  D_1^{\alpha_2} g \big \|_{L_x^{p_2}L_t^{q_2}}.
\end{equation}
Here $\alpha \in (0, 1), \alpha_1+\alpha_2=\alpha$ and $\ds \frac{1}{p_1}+\frac{1}{p_2}=\frac{1}{p}$, $\ds \frac{1}{q_1}+\frac{1}{q_2}=\frac{1}{q}$. Also,  $ p, p_1, p_2, q, q_1, q_2 \in \left( 1, \infty \right)$, but one can allow $q_1=\infty$ if $\alpha_1=0$.
 
The fractional derivatives appear as a consequence of the smoothness requirement on the initial data: $u_0$ is assumed to be in some Sobolev space $H^\alpha(\rr{R})$, where $\alpha$ depends on the value of $k$ in \eqref{eq: gen KdV}.
 
Question \ref{Q1} is an extension of \eqref{eq: leibniz rule Kenig}, and we managed to provide an answer by proving estimates for $\Pi \otimes \Pi$ in $L^p$ spaces with mixed norms.

Bi-parameter bilinear operators where first studied in \cite{JournCZopRF2}, where Journ\'{e} is introducing a new way of generalizing Calder\'{o}n-Zygmund operators on product spaces. More exactly, it is proved in \cite{JournCZopRF2} that ``bi-commutators of Calder\'{o}n-Coifman type" are bounded, which translates to ``$\Pi \otimes \Pi$ maps $L^2(\rr{R}^2) \times L^\infty(\rr{R}^2)$ into $L^2(\rr{R}^2)$''. The full range of estimates for $\Pi \otimes \Pi$ was established in \cite{bi-parameter_paraproducts}, where was also noticed that $BHT \otimes BHT$ does not satisfy any $L^p$ estimates. What remained undecided for some time was the following question:
\begin{question}
\label{Q2}
Does the tensor product $BHT \otimes \Pi$ satisfy any $L^p$ estimates? Would it be possible to prove it satisfies the same estimates as the $BHT$ itself?
\end{question}
Some significant progress in answering this question was made by Silva in \cite{vv_bht-Prabath}. It was showed that $BHT \otimes \Pi$ maps $L^p \times L^q$ into $L^s$ under the constraints that $\ds \frac{1}{p}+\frac{2}{q} <2$ and $\ds \frac{1}{q}+\frac{2}{p}<2$. Our helicoidal method allows us to remove these restrictions, proving in this way that $BHT\otimes \Pi$ satisfies indeed the same $L^p$ estimates as $BHT$.

As it turned out, the study of Question \ref{Q1} and Question \ref{Q2} is related to proving (sometimes multiple) vector-valued inequalities for $\Pi$ and $BHT$. Let $\vec{r}=(r_1, r_2, r)$ be a tuple so that $\ds 1< r_1, r_2 \leq \infty, 1 \leq r < \infty $ and $\ds \frac{1}{r_1}+\frac{1}{r_2}=\frac{1}{r}$. We say that an inequality of the type
\begin{equation}
\label{eq: vvBHT}
\left\| \left(\sum_k \big \vert BHT(f_k, g_k) \big \vert ^r \right)^{1/r}  \right \|_s \lesssim \left \| \left(  \sum_k \big \vert f_k \big \vert^{r_1}  \right)^{r_1}  \right \|_p \left \| \left(  \sum_k \big \vert g_k \big \vert^{r_2}  \right)^{r_2}  \right \|_q
\end{equation}
represents $L^p$ estimates for vector-valued $BHT$, corresponding to the exponent $\vec{r}$; in short, we have $L^p$ estimates for $\overrightarrow{BHT}_{\vec{r}}$.

Some $L^p$ estimates for vector-valued $BHT$ have been proved recently by Silva in \cite{vv_bht-Prabath}, provided $\ds r \in \left( 4/3, 4 \right)$. $UMD$-valued extensions for the quartile operator (the Fourier-Walsh analogue of $BHT$) were studied by Hyt\"{o}nen, Lacey and Parissis in \cite{vv_quartile}. The results in \cite{vv_quartile}, transferred to the $L^p$ setting, hold under the same constraint that $\ds r \in \left( 4/3, 4 \right)$. Moreover, through this method it is impossible to obtain vector-valued extensions when $L^1$ or $L^\infty$ spaces are involved, as these are not $UMD$ spaces. A similar abstract approach was taken in \cite{francesco_UMDparaproducts}, where Banach-valued estimates for paraproducts were proved.

In spite of these results, some important questions remained unsettled:
\begin{question}
\label{Q3}Are there any exponents $\vec{r}$ as before, for which the corresponding vector-valued $\ds \vBHT_{\vec{r}}$ satisfy the same $L^p$ estimates as the $BHT$ itself?
\end{question}

As the question suggests, until the present work, there was not even a single example of such an exponent. We show that whenever $\vec{r}$ is in the ``local $\ell^2$ range"(that is, $\ds 0 \leq \frac{1}{r_1}, \frac{1}{r_2}, \frac{1}{r'} \leq \frac{1}{2}$), $\overrightarrow{BHT}_{\vec{r}}$ satisfies the same $L^p$ estimates as the $BHT$ operator. Moreover, whenever $2 \leq p, q \leq \infty$, $L^p$ estimates exist for any exponent $\vec{r}=(r_1, r_2, r)$.

To summarize, the main task of the present work is to give affirmative answers to Question \ref{Q1}, Question \ref{Q2}, and Question \ref{Q3} described above. In what follows, we will present our main results, sometimes in a more general setting.

\begin{theorem}
\label{Leibniz rule}
For any $\alpha, \beta >0$
\begin{align*}
\Big \| D_1^\alpha D^\beta_2 (f \cdot g)  \Big \|_{L^{s_1}_xL^{s_2}_y} &\lesssim \Big \| D_1^\alpha D^\beta_2 f  \Big \|_{L^{p_1}_xL^{p_2}_y} \cdot \|g\|_{L^{q_1}_xL^{q_2}_y} +\|f\|_{L^{p_3}_xL^{p_4}_y} \cdot \Big \| D_1^\alpha D^\beta_2 g  \Big \|_{L^{q_3}_xL^{q_4}_y} \\
&+\Big \| D_1^\alpha f  \Big \|_{L^{p_5}_xL^{p_6}_y} \cdot \Big \| D^\beta_2 g  \Big \|_{L^{q_5}_xL^{q_6}_y} +\Big \| D_2^\beta f  \Big \|_{L^{p_7}_xL^{p_8}_y} \cdot \Big \| D^\alpha_1 g  \Big \|_{L^{q_7}_xL^{q_8}_y} ,
\end{align*}
whenever $1< p_j,q_j \leq \infty$, $\frac{1}{2}< s_1 <\infty$, $1\leq s_2< \infty$, with $\max \big( \frac{1}{1+\alpha}, \frac{1}{1+\beta}  \big) <s_1$ and so that the indices satisfy the natural H\"{o}lder-type conditions.
\end{theorem}

This answers Question \ref{Q1} in the affirmative. Of course, one may wonder if Theorem \ref{Leibniz rule} holds in arbitrary dimensions. As the careful reader will notice, our methods allow for such a generalization, with the outer-most Lebesgue exponent possibly less than $1$, if all the indices $p_i, q_i$ involved are \emph{strictly} between $1$ and $\infty$. However, in applications $L^\infty$ norms appear, so it will be of interest to have a more general theorem, for $1< p_i, q_i \leq \infty$. Although we cannot obtain this result in this paper due to some delicate technical issues, we plan to return to this problem sometimes in the future.

An $n$-dimensional version of a Leibniz rule was presented in \cite{TorresWardLeibniz}, for indices that are again strictly between $1$ and $\infty$:
{
\fontsize{10}{11}
\[
\Big \| D^\beta_2 (f \cdot g)  \Big \|_{L^{s_1}_xL^{s_2}_y\left( \rr{R}\times \rr{R}^n \right)} \lesssim \Big \|  D^\beta_2 f  \Big \|_{L^{p_1}_xL^{p_2}_y\left( \rr{R}\times \rr{R}^n \right)} \cdot \|g\|_{L^{q_1}_xL^{q_2}_y\left( \rr{R}\times \rr{R}^n \right)} +\|f\|_{L^{p_1}_xL^{p_2}_y\left( \rr{R}\times \rr{R}^n \right)} \cdot \Big \|  D^\beta_2 g  \Big \|_{L^{q_1}_xL^{q_2}_y\left( \rr{R}\times \rr{R}^n \right)}. 
\]
}
This can be regarded as an $n$-dimensional generalization of \eqref{eq: leibniz rule Kenig}, and it is simpler than our variant of the Leibniz rule because it doesn't require a multi-parameter analysis.

Our Theorem \ref{Leibniz rule} is a consequence, modulo technical but ``classical" complications, of the following result:
\begin{theorem}[Mixed norm estimates for paraproducts on the bi-disc]
\label{bi-parameter paraproducts in L^p spaces with mixed norms}
Let $1< p_j, q_j \leq \infty $, $\frac{1}{2}< s_1 <\infty$, $1\leq s_2< \infty$, so that $\ds \frac{1}{p_j}+\frac{1}{q_j}=\frac{1}{s_j}$, $1 \leq j \leq 2$. Then 
\[
\big \| \Pi \otimes \Pi(f, g)    \big \|_{L^{s_1}_x L^{s_2}_y} \lesssim \big \| f  \big\|_{L^{p_1}_xL^{p_2}_y} \big \| g  \big\|_{L^{q_1}_xL^{q_2}_y}.
\]
\end{theorem} 
The above theorem provides $L^p$ estimates for $\ds \Pi \otimes \Pi$  in mixed norm $L^p$ spaces. Through our methods, we can also recover the results from \cite{multi-parameter_paraproducts}, stating that $\ds \Pi \otimes \ldots \otimes \Pi$ maps $L^p(\rr{R}^n) \times L^q(\rr{R}^n)$ into $L^s(\rr{R}^n)$ whenever $ 1< p, q \leq \infty, \frac{1}{2}< s< \infty$ and $\frac{1}{p}+\frac{1}{q}=\frac{1}{s}$. Moreover, we answer Question \ref{Q2} by proving that $BHT\otimes \Pi$ and $BHT \otimes \Pi^{\otimes n}$ satisfy the same $L^p$ estimates as $BHT$:
\begin{theorem}
\label{tensor product BHT d-paraproducts_intro}
For any $p, q, r$ with $\ds \frac{1}{p}+\frac{1}{q}=\frac{1}{r}$, with $1< p, q \leq \infty$ and $2/3 < r < \infty$:
\[
\|  BHT \otimes \Pi \otimes \ldots \otimes \Pi(f, g)   \|_{L^r(\rr{R}^{n+1})} \lesssim \|f\|_{L^{p}(\rr{R}^{n+1})} \| g  \|_{L^q(\rr{R}^{n+1})}.
\]
The same is true for $\ds \Pi \otimes \ldots \otimes \Pi \otimes BHT \otimes \Pi \otimes \ldots \otimes \Pi$.
\end{theorem} 

For $n \geq 2$, no such results were known previously, and furthermore,  a new approach was necessary for $n \geq 3$. This will be explained later in Remark \ref{remark: BHT vector spaces}.

Some mixed norm $L^p$ estimates for $BHT \otimes \Pi$ and $\Pi^{\otimes d_1} \otimes BHT \otimes \Pi^{\otimes d_2}$ can be obtained, which are similar to those in \cite{francesco_UMDparaproducts} in the case $n=1$. These are presented in Section \ref{section proof tensor products}. We recently learned that in \cite{francesco_UMDparaproducts} mixed norm estimates for $\Pi \otimes \Pi$, similar to our Theorem \ref{bi-parameter paraproducts in L^p spaces with mixed norms} are also obtained.

In proving the results mentioned above, multiple vector-valued extensions for $BHT$ and $\Pi$ play a very important role. Given a totally $\sigma$-finite measure space $(\ii{W}, \Sigma, \mu)$, and $f, g: \rr{R}\times \ii{W} \to \rr{C}$, we define
\begin{equation*}
BHT(f, g)(x, w):=p.v.\int_{\rr{R}} f(x-t, w) g(x+t, w) \frac{dt}{t}.
\end{equation*}
Note that for a fixed value $w \in \ii{W}$, we have $\ds BHT(f, g)(x, w)=BHT(f_w, g_w)(x)$, where $f_w(x)=f(x,w)$.

\begin{theorem}
\label{vector valued BHT}
For any triple $\ds (r_1, r_2, r)$ with $1<r_1, r_2 \leq \infty,1 \leq r < \infty$ and so that $\ds \frac{1}{r_1}+\frac{1}{r_2}=\frac{1}{r}$, there exists a nonempty set $\ii{D}_{r_1, r_2, r}$ of triples $(p, q, s)$ satisfying $\ds \frac{1}{p}+\frac{1}{q}=\frac{1}{s}$ for which 
\[
BHT:L^p \left( \rr{R}; L^{r_1}(\ii{W}, \mu) \right) \times L^q\left(\rr{R}; L^{r_2}(\ii{W}, \mu)\right) \to L^s\left( \rr{R}; L^{r}(\ii{W}, \mu)  \right).
\]
This means that there exists a constant $C$ so that 
\[
\Big \| ~ \Big\| BHT(f, g)\Big \|_{L^r(\ii{W}, \mu)} \|_{L^s(\rr{R})} \leq C \big\|~ \big\| f \big \|_{L^{r_1}(\ii{W}, \mu)}\big \|_{L^p(\rr{R})}\cdot \big\|~ \big\| g \big \|_{L^{r_2}(\ii{W}, \mu)}\big \|_{L^q(\rr{R})}.
\]
Depending on the values of $r_1, r_2, r'$, we can give an explicit characterization of $\ii{D}_{r_1, r_2, r}$, as follows:
\begin{itemize}
\item[i)]If $\ds \frac{1}{r_1}, \frac{1}{r_2}, \frac{1}{r'} \leq \frac{1}{2} $, then $\ii{D}_{r_1, r_2, r}=Range(BHT)$.

\item[ii)] If $\ds\frac{1}{r_2},  \frac{1}{r'} \leq \frac{1}{2},  \frac{1}{r_1}>\frac{1}{2}$, then $\ii{D}_{r_1, r_2, r}$ corresponds to those $(p, q, s)\in Range(BHT)$ for which  $\ds 0\leq \frac{1}{q}< \frac{3}{2}-\frac{1}{r_1}.$                                                                                                                                                          
\item[iii)] If $\ds\frac{1}{r_1},  \frac{1}{r'} \leq \frac{1}{2},  \frac{1}{r_2}>\frac{1}{2}$, then  the range of exponents is similar to the  one in $ii)$, with the roles of $r_1$ and $r_2$ interchanged. That is, $\ii{D}_{r_1, r_2, r}$ consists of tuples  $(p, q, s)\in Range(BHT)$ for which $\ds 0\leq \frac{1}{p}< \frac{3}{2}-\frac{1}{r_2}.$

\item[iv)] If $\ds\frac{1}{r_1},  \frac{1}{r_2} \leq \frac{1}{2},  \frac{1}{r'}>\frac{1}{2}$, then $\ii{D}_{r_1, r_2, r}$ corresponds to those $(p, q, s) \in Range(BHT)$ for which $\ds 0 \leq \frac{1}{p}, \frac{1}{q}<\frac{1}{2}+\frac{1}{r}, \quad -\frac{1}{r} < \frac{1}{s'}<1.$
\end{itemize}
\psscalebox{.7 .7} 
{
\begin{pspicture}(-2,-4)(11.720909,1)
\definecolor{colour0}{rgb}{0.8,0.8,0.8}
\pscustom[linecolor=black, linewidth=0.04]
{
\newpath
\moveto(1.2857143,-5.86387)
}
\pscustom[linecolor=black, linewidth=0.04]
{
\newpath
\moveto(5.0,-5.86387)
}
\rput[bl](7.8727274,-1.9963378){ $\left(1, 0, 0 \right)$}
\rput[bl](8.163636,-3.4145195){\large $\left( 1, \frac{1}{2}, -\frac{1}{2}\right)$}
\rput[bl](10.709091,0.7672986){$\left( 0, 0, 1\right)$}
\rput[bl](13.490909,-2.0327015){ $\left( 0, 1, 0 \right)$}
\rput[bl](12.509091,-3.4327013){ $\left( \frac{1}{2}, 1, -\frac{1}{2}\right)$}
\psdiamond[linecolor=black, linewidth=0.04, dimen=outer](11.271428,-1.8352988)(1.7857143,2.5714285)
\pspolygon[linecolor=black, linewidth=0.04, fillstyle=solid,fillcolor=colour0](11.2615385,0.6833825)(13.030769,-1.855079)(12.166,-3.0858483)(10.40,-3.0858483)(9.492308,-1.855079)
\pstriangle[linecolor=black, linewidth=0.04, dimen=outer](5.142857,-1.9495846)(2.5714285,2.0)
\pspolygon[linecolor=black, linewidth=0.04, fillstyle=solid,fillcolor=colour0](4.5384617,-0.96277136)(5.769231,-0.96277136)(5.1538463,-1.8858483)
\psdots[linecolor=black, dotsize=0.1](5.16129,-1.2555753)
\psdots[linecolor=black, dotsize=0.1](5.16129,-1.2555753)
\rput[bl](2.0774193,-0.5523495){\large $\left( \frac{1}{r_1}, \frac{1}{r_2}, \frac{1}{r'}\right)$}
\psline[linecolor=black, linewidth=0.04, arrowsize=0.05291666666666668cm 2.0,arrowlength=1.4,arrowinset=0.0]{->}(2.967742,-0.7394463)(5.032258,-1.2555753)
\psline[linecolor=black, linewidth=0.04](9.492308,-1.855079)(13.030769,-1.855079)
\end{pspicture}
}
\captionof{figure}{Range for vector-valued $BHT$ when $\frac{1}{r_1}, \frac{1}{r_2}, \frac{1}{r'} \leq \frac{1}{2}$}

\psscalebox{.7 .7} 
{
\begin{pspicture}(-2,-4.5)(14,2)
\definecolor{colour0}{rgb}{0.8,0.8,0.8}
\pscustom[linecolor=black, linewidth=0.04]
{
\newpath
\moveto(1.2857143,-5.86387)
}
\pscustom[linecolor=black, linewidth=0.04]
{
\newpath
\moveto(5.0,-5.86387)
}
\psdiamond[linecolor=black, linewidth=0.04, dimen=outer](11.271428,-1.8352988)(1.7857143,2.5714285)
\rput[bl](12.109091,-3.6327014){ $\left( \frac{1}{r_1}, \frac{3}{2}-\frac{1}{r_1}, -\frac{1}{2}\right)$}
\rput[bl](13.090909,-2.2327015){ $\left( 0, 1, 0 \right)$}
\rput[bl](10.709091,0.7672986){$\left( 0, 0, 1\right)$}
\rput[bl](8.163636,-3.4145195){\large $\left( 1, \frac{1}{2}, -\frac{1}{2}\right)$}
\rput[bl](7.8727274,-1.9963378){ $\left(1, 0, 0 \right)$}
\pstriangle[linecolor=black, linewidth=0.04, dimen=outer](5.142857,-1.9495846)(2.5714285,2.0)
\rput[bl](2.0774193,-0.5523495){\large $\left( \frac{1}{r_1}, \frac{1}{r_2}, \frac{1}{r'}\right)$}
\pspolygon[linecolor=black, linewidth=0.04, fillstyle=solid,fillcolor=colour0](4.6153846,-0.8089252)(5.230769,-1.9371303)(3.897436,-1.9371303)
\psdots[linecolor=black, dotsize=0.1](4.5612903,-1.6555753)
\psline[linecolor=black, linewidth=0.04, arrowsize=0.05291666666666668cm 2.0,arrowlength=1.4,arrowinset=0.0]{->}(2.871795,-0.603797)(4.4102564,-1.629438)
\pspolygon[linecolor=black, linewidth=0.04, fillstyle=solid,fillcolor=colour0](11.272727,0.6945713)(9.454545,-1.8508832)(10.454545,-3.2145195)(11.545455,-3.2145195)(12.727273,-1.3963377)(12.727273,-1.3963377)
\psline[linecolor=black, linewidth=0.04](9.454545,-1.8508832)(13.0,-1.8508832)
\rput[bl](12.733334,-1.2448226){$\left(0, \frac{3}{2}-\frac{1}{r_1}, \frac{1}{r_1}-\frac{1}{2} \right)$}
\psline[linecolor=black, linewidth=0.04, arrowsize=0.05291666666666668cm 2.0,arrowlength=1.4,arrowinset=0.0]{->}(12.133333,-3.4448225)(11.6,-3.3114893)
\end{pspicture}
}
\captionof{figure}{Range for vector-valued  $BHT$ when $\frac{1}{r_1} > \frac{1}{2}$}

\psscalebox{.7 .7} 
{
\begin{pspicture}(-2,-4)(14,1)
\definecolor{colour0}{rgb}{0.8,0.8,0.8}
\pscustom[linecolor=black, linewidth=0.04]
{
\newpath
\moveto(1.2857143,-5.86387)
}
\pscustom[linecolor=black, linewidth=0.04]
{
\newpath
\moveto(5.0,-5.86387)
}
\psdiamond[linecolor=black, linewidth=0.04, dimen=outer](11.271428,-1.8352988)(1.7857143,2.5714285)
\rput[bl](13.090909,-2.0327015){ $\left( 0, 1, 0 \right)$}
\rput[bl](10.709091,0.7672986){$\left( 0, 0, 1\right)$}
\rput[bl](8.072727,-1.9963378){ $\left(1, 0, 0 \right)$}
\pstriangle[linecolor=black, linewidth=0.04, dimen=outer](5.142857,-1.9495846)(2.5714285,2.0)
\rput[bl](1.8774194,-1.1523495){\large $\left( \frac{1}{r_1}, \frac{1}{r_2}, \frac{1}{r'}\right)$}
\pspolygon[linecolor=black, linewidth=0.04, fillstyle=solid,fillcolor=colour0](5.142857,0.0504154)(4.499999875,-0.9495846)(5.784714125,-0.9495846)
\psdots[linecolor=black, dotsize=0.1](5.3612905,-0.65557534)
\psline[linecolor=black, linewidth=0.04, arrowsize=0.05291666666666668cm 2.0,arrowlength=1.4,arrowinset=0.0]{->}(3.8367348,-0.80264574)(5.2244897,-0.7210131)
\pspolygon[linecolor=black, linewidth=0.04, fillstyle=solid,fillcolor=colour0](11.25,0.671844)(10.0,-1.1614894)(11.0,-2.578156)(11.5,-2.578156)(12.583333,-1.1614894)
\rput[bl](7.8333335,-3.8448226){$\left( \frac{1}{2}+\frac{1}{r}, \frac{1}{2}, -\frac{1}{r}\right)$}
\rput[bl](12.833333,-3.9448225){$\left( \frac{1}{2}, \frac{1}{2}+\frac{1}{r}, -\frac{1}{r} \right)$}
\rput[bl](12.55,-0.778156){$\left( 0, \frac{1}{2}+\frac{1}{r}, \frac{1}{2}-\frac{1}{r} \right)$}
\rput[bl](7.3333335,-0.728156){$\left( \frac{1}{2}+\frac{1}{r}, 0, \frac{1}{2}-\frac{1}{r} \right)$}
\psline[linecolor=black, linewidth=0.04, arrowsize=0.05291666666666668cm 2.0,arrowlength=1.4,arrowinset=0.0]{->}(8.833333,-3.1614892)(10.833333,-2.6614892)
\psline[linecolor=black, linewidth=0.04, arrowsize=0.05291666666666668cm 2.0,arrowlength=1.4,arrowinset=0.0]{->}(13.666667,-3.1614892)(11.666667,-2.6614892)
\psline[linecolor=black, linewidth=0.04](9.488372,-1.8107141)(13.023255,-1.8107141)
\end{pspicture}
}

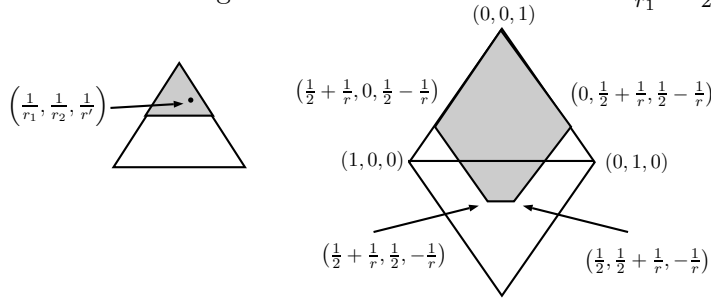
\captionof{figure}{Range for vector-valued  $BHT$ when $\frac{1}{r'} > \frac{1}{2}$}
\end{theorem}

We emphasize that whenever $(p, q, s)$ are so that $\ds 0 \leq \frac{1}{p}, \frac{1}{q} \leq \frac{1}{2}$ (and consequently $1 \leq s < \infty$), vector-valued estimates exist for any tuple $(r_1, r_2, r)$. These are the first examples of tuples $(p, q, s)$ which allow for any $\overrightarrow{BHT}_{\vec{r}}$ extension.

Theorem \ref{vector valued BHT} can be further generalized to multiple vector-valued inequalities. For an $n-$tuple $\ds P=\left(p_1, \ldots, p_n \right)$, the mixed $L^P$ norm on the product space $$\left( \ii{W}, \Sigma, \mu  \right)=\left( \prod_{j=1}^n \ii{W}_j, \prod_{j=1}^n \Sigma_j, \prod_{j=1}^n \mu_j\right)$$ is defined as:
\[
\big \| f  \big \|_{P}:= \left(  \int_{\ii{W}_1}~ \ldots \left(  \int_{\ii{W}_n}  \big\vert f(w_1, \ldots, w_n)    \big \vert^{p_n} d \mu_n(w_n)      \right)^{p_{n-1}/p_n} \ldots d \mu_1(w_1) \right)^{1/{p_1}}.
\]
Consider the tuples $R_1=\left(r_1^1, \ldots, r_1^n \right)$, $R_2=\left(r_2^1, \ldots, r_2^n \right)$ and $R=\left(r^1, \ldots, r^n \right)$ satisfying for every $1 \leq j \leq n$: $ 1 <r_1^j, r_2^j \leq \infty, 1\leq r^j < \infty, \ds \frac{1}{r_1^j}+\frac{1}{r_2^j}=\frac{1}{r^j}$ (from now on, this will be written as  $1<R_1, R_2\leq \infty$, $1 \leq R < \infty$, and $\ds \frac{1}{R_1}+\frac{1}{R_2}=\frac{1}{R}$). Then we have the following multiple vector-valued result:
\begin{theorem}
\label{multiple vector valued BHT}
Let $R_1, R_2$ and $R$ be as above. If the tuples $R_1, R_2, R$ satisfy the condition $\left(r_1^{j}, r_2^{j}, r^{j}\right) \in \ii{D}_{r_1^{j+1}, r_2^{j+1}, r^{j+1}}$ for every $1 \leq j \leq n-1$, then there exists a set $\ii{D}_{R_1, R_2, R}$ of triples $(p, q, s)$ for which
\[
BHT:L^p \left( \rr{R}; L^{R_1}(\ii{W}, \mu) \right) \times L^q\left(\rr{R}; L^{R_2}(\ii{W}, \mu)\right) \to L^s\left( \rr{R}; L^{R}(\ii{W}, \mu)  \right),
\]
In addition, $\ii{D}_{R_1, R_2, R}= \ii{D}_{r_1^1, r_2^1, r^1}$.
\end{theorem}
\begin{remark}
\begin{enumerate}
\item The vector spaces $L^r(\ii{W}_j, \Sigma_j, \mu_j)$ can be both discrete $\ell^r$ spaces or the Euclidean $L^r(\rr{R})$ spaces. For our applications, they are going to be either of these.
\item If the exponents $R_1=\left(r_1^1, \ldots, r_1^n \right)$, $R_2=\left(r_2^1, \ldots, r_2^n \right)$ and $R=\left(r^1, \ldots, r^n \right)$ are in the ``local $L^2$" range, then the multiple vector-valued inequalities hold for any $(p, q, s) \in Range(BHT)$. As particular cases, we mention the following:
\begin{align*}
&BHT: L^p\left( \ell^2 \left(  \ell^\infty \right)\right) \times L^q\left( \ell^\infty \left(  \ell^2 \right)   \right) \to L^s\left( \ell^2 \left(  \ell^2 \right) \right), \\
&BHT: L^p\left( \ell^2 \left(  \ell^\infty \right)\right) \times L^q\left( \ell^2 \left(  \ell^2 \right)   \right) \to L^s\left( \ell^1\left(  \ell^2 \right) \right),
\end{align*}
for any $(p, q, s) \in Range(BHT)$.

 Also, for proving an equivalent of Theorem \ref{tensor product BHT d-paraproducts_intro} in mixed norm spaces, we need the more complex version
\[
BHT: L_x^{p_1} \left( L_y^{p_2} \left(  \ell^{\infty}(\ell^2) \right)\right) \times L_x^{q_1} \left( L_y^{q_2} \left(  \ell^{2}(\ell^2) \right)\right) \to L_x^{s_1} \left( L_y^{s_2} \left(  \ell^{2}(\ell^1) \right)\right).
\]
\item
\label{remark: BHT vector spaces}
As mentioned earlier, multiple vector-valued estimates for $BHT$ play an important role in estimating $BHT \otimes \Pi^{\otimes^n}$. In the case $n=1$, one can obtain estimates for $BHT \otimes \Pi$ in the Banach range by using duality and vector-valued inequalities of the type 
\begin{equation*}
BHT: L^p\left( \ell^2 \right) \times L^q\left( \ell^\infty \right) \to L^s\left( \ell^2 \right) \quad \text{and }\quad BHT: L^p\left(  \ell^\infty \right) \times L^q\left(  \ell^2  \right) \to L^s\left( \ell^2 \right).
\end{equation*}
However, $\ell^1$- valued estimates cannot be avoided for $n \geq 3$, for example if $\ds \Pi \otimes \Pi \otimes \Pi$ has the form
\[
\Pi \otimes \Pi \otimes \Pi(f, g)(x, y, z)=\sum_{k, l, m} Q_k^1 Q_l^2 P_m^3 \left( P_k^1 Q_l^2 Q_m^3 f \cdot Q_k^1 P_l^2 Q_m^3  \right)(x,y,z).
\]
This is in part the novelty of our approach in Theorem \ref{tensor product BHT d-paraproducts_intro}, and it contrasts with the situation of classical Calder\'{o}n-Zygmund operators, where $\ell^1$- valued estimates cannot be expected. 
\item The optimality of the range in Theorem \ref{vector valued BHT} or that in Theorem \ref{multiple vector valued BHT} remains without answer, for now. Since we use in our proofs the model operator for $BHT$, the obstructions appearing are similar to those in \cite{initial_BHT_paper}. These are described in \ref{eq:condition}.
\end{enumerate}
\end{remark}

Equally important are multiple vector-valued inequalities for paraproducts, as they are essential in proving Theorem \ref{Leibniz rule}.
\begin{theorem}
\label{multiple vector valued paraproducts}
For any tuples $R_1=\left(r_1^1, \ldots, r_1^n \right)$, $R_2=\left(r_2^1, \ldots, r_2^n \right)$ and $R=\left(r^1, \ldots, r^n \right)$ satisfying component-wise $1<R_1, R_2\leq \infty$, $1 \leq R < \infty$, and $\ds \frac{1}{R_1}+\frac{1}{R_2}=\frac{1}{R}$, 
\[
\Pi:L^p \left( \rr{R}; L^{R_1}(\ii{W}, \mu) \right) \times L^q\left(\rr{R}; L^{R_2}(\ii{W}, \mu)\right) \to L^s\left( \rr{R}; L^{R}(\ii{W}, \mu)  \right),
\]
provided $1< p, q \leq \infty$, $\ds \frac{1}{2} < s < \infty$, and $\ds \frac{1}{p}+\frac{1}{q}=\frac{1}{s}$.
\end{theorem}
In other words, vector-valued estimates for paraproducts exist within the same range as that of scalar paraproducts. This is also the case with classical Calder\'{o}n-Zygmund operators.

\subsection*{Original Motivation}~\\
We now describe the previously mentioned Rubio de Francia operator for iterated Fourier integrals, and the context where it appeared. AKNS systems are systems of differential equations of the form
\begin{equation}
u'=i \lambda D u +A u
\end{equation}
where $u=[u_1,\ldots, u_n]^t$ is a vector-valued function defined on $\mathbb{R}$, $D$ is a diagonal $n \times n$ matrix with real and distinct entries $d_1,d_2, \ldots d_n$, and $A=\left( a_{jk} (\cdot)  \right)_{j,k=1}^n$ is a matrix valued function defined on $\mathbb{R}$, and so that $a_{jj} \equiv 0$ for all $1 \leq j \leq n$.

Then one would like to prove that the solutions $u_j^\lambda$ (which depend on $\lambda$ as well), are bounded ``for all times"; that is,
\begin{equation}
\label{eq: akns}
\|u_j^\lambda  \|_{\infty} < \infty \quad \text{ for a. e. } \lambda \quad \text{and all } 1\leq j \leq n.
\end{equation}
We want to have such an estimate under the weakest possible assumptions, so we only require the entries of the potential matrix $A$ to be integrable in some $L^p$ spaces:
\[
a_{jk}(\cdot) \in L^{p_{jk}}(\mathbb{R}), \quad \text{for all } 1 \leq j,k \leq n, j \neq k.
\]

In the case of an upper triangular matrix $A$, whose entries are functions $g_k \in L^{p_k}$, the solutions $u_j(t)$ at a fixed time $t$ are a finite sum of expressions of the form
\[
C \int_{x_1< \ldots< x_m <t}g_1(x_1) \ldots g_m(x_m) e^{i \lambda \left( \alpha_1 x_1+\ldots+\alpha_m x_m  \right)}dx_1 \ldots d x_m. 
\]
Here $m \leq n $ and $\alpha_k \neq 0$ for all $k$, as a consequence of $d_1 \neq \ldots \neq d_n$.
Hence the problem \eqref{eq: akns} reduces to estimating
\[
\tilde{C}_m^\alpha \left( g_1,g_2, \ldots, g_m \right)(\lambda):=\sup_t \big \vert\int_{x_1< \ldots< x_m <t}g_1(x_1) \ldots g_m(x_m) e^{i \lambda \left( \alpha_1 x_1+\ldots+\alpha_m x_m  \right)}dx_1 \ldots d x_m \big \vert.
\]
It was proved by Christ and Kiselv \cite{maximal_multilinear_and_filtrations}, \cite{wkb} that $\tilde{C}_m^\alpha $ is a bounded operator:
\[
\| \tilde{C}_m^\alpha \left(g_1, \ldots, g_m  \right) \|_{s_m} \lesssim \prod\limits_{k=1}^m \| g_k \|_{p_k}
\]
for all $1\leq p_k < 2$, such that $\frac{1}{s_m}=\frac{1}{p_1'}+ \ldots + \frac{1}{p_m'}$.

On the other hand, if the entries of the matrix $A$ are $L^2$ functions, the previous expression becomes equivalent to
 \begin{align}
 \label{eq: bi-est}
&\sup_{t} \vert  \int_{x_1< \ldots< x_m <t}\hat{f}_1(x_1)  \ldots \hat{f}_m(x_m) e^{i \lambda \left( \alpha_1 x_1+\ldots+\alpha_m x_m  \right)}dx_1 \ldots d x_m \vert
\end{align}
denoted  $C_m^\alpha (f_1, \ldots, f_m)(\lambda)$. For $m=1$, this is exactly the Carleson operator, while $m=2$ corresponds to the Bi-Carleson operator of \cite{bi-Carleson}, both of which are known to be bounded operators(with the remark that for the Bi-Carleson, the $\alpha_k$s need to satisfy some non-degeneracy condition):
\[
\| C_2^\alpha(h_1,h_2) \|_{s_2} \lesssim \|h_1  \|_{p_1} \| h_2  \|_{p_2}
\]
for $1< p_1, p_2\leq \infty$, $\displaystyle \frac{1}{s_2}=\frac{1}{p_1}+\frac{1}{p_2}$, and $\frac{2}{3} < s_2 <\infty$.


Moreover, if instead of considering the $\sup$ in the expression \eqref{eq: bi-est}, we look at the limiting behavior $\ds \lim_{t \to \infty} u_j(t)$, then we encounter iterated Fourier integrals: for example, the $BHT$ operator as seen in \eqref{eq: BHT multiplier}, or the Bi-est operator of \cite{biest}:
\[
 \int_{\xi_1<\xi_2<\xi_3} \hat{f_1}(\xi_1) \hat{f_2}(\xi_2) \hat{f_3}(\xi_3) e^{2 \pi i x(\xi_1+\xi_2+\xi_3)} d \xi_1 d \xi_2 d \xi_3.
\]

Now we consider the following mixed problem: the matrix $A$ is the sum of a lower triangular matrix with entries $\hat{f}_k \in L^2$, and an upper triangular matrix with entries $g_k \in L^{p_k}$, where $1 \leq p_k <2$. Using Picard iteration, the solutions $u_j(t)$ can be expressed as a series of terms of the form
{
\fontsize{9}{10}
\[
 C \cdot \int_R \hat{f}_{11}(\xi_{11}) \ldots \hat{f}_{1m_1}(\xi_{1m_1})  \ldots g_{21}(x_{21}) \ldots g_{2n_2}(x_{2 n_2}) \hat{f}_{l1}(\xi_{l1}) \ldots \hat{f}_{lm_l}(\xi_{l m_l}) dx d\xi,
\]
}
where $R=\lbrace \xi_{11}<\ldots \xi_{1m_1}< \ldots < x_{21}< \ldots x_{2n_2}< \ldots <\xi_{l1}< \ldots < \xi_{l m_l} <t \rbrace$.

The simplest of these operators, where $\sup$ is dropped, is given by
\begin{equation}
\label{hybrid operator}
M(f_1, f_2, g)(\xi)= \int_{x_1<x_2<x_3} \hat{f}_1(x_1) \hat{f}_2(x_2) g(x_3) e^{2 \pi i \xi (x_1+x_2+x_3)} dx_1 d x_2 d x_3,
\end{equation}
where $f_1 \in L^{p_1}, f_2 \in L^{p_2}$, $1 < p_1, p_2 < \infty$, and $g \in L^p$ with $1<p<2$.
The techniques from \cite{maximal_multilinear_and_filtrations}, \cite{wkb},\cite{absolutely_continuous_spectrum}, akin to those used by Paley in \cite{PaleyOrthFun}, are based on a dyadic filtration associated to one of the functions. This involves a structure on $\rr{R}$ similar to that of the dyadic mesh: on every level of the filtration, one has a partition of $\rr{R}$, and passing to the next level of the filtration means refining the previous partition. We want to use $g$ in order to obtain this structure and for simplicity we assume $\|g\|_p=1$. Define the function 
\[
\varphi(x)=\int_{- \infty}^x |g(y)|^p dy.
\]
Its image is the unit interval $[0,1]$, and the filtration will consist of pre-images through $\varphi$ of the collection $\mathcal{D}$ of dyadic intervals in $[0,1]$. Because $\varphi$ is increasing, whenever $x_2< x_3$ we have $0 \leq \varphi(x_2) \leq \varphi(x_3) \leq 1$. Hence there exists a unique dyadic interval $\omega \subset [0,1]$ so that $\varphi(x_2)$ is contained in the left half of $\omega$, which we denote $\omega_L$, while $\varphi(x_3)$ is contained in the right half $\omega_R$. To simplify notation, we identify $\varphi^{-1}(\omega)$ with $\omega$.

Then the operator $M$ can be written as
\begin{align}
&\sum_{\omega \in \mathcal{D}} \int_{\substack{ x_1<x_2\\ x_2 \in \omega_L ,  x_3 \in \omega_R   }} \hat{f}_1(x_1) \hat{f}_2(x_2) g(x_3) e^{2 \pi i \xi (x_1+x_2+x_3)} d x_1 d x_2 d x_3  \nonumber \\
&= \sum_{\omega} \int_{\substack{ x_1< x_2, x_1, x_2 \in \omega_L \\ x_3 \in \omega_R }} \hat{f}_1(x_1) \hat{f}_2(x_2) g(x_3) e^{2 \pi i \xi (x_1+x_2+x_3)} d x_1 d x_2 d x_3\label{breaking op 1}\\
&+\sum_{\omega} \int_{\substack{x_1< L(\omega_L), x_2 \in \omega_L \\ x_3 \in \omega_R}} \hat{f}_1(x_1) \hat{f}_2(x_2) g(x_3) e^{2 \pi i \xi (x_1+x_2+x_3)} d x_1 d x_2 d x_3. \label{breaking op 2}
\end{align}

Here $L(\omega_L)$ denotes the left endpoint of the interval $\omega_L$. We call the operators in \eqref{breaking op 1} and \eqref{breaking op 2} $M_1$ and $M_2$ respectively. The first term $M_1$ accounts for the occurrence of arbitrary intervals (they are in fact $\varphi^{-1}(\omega_L)$), and this combined with H\"{o}lder's inequality motivates the operator 
\begin{equation}
\label{RF for iterated Fourier integrals}
T_r(f,g)(x)=\left(  \sum_{k=1}^N \big|   \int\limits_{a_k < \xi_1<\xi_2<b_k} \hat{f}(\xi_1) \hat{g}(\xi_2) e^{2\pi i x(\xi_1 +\xi_2)} d\xi_1 d\xi_2  \big| ^r  \right) ^{\frac{1}{r}}.
\end{equation}

We have the following result:
\begin{theorem}
\label{main theorem}
If  $1 \leq r \leq 2$, then
\[
\| T_r(f,g) \|_s \lesssim \|f\|_p  \|g\|_q
\]
whenever $\dfrac{1}{p}+\dfrac{1}{q}=\dfrac{1}{s}$, and $p, q, s$ satisfy
\[
0 \leq \frac{1}{p}, \frac{1}{q}<\frac{1}{2}+\frac{1}{r}, \qquad -\frac{1}{r'}<\frac{1}{s'}<1.
\]

On the other hand, if $r \geq 2$, $T_r$ is a bounded operator with the same range as the $BHT$ operator.

\psscalebox{.7 .7} 
{
\begin{pspicture}(0,-4.421844)(15.3,2)
\definecolor{colour0}{rgb}{0.8,0.8,0.8}
\pscustom[linecolor=black, linewidth=0.04]
{
\newpath
\moveto(1.2857143,-5.86387)
}
\pscustom[linecolor=black, linewidth=0.04]
{
\newpath
\moveto(5.0,-5.86387)
}
\psdiamond[linecolor=black, linewidth=0.04, dimen=outer](11.271428,-1.8352988)(1.7857143,2.5714285)
\rput[bl](13.090909,-2.0327015){ $\left( 0, 1, 0 \right)$}
\rput[bl](10.709091,0.7672986){$\left( 0, 0, 1\right)$}
\rput[bl](8.072727,-1.9963378){ $\left(1, 0, 0 \right)$}
\psline[linecolor=black, linewidth=0.04](9.488372,-1.8107141)(13.023255,-1.8107141)
\pspolygon[linecolor=black, linewidth=0.04, fillstyle=solid,fillcolor=colour0](11.271428,0.7361297)(10.0,-1.1614894)(11.0,-2.578156)(11.5,-2.578156)(12.583333,-1.1614894)
\rput[bl](7.8333335,-3.8448226){$\left( \frac{1}{2}+\frac{1}{r'}, \frac{1}{2}, -\frac{1}{r'}\right)$}
\rput[bl](12.833333,-3.9448225){$\left( \frac{1}{2}, \frac{1}{2}+\frac{1}{r'}, -\frac{1}{r'} \right)$}
\rput[bl](12.55,-0.778156){$\left( 0, \frac{1}{2}+\frac{1}{r'}, \frac{1}{2}-\frac{1}{r'} \right)$}
\rput[bl](7.3333335,-0.728156){$\left( \frac{1}{2}+\frac{1}{r'}, 0, \frac{1}{2}-\frac{1}{r'} \right)$}
\psline[linecolor=black, linewidth=0.04, arrowsize=0.05291666666666668cm 2.0,arrowlength=1.4,arrowinset=0.0]{->}(8.833333,-3.1614892)(10.833333,-2.6614892)
\psline[linecolor=black, linewidth=0.04, arrowsize=0.05291666666666668cm 2.0,arrowlength=1.4,arrowinset=0.0]{->}(13.666667,-3.1614892)(11.666667,-2.6614892)
\end{pspicture}
}

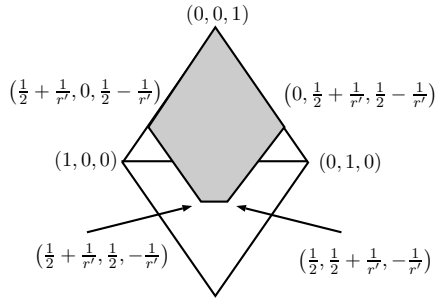
\captionof{figure}{Range for $T_r$ operator for $1 \leq r \leq 2$}

\end{theorem}

In Section \ref{Rubio de Francia Theorem for Iterated Fourier Integrals} we will show how both $M_1$ and $M_2$ are bounded operators: 
\begin{theorem} The operators $M_1$ and $M_2$ satisfy the following:
\[
M_1 : L^{p_1} \times L^{p_2} \times L^p \to L^q \quad \text{provided } 1<p<2 \text{ and } \frac{1}{p_1}+\frac{1}{p_2}+\frac{1}{p'}=\frac{1}{q} , \quad \text{while}
\]
\[
M_2 : L^{p_1} \times L^{p_2} \times L^p \to L^q \quad \text{provided } 1<p<2,  \frac{1}{p_2}+\frac{1}{p'}<1 \text{ and } \frac{1}{p_1}+\frac{1}{p_2}+\frac{1}{p'}=\frac{1}{q}.
\]
Hence $M=M_1+M_2$ is a bounded operator from $\ds  L^{p_1} \times L^{p_2} \times L^p \to L^q$ provided $\ds 1<p<2,  \frac{1}{p_2}+\frac{1}{p'}<1 \text{ and } \frac{1}{p_1}+\frac{1}{p_2}+\frac{1}{p'}=\frac{1}{q}$.

\end{theorem}
However, as Robert Kesler noticed in \cite{Robert-MixedDeg}, the boundedness of the operator $M$ can also be proved by making use of a vector-valued extension for the $`$linear' operator $BHT( f_1, \cdot)$. The constraint for the exponents is given by $\ds \frac{1}{p_2}+\frac{1}{p'}<1$. So even if $M$ splits as $M=M_1+M_2$ and the range of $M_1$ is larger, one gets the same range for $M$ through both methods.

Because the intervals $\ds \lbrace \left[ a_k, b_k  \right]  \rbrace_k$ are disjoint and arbitrary, we refer to $T_r$ as a \emph{bilinear Rubio de Francia operator for iterated Fourier integrals}. Recall that \emph{Rubio de Francia's square function} is the operator
\[
 f \mapsto  RF(f)(x):= \left(  \sum_{k=1}^N  \vert  \int_{I_k} \hat{f}(\xi) e^{2 \pi i \xi x} d \xi   \vert^2   \right)^{1/2}=    \left( \sum_{k=1}^N   \vert  P_{I_k}f(x)  \vert ^2  \right)^{1/2},
\]
where $\lbrace  I_k=\left[ a_k,b_k \right] \rbrace_{1 \leq k \leq N}$ is a family of disjoint intervals, and $P_I(f)$ denotes the Fourier projection of $f$ onto the interval $I$. Using vector-valued singular integrals theory, Rubio de Francia \cite{RF} proved the boundedness of the $RF$ operator on $L^p$, for $p \geq 2$. Interpolating this result with estimates for Carleson's operator from \cite{initial_Carleson}, one gets more generally that the operator
\[
RF_\nu(f)(x):=\left( \sum_{k=1}^N   \vert  P_{I_k}f(x)  \vert ^\nu  \right)^{1/\nu}
\]
is bounded on $L^p$, as long as $\dfrac{1}{p}+\dfrac{1}{\nu} <1$.

In the particular case of a $``$\emph{lacunary}'' family of intervals (that is, $I_k=[2^{k-1}, 2^k]$ and $k \in \rr{Z}$), the above operator corresponds to a Littlewood-Paley square function with sharp cutoffs, which is bounded on $L^p(\rr{R})$ for any $1< p < \infty$. Even more, the $L^p$ norm of the square function is comparable to the $L^p$ norm of the initial function: 
\[
C_p^{-1} \|f\|_p \leq \| \left( \sum_{k \in  \rr{Z}} \vert  \int_{\rr{R}} \one_{ \lbrace 2^{k-1}  \leq \xi <2^k   \rbrace } \hat{f}(\xi) e^{2 \pi i x \xi} d \xi   \vert ^2   \right)^{1/2}\|_p \leq C_p \|f\|_p.
\]
Rubio de Francia's theorem addresses the boundedness of a square function associated to an arbitrary family of intervals, and in this sense it is optimal: in the case $\nu=2$, the condition $p \geq 2$ is necessary, while for $\nu>2$, we need the strict inequality $\nu > p'$.

Returning to our operator $T_r$, note that it can also be regarded as a vector-valued bilinear Hilbert transform
\[
T_r(f,g)(x)=\left( \sum_k |BHT(P_{I_k}f, P_{I_k} g)(x)|^r   \right)^{1/r},
\]
because the multiplier of the $BHT$ operator is equivalent to $\ds \one_{\lbrace \xi_1 < \xi_2  \rbrace}$, as seen in \eqref{eq: BHT multiplier}.

Using solely Khintchine's inequality, it was proved in \cite{bilinear_disc_multiplier} that:
\[
\|  \left( \sum_k |BHT(f_k, g_k)|^2   \right)^{1/2}  \|_s  \lesssim \| \left( \sum_k |f_k|^2   \right)^{1/2}  \|_{p} \| \left( \sum_k |g_k|^2   \right)^{1/2}  \|_q.
\]
This implies the boundedness of $T_r$ for $r \geq 2$, $p, q \geq 2$. But this is a very limited range, and in order to obtain estimates in the case $p<2$ or $q<2$ one needs the full power of vector-valued extensions. 

We note that our estimates for the operator $T_r$ are sharp, in the sense that the same estimates are satisfied by
\begin{equation}
\label{eq: RF+Holder}
(f, g) \mapsto \left( \sum_k \big \vert P_{I_k}f(x) \cdot P_{I_k} g (x) \big \vert^r   \right)^{1/r}.
\end{equation}
In \eqref{eq: RF+Holder}, $BHT(P_{I_k}f, P_{I_k}g)$ is replaced by the product of the functions $P_{I_k}f \cdot P_{I_k}g$. In general, the best one can hope for a bilinear Fourier multiplier operator is that it satisfies the same $L^p$ estimates as the product $\ds (f, g) \mapsto f \cdot g$, and this is the case for $T_r$.

Moreover, in the special case of lacunary dyadic intervals, for any $1 \leq r <\infty$, we have that
\begin{equation*}
(f, g) \mapsto \left(  \sum_k   \big \vert  \int_{2^k < \xi < \eta <2^{k+1}} \hat{f}(\xi) \hat{g}(\eta) e^{2 \pi i x(\xi+\eta)} d \xi d \eta     \big \vert^r \right)^{1/r}
\end{equation*}
is a bounded operator from $L^p \times L^q$ to $L^s$ for any $(p, q, s) \in Range(BHT)$. The cases $p=\infty$ or $q=\infty$ cannot be obtained directly, but follow by duality. 

Our initial proof of Theorem \ref{main theorem} did not involve vector-valued bilinear Hilbert transform operators, but it was built around localizations of $BHT$, in conjunction with several stopping times. Afterwards we realized that this method is suitable for other general situations, which eventually led to the development of the helicoidal method. This applies to paraproducts, $BHT$, Carleson operator, Rubio de Francia operator, etc. In the study of the $T_r$ operator, the stopping times were dictated by level sets of linear Rubio de Francia operators: $RF_{r_1}(f)$ and $RF_{r_2}(g)$. For the vector-valued $BHT$, the three stopping times that are used for estimating the trilinear form are dictated by level sets of $\left( \sum_k \lft f_k\rg^{r_1} \right)^{\frac{1}{r_1}}$, $\left( \sum_k \lft g_k\rg^{r_2} \right)^{\frac{1}{r_2}}$ and $\left( \sum_k \lft h_k\rg^{r'} \right)^{\frac{1}{r'}}$. The method of the proof is described in more detail in Section \ref{subsec:method_of_the_proof}.

Lastly, we want to point out an interesting connection with another open problem in time-frequency analysis: the boundedness of the Hilbert transform along vector fields. More exactly, if $v : \rr{R}^2 \to \rr{R}^2$ is a non-vanishing measurable vector field, then one defines the Hilbert transform along $v$ as
\[
H_v f(x,y)=p.v. \int_{\rr{R}} f\left((x, y)-t \cdot v(x,y) \right) \frac{dt}{t}.
\]
It was conjectured by Stein that $H_v$ is a bounded operator on $L^2$ whenever $v$ is Lipschitz. Some partial results in this direction are known in the case of a one-variable vector field. In \cite{BatemanThiele2011}, M. Bateman and C. Thiele proved the $L^p$ boundedness of $H_v$ for $\ds \frac{3}{2}< p <\infty$, and provided that $\ds v(x,y)=v(x,0)$.

The proof is making use of the Littlewood-Paley square function in the second variable, restrictions to certain fixed sets $G$ and $H$, together with single annulus estimates for $H_v$ from \cite{Bateman-annulus}. In the special case when $f(x,y)=g(x)h(y)$, estimates for the variational Carleson from \cite{variational_Carleson} yield the same result whenever $\ds p> \frac{4}{3}$. It is still not known if this can be extended to general functions $f(x,y)$, or whether one can push the lower bound for $p$ below $4/3$.

In \cite{vv_bht-Prabath}, Silva is using ideas similar to the ones described above, obtaining in this way vector-valued extensions for $BHT$ whenever $\ds \frac{4}{3}<r < 4$. Our methods allow us to prove that vector-valued extensions exist for any $1 \leq r < \infty$ (in fact, for any triple $(r_1, r_2, r)$). It would be interesting to understand whether the localization argument that we are employing can be transferred to the study of the Hilbert transform along vector fields.

Besides having sharp estimates for the local version of the operator, the structure of the intervals chosen through the triple stopping time can play a role in itself. The collections of intervals constitute a maximal covering for the level sets of certain maximal operators, and for that reason, they form a \emph{sparse} collection of intervals (in the sense of \cite{Lerner-simplerA_2}). From here, weighted estimates can be deduced, and a similar approach was carried out in \cite{weighted_BHT}. 

The rest of the paper is organized as follows: in Section \ref{some classical results BHT} we recall some definitions and results regarding multilinear operators. The helicoidal method is described in details in Section \ref{subsec:method_of_the_proof}. Multiple vector-valued extensions for $BHT$ are presented in Section \ref{sec: multiple vector estimates for BHT}, and those for paraproducts in Section \ref{paraproduct results}. Following in Section \ref{tensor products} are the estimates for $BHT\otimes \Pi^{\otimes^n}$. The Leibniz rules are a modification of mixed norm $L^p$ estimates for $\Pi \otimes \Pi$ and are discussed in Section \ref{sec: Leibniz rule}. The Rubio de Francia theorem for iterated Fourier integrals and its application to the AKNS system problem appear in  Section \ref{Rubio de Francia Theorem for Iterated Fourier Integrals}.

\subsection*{Acknowledgements} The first author was partially supported by NSF grant DMS 1500262  and ERC project FAnFArE no. 637510; the second author was partially supported by NSF grant DMS 1500262.

\section{Some Classical Results on the Bilinear Hilbert Transform}
\label{some classical results BHT}
In this paper we use Chapter 6 of \cite{multilinear_harmonic} as a black box, but we recall a few definitions and results to ease the reading of the presentation. Essential here are the notions of \emph{size} and \emph{energy}, which are quantities associated to certain subsets of the phase-frequency space.

\begin{notation}
For any interval $I \subset \rr{R}$, define
\[
\ci_{I}(x):= \left( 1+ \frac{\dist(x, I)}{|I|}  \right)^{-100}.
\]

The mesh of dyadic intervals is denoted by $\mathcal{D}$.
\end{notation}

\begin{definition} A \emph{tile} is a rectangle $P=I_P \times \omega_P$ with the property that $I_P, \omega_P \in \mathcal{D}$ or $\omega_P$ is in a shifted variant of $\mathcal{D}$. We define a \emph{tri-tile} to be a tuple $P=(P_1, P_2, P_3)$ where each $P_i$ is a tile as defined above and the spatial intervals are the same: $I_{P_i}=I_P$ for all $1 \leq i \leq 3$.
\end{definition}

\begin{definition}[Order relation] Given two tiles $P$ and $P'$, we say $P' < P$ if $I_{P'} \subsetneq I_P$ and $\omega_P \subset 3 \omega_{P'}$. $P' \leq P$ if $P'< P$ or $P'=P$.
Also, $P' \lesssim P$ if $I_{P'} \subset I_P$ and $\omega_P \subseteq 100 \omega_{P'}$, and $P' \lesssim ' P$ if $P' \lesssim P$ but $ P' \nleq P$.
\end{definition}

\begin{definition}A collection $\rr{P}$ of tri-tiles is said to have \emph{ rank 1} if for any $P, P' \in \rr{P}$ the following conditions are satisfied:
\begin{itemize}
\item[-] if the tri-tiles are distinct $P \neq P'$, then $P'_j \neq P_j$ for all $ 1\leq j \leq 3$.   
\item[-] if $\omega_{P_{j_0}}=\omega_{P'_{j_0}}$ for some $j_0$, then $\omega_{P_j}=\omega_{P'_j}$ for all $1 \leq j \leq 3$.
\item[-] if $P'_{j_0} \leq P_{j_0}$ for some $j_0$, then $P'_j \lesssim P_j$ for all $1 \leq j \leq 3$.
\item[-] if in addition to $P'_{j_0} \leq P_{j_0}$ one also assumes $|I_{P'}| << |I_P|$, then $P'_{j} \lesssim ' P_j$ for all $j \neq j_0$.
\end{itemize}
\end{definition}

\begin{definition}
Let $\rr{P}$ be a sparse rank 1 collection of tri-tiles, and let $1 \leq j \leq 3$. A subcollection $T$ of $\rr{P}$ is called a \emph{$j$ -tree} if and only if there exists a tri-tile $P_T$ (called the \emph{top} of the tree) such that $\ds P_j \leq P_{T, j}$ for all $P \in T$. We write $I_T$ for $I_{P_T}$  and $\omega_{T_j}$ for $\omega_{P_T,j}$ and we say $T$ is a \emph{tree} if it is a $j$- tree for some $1 \leq j \leq 3$.
\end{definition}

\begin{definition}
Let $1 \leq i \leq 3$. A finite sequence of trees $T_1, \ldots, T_M$ is said to be a chain of strongly $i$-disjoint trees if and only if
\begin{itemize}
\item[(i)]$P_i \neq P'_i$ for every $P \in T_{l_1}$ and $P' \in T_{l_2}$, with $l_1 \neq l_2$;
\item[(ii)] whenever $P \in T_{l_1}$ and $P' \in T_{l_2}$ with $l_1 \neq l_2$ are such that $2 \omega_{P_i} \cap 2 \omega_{P'_i} \neq \emptyset$, then if $\vert \omega_{P_i} \vert < \vert \omega_{P'_i} \vert $ one has $I_{P'} \cap I_{T_{l_1}} =\emptyset$, and if $\vert \omega_{P'_i} \vert < \vert \omega_{P_i} \vert $, one has $I_P \cap I_{T_{l_2}}=\emptyset$.
\item[(iii)]whenever $P \in T_{l_1}$ and $P' \in T_{l_2}$ with $l_1 < l_2$ are such that $2 \omega_{P_i} \cap 2 \omega_{P'_i} \neq \emptyset$ and $\vert \omega_{P_i} \vert = \vert \omega_{P'_i}  \vert$, then $I_{P'} \cap I_{T_{l_1}}= \emptyset$. 
\end{itemize}
\end{definition}

\begin{definition}
Let $P$ be a tile. A \emph{wave packet} on $P$ is a smooth function $\phi_P$ which has Fourier support inside $\dfrac{9}{10} \omega_P$ and is $L^2$- adapted to $I_P$ in the sense that
\begin{equation}
|\phi_P^{(l)}(x)| \leq C_{l,M} \frac{1}{|I_P|^{1/2+l}} \left( 1+ \frac{\dist(x, I_P)}{|I_P|}   \right)^{-M}
\end{equation}
for sufficiently many derivatives $l$, and any $M>0$.
\end{definition}

\subsection{Model Operator for $BHT$} ~\\
A discretized model operator for $BHT$ is given by
\begin{equation}
\label{eq: BHT model operator}
BHT_{\rr{P}}(f,g)(x)=\sum_{P \in \rr{P}} \frac{1}{|I_P|^{1/2}} \langle f, \phi_{P_1}^1 \rangle  \langle g, \phi_{P_2}^2 \rangle \phi^3_{P_3}(x)
\end{equation}
where the family $\rr{P}$ of tri-tiles is sparse and has rank $1$, while $(\phi_{P_j}^j)_{P \in \rr{P}}$ are wave packets associated to the tiles $P_j$. In some sense, the bilinear Hilbert transform is the canonical example of such an operator. Above we also included the definitions of \emph{trees} and \emph{chains of strongly disjoint trees} because they are essential in understanding such singular bilinear operators.

The model operator from \ref{eq: BHT model operator} was introduced in \cite{initial_BHT_paper}, and the bilinear Hilbert transform itself can be represented as an average of such shifted model operators. The detailed reduction can be found in \cite{multilinear_harmonic}, Chapter 6. As a consequence, the boundedness of the bilinear Hilbert transform within $Range(BHT)$ can be deduced from similar estimates for the model operator. Similarly, estimates for vector-valued and for the localized bilinear Hilbert transform will follow once we prove their equivalents for the model operator, and we will not insist on the exact distinction between the two. 

It is worth mentioning however, that the model operator fails to be bounded for $s \leq \frac{2}{3}$, leaving undecided the boundedness of the bilinear Hilbert transform itself for $\frac{1}{2}< s \leq \frac{2}{3}$.

Bilinear operators are often studied with the use of the associated trilinear form. In the case of the (model operator for) $BHT$ operator, the trilinear form is given by
\begin{equation}
\Lambda_{BHT; \rr{P}} (f, g, h)=\sum_{P \in \rr{P}} \frac{1}{|I_P|^{1/2}} \langle f, \phi_{P_1}^1 \rangle \langle g, \phi_{P_2}^2 \rangle \langle h, \phi_{P_3}^3 \rangle.
\end{equation}

\begin{definition}
If $\rr{P}$ is a collection of tri-tiles and $I_0$ is a dyadic interval, we denote by $\rr{P}(I_0)$ the tiles $P$ in $\rr{P}$ whose spatial interval $I_P$ is contained in $I_0$:
\[
\rr{P}(I_0):=\lbrace P \in \rr{P}: I_P \subseteq I_0 \rbrace.
\]
\end{definition}

\begin{definition}
\label{def size}
Let $\mathbb{P}$ be a finite collection of tri-tiles, let $j \in \lbrace 1, 2, 3  \rbrace$, and let $f$ be an arbitrary function. We define the \emph{size} of the sequence $\displaystyle  \langle f, \Phi_{P_j}^j  \rangle_P$ by
\begin{equation}
\ssize \left(   \langle f, \phi_{P_j}^j  \rangle_P \right):= \sup_{T \subseteq \mathbb{P}} \left( \frac{1}{|I_T|} \sum_{P \in T} \vert   \langle f, \phi_{P_j}^j  \rangle  \vert ^2 \right)^{1/2},
\end{equation}  
where $T$ ranges over all trees in $\mathbb{P}$ that are $i$-trees for some $i \neq j$.
\end{definition}

\begin{lemma}[Lemma 6.13 of \cite{multilinear_harmonic}]
\label{BHT size estimate}
Let $j \in \lbrace 1,2,3  \rbrace$ and let $E$ be a set of finite measure. Then for every $|f|\leq \one_E$ one has
\[
\text{size}\left(\langle f, \phi_{P_j}^j  \rangle_P  \right) \lesssim \sup_{P \in \mathbb{P}} \frac{1}{|I_P|} \int_E \tilde{\chi}_{I_P}^M dx
\]
for all $M>0 $, with implicit constants depending on $M$.
\end{lemma}

Thanks to Lemma \ref{BHT size estimate}, which is a consequence of the John-Nirenberg inequality, we can work with the simpler ``sizes"
\[
\ssize \left( f \right) \sim \sup_{P \in \rr{P}} \frac{1}{|I_P|} \int_{\rr{R}} |f| \cdot \ci_{I_P}^M dx,
\]
where $M$ is some large number to be chosen later.

WE will also need a size that behave well with respect to localization. In the formula above we consider the supremum over the spacial intervals $I_P$ of the collection $\rr P$. In our proofs, we will need to compare $\ssize_{\rr P \left( I_0 \right)} f$ and $\ds \frac{1}{\lft  I_0\rg} \int_{\rr R} \lft f \rg \cdot \ci_{I_0} dx$ so the following definition is natural:

\begin{definition}
\label{def:modified-size}
If $I_0$ is a fixed dyadic interval, then we define
\begin{equation}
\label{eq:mod-size}
\sssize_{\rr P \left( I_0 \right)} f :=\sup_{\substack{J \subseteq 3 I_0 \\ \exists P \in \rr P \left( I_0\right), I_P \subseteq J}} \frac{1}{\lft J \rg } \int_{\rr R} \lft f \rg \cdot \ci_{J}^M dx. 
\end{equation}
We note that for any function $f$,
\[
\ssize_{\rr P \left( I_0 \right)} f \leq \sssize_{\rr P \left( I_0 \right)} f.
\]
\end{definition}

\begin{definition}
\label{def:energy}Let $\mathbb{P}$ be a  finite collection of tri-tiles,  $j \in \lbrace 1,2,3 \rbrace$ and let $f$ be  a fixed function. We define the $\eenergy$ of the sequence $\displaystyle  \langle f, \Phi_{P_j}^j  \rangle_P$ by
\begin{equation}
\eenergy \left( \langle f, \phi_{P_j}^j  \rangle_P \right):=\sup_{n \in \mathbb{Z}} 2^n \sup_{\mathbb{T}} \left( \sum_{T \in \mathbb{T}}\vert I_T \vert  \right)^{1/2}
\end{equation}
where $\mathbb{T}$ ranges over all chains of strongly $j$-disjoint trees in $\mathbb{P}$ (which are $i$-trees for some $i \neq j$) having the property that
\[
\left(  \sum_{P \in T} \vert \langle f, \phi_{P_j}^j  \rangle  \vert ^2 \right)^{1/2} \geq 2^n \vert  I_T \vert^{1/2}
\]
for all $T \in \mathbb{T}$ and such that
\[
\left( \sum_{P \in T'} \vert \langle f, \phi_{P_j}^j  \rangle \vert ^2 \right)^{1/2} \leq 2^{n+1} |I_{T'}|^{1/2}
\]
for all subtrees $T' \subseteq T \in \mathbb{T}$.
\end{definition} 
 
We have the following estimates for the trilinear form and $\eenergy$: 

\begin{proposition}[Prop. 6.12 of \cite{multilinear_harmonic}]
\label{BHT trilinear estimate}
 Let $\mathbb{P}$ be a finite collection of tri-tiles. Then
\[
\Lambda_{BHT;\mathbb{P}}(f_1,f_2,f_3) \lesssim \prod _{j=1}^{3} \left( \text{size }(\langle f_j, \phi_{P_j}^j  \rangle_P)  \right)^{\theta_j}  \left( \text{energy }(\langle f_j, \phi_{P_j}^j  \rangle_P)\right)^{1-\theta_j}
\]
for any $0 \leq \theta_1, \theta_2, \theta_3 <1$ with $\theta_1+\theta_2+\theta_3=1$ ; the implicit constants depend on the $\theta_j$ but are independent of the other parameters.
\end{proposition}

\begin{lemma}[Lemma 6.14 of \cite{multilinear_harmonic}]
\label{BHT energy estimate}
 Let $j \in \lbrace  1,2,3 \rbrace $ and $f \in L^2(\mathbb{R})$. Then
\[
\text{energy}\left( \langle f, \phi_{P_j}^j  \rangle_P  \right) \lesssim \| f \|_2.
\]
\end{lemma}
However, for our specific problem we need more accurate estimates for the localized trilinear form. This will follow in Section \ref{subsec:technicalLemmas} and in Section\ref{Estimates for Localized $BHT$}.

\bigskip
\subsection{Interpolation}
\label{interpolation}~\\

Since this is a fundamental tool in harmonic analysis, we recall a few facts about interpolation methods. We adapt the results from \cite{wave_packet} and emphasize how the constants change through interpolation. In our applications, we need to keep track of the constants. Many of the proofs in the following sections are iterative, and the operatorial norm obtained after interpolation becomes a ``size" on the subsequent step of the induction. 
We recall a few definitions and results, but we will be mainly using their generalization to Banach spaces.

\begin{definition}
For  a subset $E \subset \rr{R}$ of finite measure, define 
\[
X(E)= \lbrace  f: |f| \leq \one_E \text{  a.e.}  \rbrace.
\]
We will denote by $V$ the linear span of all $X(E)$, which plays an important role because it is a dense subspace of all $L^p$ spaces, for $1 \leq p < \infty$.
\end{definition}

\begin{definition}
A tuple $\alpha=(\alpha_1, \ldots, \alpha_n)$ is called \emph{admissible} if for all $1 \leq i \leq n$ 
\[
-\infty < \alpha_i <1 \quad \text{and    } \alpha_1+\ldots+\alpha_n=1,
\]
and there is at most one index $j_0$ so that $\alpha_{j_0}<0$. We call an index \emph{good} if $\alpha_i >0$ and \emph{bad} if $\alpha_i \leq 0$. 
\end{definition}

\begin{definition}
A multilinear form $\Lambda: V \times \ldots \times V \to \rr{C}$ is of restricted type $\alpha=(\alpha_1, \ldots, \alpha_n)$ with $0 \leq \alpha_i \leq 1$ if there exists a constant $C$ (possibly depending on $\alpha$) such that for each tuple $E=(E_1, \ldots , E_n)$ of measurable subsets of $\rr{R}$ and for each tuple $f=(f_1, \ldots, f_n)$ with $f_j \in X(E_j)$, we have
\[
|\Lambda(f_1, \ldots, f_n)| \leq C \prod_j |E_j|^{\alpha_j}.
\]
\end{definition}

\begin{theorem} \label{interp thm restricted type}
[Similar to Theorem 3.2 in \cite{wave_packet}] Let $\beta=(\beta_1, \ldots, \beta_n)$ be a tuple of real numbers such that $\sum_j \beta_j=1$ and $\beta_j > 0$ for all $j$. Assume $\Lambda$ is of restricted type $\alpha$ for all $\alpha$ in a neighborhood of $\beta$ satisfying $\sum_{j} \alpha_j =1$, with constant $C(\alpha)$ depending continuously on $\alpha$. Then $\Lambda$ is of strong type $\beta$ with constant $C(\beta)$:
\[
|\Lambda(f_1, \ldots, f_n)| \leq C(\beta) \prod_{j=1}^n \|   f_j \|_{1/{\beta_j}} \quad \text{for all  } f_j \in V.
\] 
\end{theorem}

For multilinear operators, it often happens that the target space is an $L^p$ space with $0<p<1$. This is not a Banach space, but we can conclude the desired outcome by interpolating weak-$L^q$ estimates, for $q$ in a neighborhood of $p$. $L^{q, \infty}$ norms are dualized in the following way:
\begin{lemma}[Lemma 2.5 from \cite{multilinear_harmonic}] Let $0<r \leq 1$, and $A>0$. Then the following statements are equivalent:
\begin{itemize}
\item[i)] $\|f\|_{r, \infty} \leq A$
\item[ii)] for every set $E$ with $0< \vert E \vert < \infty$, there exists a major subset $E' \subseteq E$( i. e. $\vert E' \vert \geq \vert E \vert /2$) so that $\ds \vert \langle f, \one_{E'}  \rangle \vert \lesssim A \vert E \vert^{1/{r'}}$, where $\ds \frac{1}{r}+\frac{1}{r'}=1$. (Note that for $r \neq 1$, $r'$ is a negative number).
\end{itemize}
\end{lemma}

\begin{definition}
Let $\alpha$ be an $n$- tuple of real numbers and assume $\alpha_j \leq 1$ for all $j$. An $n$-linear form $\Lambda$ is called of \emph{generalized restricted type $\alpha$} if there is a constant $C$ (possibly depending on $\alpha$) such that for all tuples $E=(E_1, \ldots, E_n)$, there is an index $j_0$ and a major subset $E'_{j_0} \subseteq E_{j_0}$ so that for all tuples $f=(f_1, \ldots, f_n)$ with $f_j \in X(E_j)$ for $j \neq j_0$ and $f_{j_0} \in X(E'_{j_0})$,
\begin{equation}
|\Lambda(f_1, \ldots, f_n)| \leq C \prod_{j=1}^n |E_j|^{\alpha_j}.
\end{equation}
\end{definition}

If a tuple $\alpha=\left(\alpha_1, \ldots, \alpha_n \right)$ is good, then generalized restricted type estimates coincide with restricted type estimates:
\begin{proposition}[Similar to Lemma 3.6 in \cite{wave_packet}]
If $\alpha=(\alpha_1, \ldots, \alpha_n)$ is a good tuple, and $\Lambda$ is of generalized restricted type $\alpha$ with constant $C(\alpha)$ and the major subset corresponds to the index $j_0$, then $\Lambda$ is of restricted type $\alpha$ with constant $\ds \frac{C(\alpha)}{1-2^{-j_0}}$.
\end{proposition}

\begin{theorem}(Thm. 3.8 of \cite{wave_packet})
\label{interpolation for bad index}
 Assume $\ds \Lambda=\langle T(f_1, \ldots , f_{n-1}) , f_n \rangle $ is of generalized restricted type $\beta$ where $\sum_j \beta_j=1$. Assume $\beta_k >0$ for $1 \leq k \leq n-1$ and $\beta_n \leq 0$.
Assume $\Lambda$ is also of generalized restricted type $\alpha$ with constant $C(\alpha)$ (continuously depending on $\alpha$) for all $\alpha$ in a neighborhood of $\beta$ satisfying $\sum_j \alpha_j=1$. Then the multilinear operator $T$ satisfies
\begin{equation}
\|  T(f_1, \ldots , f_{n-1}) \|_{1/{(1-\beta_n)}} \leq C(\beta) \prod_{j=1}^{n-1} \| f_j   \|_{1/{\beta_j}}.
\end{equation}
\end{theorem}

\subsection{Interpolation for  Banach-valued Functions}~\\
The Banach space interpolation theory is very similar to the scalar version, the difference consisting in replacing the norm $\vert \cdot \vert$ on $\rr{C}$ by $\| \cdot \|_X$ on a Banach space $X$.

We say that $F \in L^p(\rr{R}; X)$ provided
\[
\|F\|_{L^p(\rr{R};X)}:= \left( \int_{\rr{R}} \| F(x)  \|^p_X  dx  \right)^{1/p} < \infty.
\]
The question of integrability of $F(x)$ is reduced to the Lebesgue integrability of $\ds x \mapsto \| F(x) \|_X$. The set of vector-valued step functions is dense in $L^p(\rr{R}; X)$ and for this reason, similarly to the scalar case, it will be enough to deal with function in 
\[
\lbrace F: \| F(x)\|_{X} \leq \one_E(x) \text{  a.e.} , E \subset \rr{R} \text{  subset of finite measure }  \rbrace.
\]
The linear span of such sets will be denoted $V_X$.

The multilinear form associated with an operator is obtained through dualization. More exactly, 
\[
\| F \|_{L^p(\rr{R}; X)}:= \sup\limits_{\|  G  \|_{L^{p'}(\rr{R}; X^*)} \leq 1} \big\vert \int_{\rr{R}} \langle G(x), F(x)   \rangle dx   \big\vert,
\]
whenever $1 \leq p < \infty$.

We will deal with a vector-valued multilinear (or multi-sublinear) operator of the form
\[
\vec{T} : L^{p_1}(\rr{R}; X_1) \times \ldots \times  L^{p_{n-1}}(\rr{R}; X_{n-1}) \to  L^{p_n}(\rr{R}; X_n).
\]
The multilinear form associated with this operator, $\Lambda : V_{X_1}\times \ldots  \times V_{X_{n-1}}\times V _{X_n^*} \to \rr{C}$ is given by:
\[
\Lambda(F_1,\ldots F_{n-1}, F_n )=\int_{\rr{R}} \langle \vec{T}(F_1, \ldots, F_{n-1})(x), F_n(x)  \rangle dx.
\]
The definitions and proofs from the scalar case are adaptable to the vector-valued situation. For completeness,  we present them here, adapting the equivalent statements from \cite{wave_packet}.

\begin{definition} A tuple $\alpha=(\alpha_1, \ldots, \alpha_n)$ is called admissible if $\alpha_1+ \ldots+\alpha_n=1$, $\alpha_1, \ldots, \alpha_n <1$ and for at most one index $j_0$ we have $\alpha_{j_0} <0$.

A multi-sublinear form $\Lambda$ as above is of restricted type $\alpha=(\alpha_1, \ldots, \alpha_n)$ for a good admissible tuple $\alpha$ if there exists a constant $C$ so that for each tuple $E=(E_1, \ldots, E_n)$ of measurable subsets of $\rr{R}$, and for each tuple $F=(F_1, \ldots, F_n)$ with $\| F_j \|_X \leq \one_{E_j}$, we have 
\[
\vert \Lambda(F_1, \ldots, F_n) \vert \leq C |E_1|^{\alpha_1} \cdot \ldots \cdot |E_n|^{\alpha_n}.
\]
\end{definition}

\begin{proposition}[Equivalent of Thm. 3.2 of \cite{wave_packet}]
Let $\beta=(\beta_1, \ldots, \beta_n)$ be an admissible tuple of real numbers such that $\beta_j>0$ for all $j$. Assume that $\Lambda$ is of restricted type $\alpha$ for all admissible tuples $\alpha$ in a neighborhood of $\beta$. Then there is a constant $C$ such that for all $F_j \in V_{X_j}$,
\[
\vert \Lambda(F_1, \ldots , F_n)   \vert \leq C \|F_1\|_{L^{1/{\beta_1}}(\rr{R}; X_1)} \cdot \ldots \cdot \|F_n\|_{L^{1/{\beta_n}}(\rr{R}; X_n)}.
\]
\end{proposition}

\begin{definition}
Let $\alpha$ be an admissible tuple; the $n$-sublinear form $\Lambda$ is of \emph{generalized restricted type} $\alpha$ if there is a constant $C$ such that for all tuples $E=(E_1, \ldots, E_n)$ there is an index $j_0$ and a major subset $E_{j_0}'$ of $E_{j_0}$(that is, $ |E_{j_0}'| \geq |E_{j_0}|/2$) such that for all tuples $F=(F_1, \ldots, F_n)$ with $\| F_j \|_{X_j} \leq \one_{E_j}$ for $j \neq j_0$, and $\|F_{j_0}\|_{X_{j_0}} \leq \one_{E_{j_0}'}$, we have 
\[
\vert \Lambda(F_1, \ldots, F_n)  \vert \leq C \prod_j |E_j|^{\alpha_j}.
\]
\end{definition}

\begin{proposition}
If $\Lambda$ is of generalized restricted type $\alpha=(\alpha_1, \ldots, \alpha_n)$, and $\alpha_j>0$ for all $j$, then $\Lambda$ is of restricted type $\alpha$.
\end{proposition}

On the other hand, if one of the indices $\alpha_j$ is $\leq 0$, the generalized restricted type implies only weak-$L^p$ estimates. This works in the case when the multi-sublinear form is given by
\begin{equation}
\label{vv multilinear form}
\Lambda(F_1, \ldots, F_n)= \int_{\rr{R}}\langle \vec{T}(F_1, \ldots, F_{n-1})(x), F_n(x) \rangle dx,
\end{equation}
and corresponds to an operator $\vec{T}$ defined on $V_{X_1} \times \ldots \times V_{X_{n-1}}$ and taking values in $V_{X_n}$.

\begin{proposition}
Let $\Lambda$ be a multi-sublinear form as in \eqref{vv multilinear form}, and $\alpha=(\alpha_1, \ldots, \alpha_n)$ an admissible tuple with $\alpha_n \leq 0$. Assuming that $\Lambda$ is of generalized restricted type $\alpha$, we have 
\[
\lambda  \vert \lbrace x: \| \vec{T}(F_1, \ldots, F_{n-1})(x) \|_{X_n} > \lambda   \rbrace  \vert^{1/{1-\alpha_n}} \leq A \prod_{j=1}^{n-1} |E_j|^{\alpha_j}
\] 
for all tuples $F=(F_1, \ldots, F_{n-1})$ with $\|f_j\|_{X_j} \leq \one_{E_j}$. 
\end{proposition}

\begin{proposition}
\label{banach interpolation gen res}
Assume $\Lambda$ is of generalized restricted type $\beta$ where $\beta$ is an admissible tuple with $\beta_n\leq 0$. Assume $\Lambda$ is also of generalized restricted type $\alpha$ for all  admissible tuples $\alpha$ in a neighborhood of $\beta$. Then $\vec{T}$ satisfies
\begin{equation}
\|  \vec{T}(F_1, \ldots F_{n-1}) \|_{L^{1/{1-\beta_n}}(\rr{R}; X_n)} \leq C \prod_{j=1}^{n-1} \| F_j \|_{L^{1/{\beta_j}}(\rr{R}; X_j)}.
\end{equation}
\end{proposition}

The proofs of the last two propositions follow exactly the same ideas as those corresponding to the scalar case, with very minor differences.

\subsection{A few technical Lemmas}~\\
\label{subsec:technicalLemmas}
In this section, we present a few results that will be useful later on for estimating a trilinear form associated to a collection $\rr P$ of tri-tiles well localized in space: $I_P \subset I_0$ for all $P \in \rr P$.

\begin{lemma}
\label{lemma:localized_energy}
If $I_0$ is a fixed dyadic interval, $k \in \rr Z^+$, and $f$ is a function so that $\ds 2^{k-1} \leq \frac{\dist( \supp f, I_0)}{\lft I_0 \rg} \leq 2^k$, then
\[
\eenergy_{\rr P\left(I_0\right)} f \lesssim 2^{Mk} \left\| f \right\|_2.
\] 
\begin{proof}

Following Definition \ref{def:energy}, there exists a collection $\rr T$ of $j$-disjoint trees $T \in \rr T \subseteq \rr P\left( I_0 \right)$, so that 
\[
\left( \eenergy_{\rr P \left(I_0\right)} f \right)^2 \sim \sum_{T \in \rr T} \sum_{P \in T} \lft \langle f, \phi_{P_j} \rangle \rg^2.
\]

We denote $\ds \ic T:= \bigcup_{T \in \rr{T}} \bigcup_{P \in T} P$ the collection of all tiles in $\rr T$, and estimate the RHS of the expression above in the following way:
\begin{align*}
\sum_{T \in \rr T} \sum_{P \in T} \lft \langle f, \phi_{P_j} \rangle \rg^2 \lesssim \sum_{m \geq 0} \sum_{\substack{I \subseteq I_0 \\ \lft I \rg=2^{-m } \lft I_0 \rg}} \sum_{ \substack{P \in \ic T \\ I_P =I}} \lft \langle f, \phi_{P_j} \rangle \rg^2.
\end{align*}

The collection of tiles $P \in \ic T$ with $I_P=I$ for a fixed interval $I$ are all disjoint in frequency. In fact, since they are of the same scale, they are translations of some fixed tile and hence 
\begin{align*}
\sum_{ \substack{P \in \ic T \\ I_P =I}} \lft \langle f, \phi_{P_j} \rangle \rg^2 &\lesssim \int_{\rr R} \lft f(x) \rg^2 \cdot \left( 1+\frac{\dist \left( x, I \right)}{\lft I \rg}   \right)^{-2M} dx. 
\end{align*}

This will imply that 
\begin{align*}
\sum_{T \in \rr T} \sum_{P \in T} \lft \langle f, \phi_{P_j} \rangle \rg^2 &\lesssim \sum_{m \geq 0} \sum_{\substack{I \subseteq I_0 \\ \lft I \rg=2^{-m } \lft I_0 \rg}} \int_{\rr R} \lft f(x) \rg^2 \cdot \left( 1+\frac{\dist \left( x, I \right)}{\lft I \rg}   \right)^{-2M} dx \\
&\lesssim  \sum_{m \geq 0} \sum_{\substack{I \subseteq I_0 \\ \lft I \rg=2^{-m } \lft I_0 \rg}} \left\| f \right\|_2^2  \cdot 2^{-2kM} \cdot \left(\frac{\lft I_0 \rg}{\lft I \rg}   \right)^{-2M} \\
&\lesssim \left\| f \right\|_2^2  \cdot 2^{-2kM}  \sum_{m \geq 0} 2^{-mM} \lesssim \left\| f \right\|_2^2  \cdot 2^{-2kM}.
\end{align*}
\end{proof}
\end{lemma}

On the other hand, if $f$ is supported inside $5 I_0$, we know already from Lemma \ref{BHT energy estimate}, that $\ds \eenergy_{\rr P\left( I_0 \right)} f \lesssim \left\| f \right\|_2$.

Since the collection $\rr P \left( I_0 \right)$ is localized in space on the interval $I_0$, we have the following estimate for the trilinear form $\Lambda_{BHT; \rr P \left( I_0 \right)}$:

\begin{lemma}[Refinement of Proposition 6.12 of \cite{multilinear_harmonic}]
\label{lemma:refined-trilinear-estimate}
The trilinear form $ \Lambda_{BHT; \rr P \left( I_0 \right)}$ satisfies
\begin{align}
\label{eq:refined-trilinear-est}
\lft  \Lambda_{BHT; \rr P \left( I_0 \right)} \left( f, g, h \right) \rg & \lesssim \left( \ssize_{\rr P \left( I_0 \right)} f  \right)^{\theta_1}  \left( \ssize_{\rr P \left( I_0 \right)} g  \right)^{\theta_2}  \left( \ssize_{\rr P \left( I_0 \right)} h  \right)^{\theta_3}\\
& \left\| f \cdot \ci_{I_0}  \right\|_2^{1-\theta_1} \cdot \left\| g \cdot \ci_{I_0}  \right\|_2^{1-\theta_2} \cdot \left\| h \cdot \ci_{I_0}  \right\|_2^{1-\theta_3}, \nonumber
\end{align}
for any $0 \leq \theta_1, \theta_2, \theta_3 <1$, with $\theta_1+\theta_2+\theta_3=1$; the implicit constants depend on the $\theta_j$, but are independent of the other parameters.
\begin{proof}

For any $l \geq 1$, we define $\ds \ic I_l:= 2^{l+1}I_0 \setminus 2^{l} I_0 $, and $\ic{I}_0:=2 I_0$. In this way, for any $x \in \ic I_l$,  $\ds 1+ \frac{\dist \left( x, I_0 \right)}{\lft I_0 \rg} \sim 2^l$.

We will be using the following decompositions:
\begin{equation}
\label{eq:dec_f-shell}
f:=\sum_{k_1 \geq 0} f_{k_1}:= \sum_{k_1 \geq 0} f \cdot \one_{\ic I _{k_1}},
\end{equation}
and similarly, 
\[
g:=\sum_{k_2 \geq 0} g_{k_2}:= \sum_{k_2 \geq 0} g \cdot \one_{\ic I _{k_2}}, \qquad h:=\sum_{k_3 \geq 0} h_{k_3}:= \sum_{k_3 \geq 0} h \cdot \one_{\ic I _{k_3}}.
\]

From Proposition \ref{BHT trilinear estimate}, the trilinear form can be estimated by
\begin{align*}
\lft \Lambda_{BHT; \rr P \left( I_0 \right)} \left( f, g, h \right)\rg &\lesssim \sum_{k_1, k_2, k_3} \lft \Lambda_{BHT; \rr P \left( I_0 \right)} \left( f_{k_1}, g_{k_2}, h_{k_3} \right)\rg \\
&\lesssim \sum_{k_1, k_2, k_3} \left( \ssize_{\rr P \left( I_0\right)} f_{k_1} \right)^{\theta_1}  \left( \ssize_{\rr P \left( I_0\right)} g_{k_2} \right)^{\theta_2}  \left( \ssize_{\rr P \left( I_0\right)} h_{k_3} \right)^{\theta_3}  \\
&\qquad \left( \eenergy_{\rr P \left( I_0 \right)} f_{k_1} \right)^{1-\theta_1} \left( \eenergy_{\rr P \left( I_0 \right)} g_{k_2} \right)^{1-\theta_2} \left( \eenergy_{\rr P \left( I_0 \right)} h_{k_3} \right)^{1-\theta_3}
\end{align*}

We will only employ the extra decay in the energy; for the size we have simply
\[
\ssize_{\rr P \left( I_0\right)} f_{k_1}  \lesssim \ssize_{\rr P \left( I_0\right)} f, 
\]
uniformly in $k_1$. 

On the other hand, since $f_{k_1}$ is supported on $\ic I_{k_1}$, Lemma \ref{lemma:localized_energy} implies that
\[
\eenergy_{\rr P \left( I_0\right)} f_{k_1} \lesssim 2^{-k_1 M} \left\| f_{k_1}  \right\|_2.
\]

Hence we obtain 
\begin{align*}
&\lft \Lambda_{BHT; \rr P \left( I_0 \right)} \left( f, g, h \right)\rg \lesssim \left( \ssize_{\rr P \left( I_0\right)} f \right)^{\theta_1}  \left( \ssize_{\rr P \left( I_0\right)} g \right)^{\theta_2}  \left( \ssize_{\rr P \left( I_0\right)} h  \right)^{\theta_3}\\
&\qquad \cdot \sum_{k_1, k_2, k_3} \left( 2^{-k_1 M} \left\| f_{k_1}  \right\|_2 \right)^{1 -\theta_1} \left( 2^{-k_2 M} \left\| g_{k_2}  \right\|_2 \right)^{1 -\theta_2} \left( 2^{-k_1 M} \left\| h_{k_3}  \right\|_2 \right)^{1 -\theta_3}.
\end{align*}

The expressions in the last line are summable, via H\"older's inequality; more exactly, since $\theta_j <1$, 
\begin{align*}
&\sum_{k_1 \geq 0} 2^{-k_1 M \frac{1-\theta_1}{2}}  \left(  2^{-k_1 \frac{M}{2 \left( 1-\theta_1\right)}}\left\| f_{k_1}  \right\|_2\right)^{1-\theta_1}  \\
&\lesssim \left( \sum_{k_1} 2^{-k_1 M \frac{1-\theta_1}{1+\theta_1}} \right)^{\frac{1+\theta_1}{2}} \cdot \left( \sum_{k_1} 2^{-k_1 \frac{M}{1-\theta_1}} \left\| f_{k_1} \right\|_2^2 \right)^{\frac{1-\theta_1}{2}} \\
&\lesssim \left\| f \cdot \ci_I  \right\|_1^{1-\theta_1},
\end{align*}
for $M$ sufficiently large. We note that the implicit constants will depend on $\theta_1$ only. This proves inequality \eqref{eq:refined-trilinear-est}.
\end{proof}
\end{lemma}

\subsection{The Helicoidal Method}~\\
\label{subsec:method_of_the_proof}
With the intention of bringing to light the ideas behind our proofs, we present the main strategy in a simplified setting. Unfortunately, we cannot avoid the specific terminology, but one should think of the \emph{sizes} as being averages, while the \emph{energies} are $L^2$ quantities that reflect orthogonality. For estimating the norms $\ds \| BHT(f,g) \|_s$, we use interpolation results for the trilinear form $\ds \Lambda_{BHT}(f,g,h)=\langle BHT(f,g), h \rangle$. In what follows, $\ds \Lambda_{I_0}(f, g, h)$ denotes a space localization of $\ds \Lambda_{BHT}(f, g, h)$ to the fixed interval $I_0$. More specifically, it is the form associated to a model operator of $BHT$ as in \eqref{eq: BHT model operator}, where the spatial intervals of the tiles lie inside the fixed dyadic interval $I_0$. Similarly, $\Lambda_{I_0}^n(f, g, h)$ denotes a space localization of the corresponding trilinear form in the multiple vector-valued setting.

The helicoidal method is an iterated induction procedure suitable for proving vector-valued estimates for linear and multilinear-operators. We describe the main ideas in the case of the $BHT$ operator, and later on we will indicate the equivalent statements for paraproducts and Carleson operator. At the heart of our argument lies the following induction statement:

\begin{main*}
\label{thm: main BHT}
Let $n \geq 0$. We fix $I_0$ a dyadic interval, and $F, G, H'$ subsets of $\rr{R}$ of finite measure. Let $R_1=\left( r_1^1, \ldots, r_1^n  \right), R_2=\left( r_2^1, \ldots, r_2^n  \right)$ and $R'=\left( (r')^1, \ldots, (r')^n  \right)$ be $n-$tuples so that $\ds \frac{1}{R_1}+\frac{1}{R_2}+\frac{1}{R'}=1$, while $f, g$ and $h$ are vector-valued functions satisfying 
\[
 \|f(x)\|_{L^{R_1}(\ii{W}, \mu)} \leq \one_F(x), \quad  \|g(x)\|_{L^{R_2}(\ii{W}, \mu)} \leq \one_G(x) \quad \text{ and     }  \|h(x)\|_{L^{R'}(\ii{W}, \mu)} \leq \one_{H'}(x).
 \]
 Then we have the following estimate for the trilinear form $\ds \Lambda_{I_0}^n$:
 \begin{equation}
\tag*{$\ii{P}(n) $}
\label{main statement}
\Big \vert \Lambda_{I_0}^n (f, g, h)  \Big \vert \lesssim \left( \sssize_{I_0} \one_F  \right)^{\frac{1}{2}+\frac{\theta_1}{2} -\epsilon}\cdot \left( \sssize_{I_0} \one_G  \right)^{\frac{1}{2}+\frac{\theta_2}{2}-\epsilon} \cdot \left( \sssize_{I_0} \one_{H'}  \right)^{\frac{1}{2}+\frac{\theta_3}{2} -\epsilon} \cdot |I_0|, 
 \end{equation}
for every $0 \leq \theta_1, \theta_2, \theta_3< 1$, $\theta_1+\theta_2+\theta_3=1$, satisfying an extra condition $C(R_1, R_2, R')$. 
\end{main*}

In the local $L^2$ case the condition $C(R_1, R_2, R')$ is satisfied automatically: that is, the $\ii{P}(n)$ statement is true for all $0 \leq \theta_1, \theta_2, \theta_3$ as above. This condition is the main obstruction in obtaining for $\overrightarrow{BHT}_{\vec{r}}$ the same range of $L^p$ estimates as that of the scalar $BHT$; in \eqref{eq: constraint origin} we point out the source of this constraint. Now we present the proofs of the induction statements $\ii{P}(0)$ and $\ii{P}(n) \Rightarrow \ii{P}(n+1)$. Also, for the reader's convenience, we include the $\ii{P}(0) \Rightarrow \ii{P}(1)$ step.

As we will see later on, the fact that $\ii{P}(n)$ implies our Theorems \ref{vector valued BHT} and \ref{multiple vector valued BHT} is based on a standard triple stopping time argument, involving the above localized sizes.
 
\subsection*{Check $\ii{P}(0)$:} ~\\
This is the scalar $BHT$ case, with $\vert f \vert \leq \one_F, \vert g \vert \leq \one_G$ and $\vert h \vert \leq \one_{H'}$. This situation is well-understood, and we have from Proposition \ref{BHT trilinear estimate}:
\begin{align*}
\vert \Lambda_{I_0}(f, g, h)  \vert & \lesssim  \left( \sssize_{I_0} f  \right)^{\theta_1} \cdot \left( \sssize_{I_0} g  \right)^{\theta_2} \cdot \left( \sssize_{I_0} h  \right)^{\theta_3} \\
&\cdot \left( \eenergy_{I_0} f \right)^{1-\theta_1} \left( \eenergy_{I_0} g \right)^{1-\theta_2} \left( \eenergy_{I_0} h \right)^{1-\theta_3},
\end{align*}
for any $0 \leq \theta_1, \theta_2, \theta_3 <1$ such that $\theta_1+\theta_2+\theta_3=1$.

Since we are considering a localized model of $BHT$, where all the tiles have their spatial intervals $I_P \subseteq I_0$, one can refine Lemma \ref{BHT size estimate} and obtain
\[
\eenergy_{I_0} f \lesssim \|  f \cdot \ci_{I_0}  \|_{2} \lesssim \left( \sssize_{I_0} \one_F   \right)^{\frac{1}{2}} \cdot |I_0|^{\frac{1}{2}}.
\]
Noticing that $\ds |I_0|^{\frac{1-\theta_1}{2}} \cdot |I_0|^{\frac{1-\theta_3}{2}} \cdot |I_0|^{\frac{1-\theta_3}{2}}=|I_0|$, we obtain the desired $\ii{P}(0)$.

\subsection*{Check $\ii{P}(0) \Rightarrow \ii{P}(1)$} ~\\

Assume that 
\begin{equation}
 \label{eq: vv assumptions}
 \left( \sum_k \vert f_k  \vert^{r_1}  \right)^{1/{r_1}} \leq \one_F, \left( \sum_k \vert g_k  \vert^{r_2}  \right)^{1/{r_2}} \leq \one_G \quad \text{ and  } \left( \sum_k \vert h_k  \vert^{r'}  \right)^{1/{r'}} \leq \one_{H'}. 
\end{equation}
Given that we know $\ii{P}(0)$, we will prove
\begin{align}
\tag*{$\ii{P}(1)$}
\big \vert \sum_k \Lambda_{I_0}(f_k, g_k, h_k)   \big \vert \lesssim  \left( \sssize_{I_0} \one_F  \right)^{\frac{1}{2}+\frac{\theta_1}{2}-\epsilon}\cdot \left( \sssize_{I_0} \one_G  \right)^{\frac{1}{2}+\frac{\theta_2}{2}-\epsilon} \cdot \left( \sssize_{I_0} \one_{H'}  \right)^{\frac{1}{2}+\frac{\theta_3}{2}-\epsilon} \cdot |I_0|, 
\end{align}
for any $0 \leq \theta_1, \theta_2, \theta_3< 1$, $\theta_1+\theta_2+\theta_3=1$, satisfying 
\begin{equation}
\label{eq:condition}
\tag*{$C(r_1, r_2, r')$}
\frac{1+\theta_1}{2} -\frac{1}{r_1} >0, \quad \frac{1+\theta_2}{2} -\frac{1}{r_2} >0, \quad \frac{1+\theta_3}{2} -\frac{1}{r'} >0.
\end{equation}
Here an intermediate step is necessary in order to get a finer estimate for each $\Lambda_{I_0}(f_k, g_k, h_k)$. That is, we need to prove
\begin{equation}
\label{eq: star}
\Lambda_{I_0}(f_k \cdot \one_F, g_k \cdot \one_G, h_k \cdot \one_{H'} ) \lesssim \| \Lambda_{I_0} \| \cdot \|f_k \cdot \ci_{I_0} \|_{r_1} \|g_k \cdot \ci_{I_0} \|_{r_2} \|h_k \cdot \ci_{I_0} \|_{r'},
\end{equation}
where the operatorial norm is given by 
\[\ds \| \Lambda_{I_0} \|= \left(\sssize_{I_0} \one_F  \right)^{\frac{1+\theta_1}{2}-\frac{1}{r_1}-\epsilon} \left(\sssize_{I_0} \one_G  \right)^{\frac{1+\theta_2}{2}-\frac{1}{r_2}-\epsilon} \left(\sssize_{I_0} \one_{H'}  \right)^{\frac{1+\theta_3}{2}-\frac{1}{r'}-\epsilon}.
\]
Once we have this, H\"{o}lder's inequality and \eqref{eq: vv assumptions} allows one to further estimate \eqref{eq: star} by
\[
\| \Lambda_{I_0} \| \cdot \frac{\| \one_F \cdot \ci_{I_0}  \|_{r_1}}{|I_0|^{1/{r_1}}} \cdot \frac{\| \one_G \cdot \ci_{I_0}  \|_{r_2}}{|I_0|^{1/{r_2}}} \cdot \frac{\| \one_{H'} \cdot \ci_{I_0}  \|_{r'}}{|I_0|^{1/{r'}}} \cdot |I_0|,
\]
which proves $\ii{P}(1)$.

\psscalebox{.55 .55} 
{
\begin{pspicture}(1.5,-5.6379685)(56.34,6)
\psframe[linecolor=black, linewidth=0.04, dimen=outer](8.8,5.237969)(0.4,4.0379686)
\rput[bl](0.8,4.4379687){ \Large $\Lambda_{BHT; \mathbb{P}(I_0)}(f_N \cdot \mathbf{1}_F, g_N \cdot \mathbf{1}_G, h_N \cdot \mathbf{1}_{H'})$}
\psdots[linecolor=black, dotsize=0.14](4.4,2.8379688)
\psdots[linecolor=black, dotsize=0.14](4.4,2.4379687)
\psdots[linecolor=black, dotsize=0.14](4.4,2.0379686)
\psdots[linecolor=black, dotsize=0.14](4.4,0.43796876)
\psdots[linecolor=black, dotsize=0.14](4.4,0.83796877)
\psframe[linecolor=black, linewidth=0.04, dimen=outer](8.8,-2.3620312)(0.4,-3.5620313)
\rput[bl](0.8,-3.1620312){ \Large $\Lambda_{BHT; \mathbb{P}(I_0)}(f_1 \cdot \mathbf{1}_F, g_1 \cdot \mathbf{1}_G, h_1 \cdot \mathbf{1}_{H'})$}
\psframe[linecolor=black, linewidth=0.04, dimen=outer](8.8,-0.36203125)(0.4,-1.5620313)
\rput[bl](0.8,-1.1620313){ \Large $\Lambda_{BHT; \mathbb{P}(I_0)}(f_2 \cdot \mathbf{1}_F, g_2 \cdot \mathbf{1}_G, h_2 \cdot \mathbf{1}_{H'})$}
\psline[linecolor=black, linewidth=0.04, arrowsize=0.05291666666666668cm 2.0,arrowlength=1.4,arrowinset=0.0]{->}(4.4,-2.3620312)(4.4,-1.5620313)
\psline[linecolor=black, linewidth=0.04, arrowsize=0.05291666666666668cm 2.0,arrowlength=1.4,arrowinset=0.0]{->}(4.4,-0.36203125)(4.4,0.03796875)
\psline[linecolor=black, linewidth=0.04, tbarsize=0.07055555555555557cm 5.0]{|*-|*}(8.8,-4.762031)(0.4,-4.762031)(0.4,-4.762031)
\rput[bl](3.2,-5.5620313){\Large $I_0$}
\psline[linecolor=black, linewidth=0.04, arrowsize=0.05291666666666668cm 2.0,arrowlength=1.4,arrowinset=0.0]{->}(4.4,3.2379687)(4.4,4.0379686)
\psframe[linecolor=black, linewidth=0.02, dimen=outer](9.2,5.6379685)(0.0,-4.3620315)
\psline[linecolor=black, linewidth=0.1, arrowsize=0.05291666666666668cm 2.0,arrowlength=1.4,arrowinset=0.0]{->}(9.2,1.2379688)(10.4,1.2379688)
\rput[bl](12.4,-1.9620312){$ $}
\rput[bl](10.8,1.5379688){\huge $\left( \text{ size}_{I_0} \mathbf{1}_F   \right)^{\frac{1+\theta_1}{2}-\frac{1}{r_1}} \cdot \left( \text{ size}_{I_0} \mathbf{1}_G \right)^{\frac{1+\theta_2}{2}-\frac{1}{r_2}} \cdot \left( \text{ size}_{I_0} \mathbf{1}_{H'} \right)^{\frac{1+\theta_3}{2}-\frac{1}{r'}}$}
\rput[bl](12.4,-1.9620312){$ $}
\rput[bl](12.4,-1.9620312){$ $}
\rput[bl](11.2,0.43796876){\huge $\cdot \| \mathbf{1}_F \cdot \tilde{\chi}_{I_0}   \|_{r_1} \| \mathbf{1}_G \cdot \tilde{\chi}_{I_0}   \|_{r_2} \| \mathbf{1}_{H'} \cdot \tilde{\chi}_{I_0}   \|_{r'}$}
\end{pspicture}
}

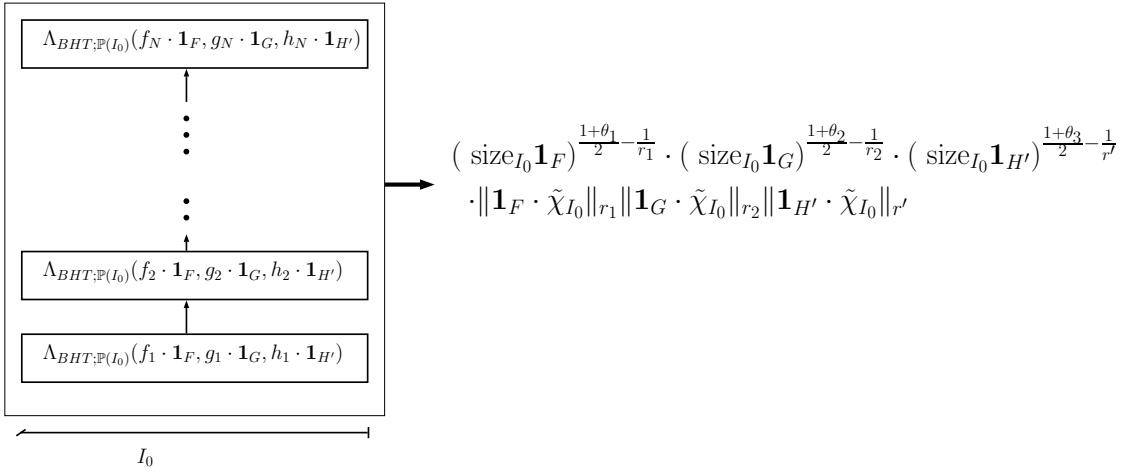
\captionof{figure}{Output of the localization process}

The proof of \eqref{eq: star} is a slight modification of the proof of the boundedness of the bilinear Hilbert transform. Using interpolation methods, we can assume that $\vert f_k \vert \leq \one_{E_1}, \vert g_k \vert \leq \one_{E_2}, \vert h_k \vert \leq \one_{E_3}$. So we need to show 
\[
\Lambda_{I_0}(f_k \cdot \one_F, g_k \cdot \one_G, h_k \cdot \one_{H'} ) \lesssim \| \Lambda_{I_0} \| \cdot |E_1|^{\alpha_1} \cdot |E_2|^{\alpha_2} \cdot |E_3|^{\alpha_3},
\]
where $\left(\alpha_1,\alpha_2, \alpha_3 \right)$ is an admissible tuple arbitrarily close to $\ds \left( \frac{1}{r_1}, \frac{1}{r_2}, \frac{1}{r'} \right)$.
In order to get the desired expression for $\ds \| \Lambda_{I_0} \|$, we need another stopping time inside $I_0$. This is illustrated in Figure \ref{fig: stopping time I_0}.

\psscalebox{.55 .55} 
{
\begin{pspicture}(2,-3.2464285)(44.756523,11)
\rput[bl](19.256521,-10.2464285){$ $}
\rput[bl](38.63652,0.91785717){$ $}
\rput[bl](38.63652,1.7178571){$ $}
\psdots[linecolor=black, dotsize=0.0](17.68,6.6464286)
\psdots[linecolor=black, dotsize=0.0](17.68,6.2464285)
\psdots[linecolor=black, dotsize=0.0](17.68,5.8464284)
\psdots[linecolor=black, dotsize=0.0](17.68,4.2464285)
\psdots[linecolor=black, dotsize=0.0](17.68,4.6464286)
\psframe[linecolor=black, linewidth=0.02, dimen=outer](10.88,9.446428)(1.68,-0.5535714)
\rput[bl](4.88,-1.7535714){ $I$}
\psline[linecolor=black, linewidth=0.04, tbarsize=0.07055555555555557cm 5.0]{|-|}(10.48,-0.95357144)(2.08,-0.95357144)(2.08,-0.95357144)
\psline[linecolor=black, linewidth=0.04, arrowsize=0.05291666666666668cm 2.0,arrowlength=1.4,arrowinset=0.0]{->}(6.08,1.4464285)(6.08,2.2464285)
\rput[bl](2.48,0.6464286){  $\Lambda_{I}(f_1 \cdot \mathbf{1}_F, g_1 \cdot \mathbf{1}_G, h_1 \cdot \mathbf{1}_{H'})$}
\psframe[linecolor=black, linewidth=0.04, dimen=outer](10.48,1.4464285)(2.08,0.24642856)
\psline[linecolor=black, linewidth=0.04, arrowsize=0.05291666666666668cm 2.0,arrowlength=1.4,arrowinset=0.0]{->}(6.08,7.0464287)(6.08,7.8464284)
\psline[linecolor=black, linewidth=0.04, arrowsize=0.05291666666666668cm 2.0,arrowlength=1.4,arrowinset=0.0]{->}(6.08,3.4464285)(6.08,3.8464286)
\rput[bl](2.48,2.6464286){  $\Lambda_{I}(f_2 \cdot \mathbf{1}_F, g_2 \cdot \mathbf{1}_G, h_2 \cdot \mathbf{1}_{H'})$}
\psframe[linecolor=black, linewidth=0.04, dimen=outer](10.48,3.4464285)(2.08,2.2464285)
\psdots[linecolor=black, dotsize=0.14](6.08,4.6464286)
\psdots[linecolor=black, dotsize=0.14](6.08,4.2464285)
\psdots[linecolor=black, dotsize=0.14](6.08,5.8464284)
\psdots[linecolor=black, dotsize=0.14](6.08,6.2464285)
\psdots[linecolor=black, dotsize=0.14](6.08,6.6464286)
\rput[bl](2.48,8.2464285){  $\Lambda_{I}(f_N \cdot \mathbf{1}_F, g_N \cdot \mathbf{1}_G, h_N \cdot \mathbf{1}_{H'})$}
\psframe[linecolor=black, linewidth=0.04, dimen=outer](10.48,9.046429)(2.08,7.8464284)
\psdots[linecolor=black, dotsize=0.14](18.9,4.1264286)
\psdots[linecolor=black, dotsize=0.14](19.98,4.1264286)
\psdots[linecolor=black, dotsize=0.14](19.26,4.1264286)
\psdots[linecolor=black, dotsize=0.14](19.62,4.1264286)
\psframe[linecolor=black, linewidth=0.04, dimen=outer](29.52,10.2464285)(0.0,-2.7135713)
\rput[bl](14.76,-3.7935715){ $I_0$}
\psframe[linecolor=black, linewidth=0.04, dimen=outer](17.460869,9.046429)(12.019131,7.8464284)
\rput[bl](12.098261,8.2464285){  $\Lambda_{I'}(f_N \cdot \mathbf{1}_F, g_N \cdot \mathbf{1}_G, h_N \cdot \mathbf{1}_{H'})$}
\psframe[linecolor=black, linewidth=0.04, dimen=outer](17.460869,3.4464285)(12.019131,2.2464285)
\rput[bl](12.278261,2.6464286){ $\Lambda_{I'}(f_2 \cdot \mathbf{1}_F, g_2 \cdot \mathbf{1}_G, h_2 \cdot \mathbf{1}_{H'})$}
\psframe[linecolor=black, linewidth=0.04, dimen=outer](17.460869,1.4464285)(12.019131,0.24642856)
\rput[bl](12.278261,0.6464286){ $\Lambda_{I'}(f_1 \cdot \mathbf{1}_F, g_1 \cdot \mathbf{1}_G, h_1 \cdot \mathbf{1}_{H'})$}
\psline[linecolor=black, linewidth=0.04, tbarsize=0.07055555555555557cm 5.0]{|-|}(17.460869,-0.95357144)(12.019131,-0.95357144)(12.019131,-0.95357144)
\rput[bl](13.833043,-1.7535714){$I'$}
\psframe[linecolor=black, linewidth=0.02, dimen=outer](17.72,9.446428)(11.76,-0.5535714)
\psdots[linecolor=black, dotsize=0.14](14.519131,4.2812114)
\psdots[linecolor=black, dotsize=0.14](14.519131,4.5942545)
\psdots[linecolor=black, dotsize=0.14](14.519131,6.6290374)
\psdots[linecolor=black, dotsize=0.14](14.519131,6.315994)
\psdots[linecolor=black, dotsize=0.14](14.519131,5.8464284)
\psline[linecolor=black, linewidth=0.04, arrowsize=0.05291666666666668cm 2.0,arrowlength=1.4,arrowinset=0.0]{->}(14.519131,6.942081)(14.519131,7.8812113)
\psline[linecolor=black, linewidth=0.04, arrowsize=0.05291666666666668cm 2.0,arrowlength=1.4,arrowinset=0.0]{->}(14.519131,3.4986024)(14.519131,4.1246896)
\psline[linecolor=black, linewidth=0.04, arrowsize=0.05291666666666668cm 2.0,arrowlength=1.4,arrowinset=0.0]{->}(14.519131,1.4638199)(14.519131,2.2464285)
\psframe[linecolor=black, linewidth=0.04, dimen=outer](27.325827,9.007857)(21.148205,7.812143)
\rput[bl](21.442379,8.210714){  $\Lambda_{I"}(f_N \cdot \mathbf{1}_F, g_N \cdot \mathbf{1}_G, h_N \cdot \mathbf{1}_{H'})$}
\psdots[linecolor=black, dotsize=0.14](24.24,7.0064287)
\psdots[linecolor=black, dotsize=0.14](24.24,6.6064286)
\psdots[linecolor=black, dotsize=0.14](24.24,6.2064285)
\psdots[linecolor=black, dotsize=0.14](24.24,4.426429)
\psdots[linecolor=black, dotsize=0.14](24.24,4.8264284)
\psframe[linecolor=black, linewidth=0.04, dimen=outer](27.325827,3.4278572)(21.148205,2.232143)
\rput[bl](21.442379,2.6307142){ $\Lambda_{I''}(f_2 \cdot \mathbf{1}_F, g_2 \cdot \mathbf{1}_G, h_2 \cdot \mathbf{1}_{H'})$}
\psframe[linecolor=black, linewidth=0.02, dimen=outer](27.62,9.406428)(20.854033,-0.55785716)
\rput[bl](23.207413,-1.7535714){ $I''$}
\psline[linecolor=black, linewidth=0.04, tbarsize=0.07055555555555557cm 5.0]{|-|}(27.325827,-0.9564286)(21.148205,-0.9564286)(21.148205,-0.9564286)
\rput[bl](21.442379,0.63785714){  $\Lambda_{I''}(f_1 \cdot \mathbf{1}_F, g_1 \cdot \mathbf{1}_G, h_1 \cdot \mathbf{1}_{H'})$}
\psframe[linecolor=black, linewidth=0.04, dimen=outer](27.325827,1.435)(21.148205,0.23928571)
\psline[linecolor=black, linewidth=0.04, arrowsize=0.05291666666666668cm 2.0,arrowlength=1.4,arrowinset=0.0]{->}(24.25816,1.387762)(24.25816,2.1317618)
\psline[linecolor=black, linewidth=0.04, arrowsize=0.05291666666666668cm 2.0,arrowlength=1.4,arrowinset=0.0]{->}(24.25816,3.4337618)(24.25816,4.363762)
\psline[linecolor=black, linewidth=0.04, arrowsize=0.05291666666666668cm 2.0,arrowlength=1.4,arrowinset=0.0]{->}(24.25816,7.153762)(24.25816,7.711762)
\end{pspicture}
}

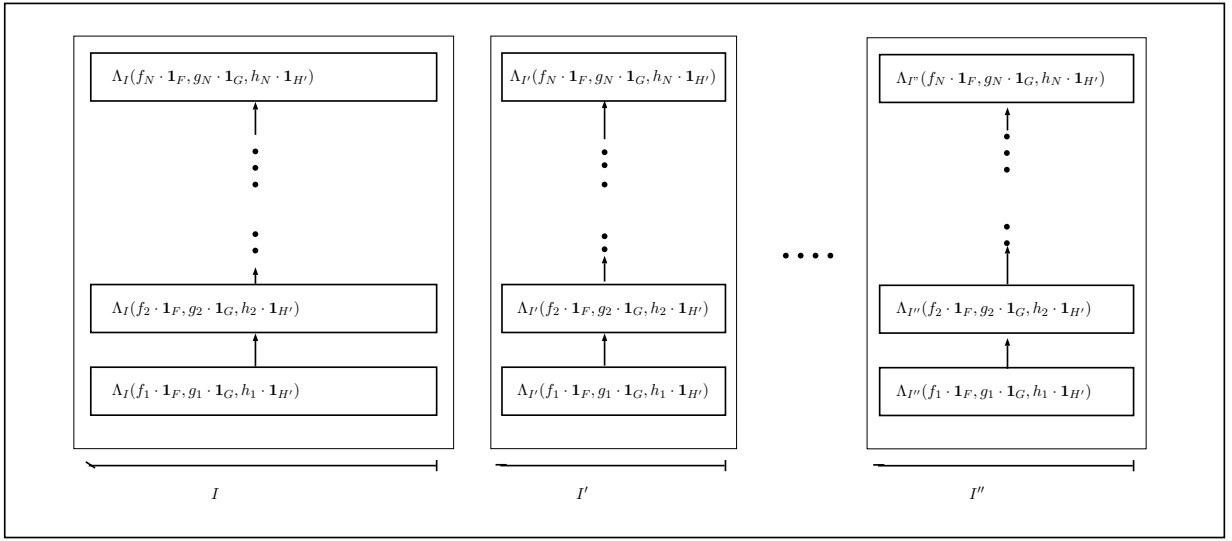
\captionof{figure}{Extra stopping time}    \label{fig: stopping time I_0}

Let $I \subseteq I_0$ be a subinterval of $I_0$. Now we use $\ii{P}(0)$ as follows:
\begin{align*}
& \vert \Lambda_{I}\left( f_k \cdot \one_F, g_k \cdot \one_G, h_k \cdot \one_{H'} \right) \vert  \\
&\lesssim  \left(\sssize_I \one_{F} \cdot \one_{E_1}   \right)^{\frac{1+\theta_1}{2} -\epsilon} \cdot \left(\sssize_I \one_{G} \cdot \one_{E_2}   \right)^{\frac{1+\theta_2}{2} -\epsilon} \cdot \left(\sssize_I \one_{H'} \cdot \one_{E_3}   \right)^{\frac{1+\theta_3}{2}-\epsilon} \cdot |I|  \\
&\lesssim \left(\sssize_{I_0} \one_{F} \right)^{\frac{1+\theta_1}{2}-\alpha_1 -\epsilon} \cdot \left(\sssize_{I_0} \one_{G} \right)^{\frac{1+\theta_2}{2}-\alpha_2 -\epsilon} \cdot \left(\sssize_{I_0} \one_{H'} \right)^{\frac{1+\theta_3}{2}-\alpha_3-\epsilon} \\
&\quad \cdot \left(\sssize_I \one_{E_1}   \right)^{\alpha_1} \cdot \left(\sssize_I \one_{E_2} \right)^{\alpha_2} \cdot \left(\sssize_I \one_{E_3}   \right)^{\alpha_3} \cdot |I|.
\end{align*}
In order to obtain the last inequality, we have to make sure that the exponents
\[
\frac{1+\theta_1}{2}-\alpha_1 -\epsilon, \quad \frac{1+\theta_2}{2}-\alpha_2 -\epsilon, \quad \frac{1+\theta_3}{2}-\alpha_3 -\epsilon
\]
are all positive, which is always the case in the local $L^2$ situation. Since $(\alpha_1, \alpha_2, \alpha_3)$ are arbitrarily close to $\left(\frac{1}{r_1}, \frac{1}{r_2}, \frac{1}{r'} \right)$, this is the origin of the constraint $C(r_1, r_2, r')$.

Summing over the intervals $I$ given by the alluded triple stopping time over the corresponding averages, we recover $\ds |E_1|^{\alpha_1} \cdot |E_2|^{\alpha_2} \cdot |E_3|^{\alpha_3}$. We note that the operatorial norm given by interpolation is
\[
\left(\sssize_{I_0} \one_F  \right)^{\frac{1+\theta_1}{2}-\frac{1}{r_1}-\tilde{\epsilon}} \left(\sssize_{I_0} \one_G  \right)^{\frac{1+\theta_2}{2}-\frac{1}{r_2}-\tilde{\epsilon}} \left(\sssize_{I_0} \one_{H'}  \right)^{\frac{1+\theta_3}{2}-\frac{1}{r'}-\tilde{\epsilon}},
\]
where $\tilde{\epsilon}$ is slightly larger than the initial $\epsilon$, but the difference between the two is irrelevant.

\subsection*{Check $\ii{P}(n) \Rightarrow \ii{P}(n+1)$}~\\
Lastly, we present the general induction step, in the case of iterated $\ell^p$ spaces. We have multi-indices $\ds \vec{r}_1=\left(r_1^1, \ldots r_1^n\right), \vec{r}_2=\left(r_2^1, \ldots r_2^n\right), \vec{r'}=\left(\right( r'\left)^1, \ldots \right( r'\left)^n\right)$, and $\ds \| f\|_{\vec{r}_1} \leq \one_F,  \| g\|_{\vec{r}_2} \leq \one_G,  \| h\|_{\vec{r'}} \leq \one_{H'}$. Then $\ds \ii{P}(n)$ is equivalent to
\begin{align}
\label{eq: p_n}
&\big \vert  \Lambda_{I_0}^n(f, g, h)\big \vert=\big \vert \int_{\rr{R}} \sum_{\vec{l}} BHT_{\rr{P}(I_0)}(f_{\vec{l}}, g_{\vec{l}})(x) h_{\vec{l}}(x) dx  \big \vert  \\
& \lesssim\left( \sssize_{I_0} \one_F  \right)^{\frac{1}{2}+\frac{\theta_1}{2}-\epsilon}\cdot \left( \sssize_{I_0} \one_G  \right)^{\frac{1}{2}+\frac{\theta_2}{2}-\epsilon} \cdot \left( \sssize_{I_0} \one_{H'}  \right)^{\frac{1}{2}+\frac{\theta_3}{2}-\epsilon} \cdot |I_0|, \nonumber
\end{align}
whenever $I_0$ is a dyadic interval.
For $\ii{P}(n+1)$ we consider $n+1$ iterated $\ell^p$ spaces, given by the multi-indices: $\vec{R}_1=\left( r_1, \vec{r}_1 \right), \vec{R}_2=\left( r_2, \vec{r}_2 \right)$ and $\vec{R'}=\left( r', \vec{r'} \right)$, while $f, g$ and $h$ are vector-valued functions satisfying
{
\fontsize{10}{11}
\begin{equation}
\label{eq:  multi spaces}
\|f\|_{\vec{R}_1}:=\left( \sum_{k} \| f_k\|_{\vec{r}_1}^{r_1}  \right)^{1/{r_1}} \leq \one_F, \|g\|_{\vec{R}_2}:=\left( \sum_{k} \| g_k\|_{\vec{r}_2}^{r_2}  \right)^{1/{r_2}} \leq \one_G, \|h\|_{\vec{R'}}:=\left( \sum_{k} \| h_k\|_{\vec{r'}}^{r'}  \right)^{1/{r'}} \leq \one_{H'}.
\end{equation}
}
We want a result similar to \eqref{eq: p_n}, so we need to estimate
\[
\Lambda^{n+1}_{I_0}(f, g, h):=\int_{\rr{R}} \sum_k \sum_{\vec{l}} BHT_{\rr{P}(I_0)}(f_{k, \vec{l}}, g_{k, \vec{l}})(x) h_{k, \vec{l}}(x) dx= \sum_k \Lambda^{n}_{I_0}(f_k, g_k, h_k).
\]
We can't directly apply $\ds \ii{P}(n)$, and instead we will need the following result, similar to \eqref{eq: star}:
\begin{equation}
\label{eq: int P_n}
\big \vert  \Lambda^{n}_{I_0}(f_k, g_k, h_k)   \big \vert \lesssim  \big \| \Lambda_{I_0}^n\big \| \cdot \|f_k \cdot \ci_{I_0} \|_{r_1} \|g_k \cdot \ci_{I_0} \|_{r_2} \|h_k \cdot \ci_{I_0} \|_{r'},
\end{equation}
where $\ds \| \Lambda_{I_0}^n \|= \left(\sssize_{I_0} \one_F  \right)^{\frac{1+\theta_1}{2}-\frac{1}{r_1}-\epsilon} \left(\sssize_{I_0} \one_G  \right)^{\frac{1+\theta_2}{2}-\frac{1}{r_2}-\epsilon} \left(\sssize_{I_0} \one_{H'}  \right)^{\frac{1+\theta_3}{2}-\frac{1}{r'}-\epsilon}$. Once we have such a result, $\ds \ii{P}(n+1)$ follows easily by H{\"o}lder, exactly as before.

We will prove \eqref{eq: int P_n} by using restricted type interpolation. Instead of estimating the trilinear form $\Lambda_{I_0}^n$, we will deal with 
\begin{equation}
\label{eq: localized form}
\Lambda_{I_0}^{n, F, G, H'}(f_k, g_k, h_k):=\Lambda_{I_0}(f_k \cdot \one_F, g_k \cdot \one_G, h_k \cdot \one_{H'}).
\end{equation} 
This is natural since condition \eqref{eq:  multi spaces} implies that the functions $f_k$ are supported on $F$, and similarly the functions $g_k$ are supported on $G$ and $h_k$ on $H'$. By interpolation theory, we can assume that 
\[
\|f_k\|_{\vec{r}_1} \leq \one_{E_1}, \|g_k\|_{\vec{r}_2} \leq \one_{E_2}, \text{  and  } \|h_k\|_{\vec{r'}} \leq \one_{E_3},
\]
and it suffices to prove
\begin{equation}
\label{eq: est Lambda_n, k, F, G, H}
\big \vert \Lambda^{n, F, G, H'}_{I_0}(f_k, g_k, h_k) \big \vert \lesssim \| \Lambda_{I_0}^n  \| \cdot |E_1|^{\alpha_1} \cdot |E_2|^{\alpha_2} \cdot |E_3|^{\alpha_3},
\end{equation}
for $\left(\alpha_1, \alpha_2, \alpha_3 \right)$ in a small neighborhood of $\left(\frac{1}{r_1}, \frac{1}{r_2}, \frac{1}{r'} \right)$.
Similarly to the case $\ds \ii{P}(0) \Rightarrow \ii{P}(1)$, we will have a stopping time inside $I_0$, so in fact we need to estimate $\Lambda^{n, F, G, H'}_{I}(f_k, g_k, h_k)$ for some $I \subseteq I_0$. It is here that we use hypothesis $\ds \ii{P}(n)$:
\begin{align*}
\big \vert \Lambda_I^{n, F, G, H'}(f_k, g_k, h_k)  \big \vert=\big \vert  \Lambda_I^n(f_k \cdot \one_F, g_k \cdot \one_G, h_k \cdot \one_{H'} )   \big \vert,
\end{align*}
with $\ds  \|  f_k \cdot \one_F \|_{\vec{r}_1} \leq \one_{F \cap E_1}, \|  g_k \cdot \one_G \|_{\vec{r}_2} \leq \one_{G \cap E_2}$ and $\ds \|  h_k \cdot \one_{H'} \|_{\vec{r'}} \leq \one_{H' \cap E_3}$. More precisely, 
\begin{align*}
&\big \vert \Lambda_I^{n, F, G, H'}(f_k, g_k, h_k)  \big \vert  \\
&\lesssim \left( \sssize_{I} \one_F \cdot \one_{E_1} \right)^{\frac{1}{2}+\frac{\theta_1}{2} -\epsilon}\cdot \left( \sssize_{I} \one_G \cdot \one_{E_2} \right)^{\frac{1}{2}+\frac{\theta_2}{2}-\epsilon} \cdot \left( \sssize_{I} \one_{H'} \cdot \one_{E_3} \right)^{\frac{1}{2}+\frac{\theta_3}{2} -\epsilon} \cdot |I|  \\
& \lesssim \left( \sssize_{I_0} \one_F  \right)^{\frac{1}{2}+\frac{\theta_1}{2}-\alpha_1-\epsilon}\cdot \left( \sssize_{I_0} \one_G  \right)^{\frac{1}{2}+\frac{\theta_2}{2} -\alpha_2 -\epsilon} \cdot \left( \sssize_{I_0} \one_{H'}  \right)^{\frac{1}{2}+\frac{\theta_3}{2} -\alpha_3 -\epsilon} \\
&\quad \cdot \left( \sssize_{I} \one_{E_1} \right)^{\alpha_1}\cdot \left( \sssize_{I} \one_{E_2} \right)^{\alpha_2} \cdot \left( \sssize_{I} \one_{E_3} \right)^{\alpha_3} \cdot |I|,
\end{align*}
for $\left(\alpha_1, \alpha_2, \alpha_3 \right)$ in a neighborhood of $\left(\frac{1}{r_1}, \frac{1}{r_2},\frac{1}{r'} \right)$. Due to the stopping time which is performed with respect to the three sizes, the expressions $\ds  \left( \sssize_{I} \one_{E_1} \right)^{\alpha_1}$ add up to $\ds |E_1|^{\alpha_1}$ and similarly for the sizes of $\one_{E_2}$ and $\one_{E_3}$. Interpolating, we get the desired \eqref{eq: est Lambda_n, k, F, G, H}. From the above equation, we can see why the operatorial norm has the form
\[
\big \| \Lambda_{I_0}^n  \big \|=\left(\sssize_{I_0} \one_F  \right)^{\frac{1+\theta_1}{2}-\frac{1}{r_1}-\tilde{\epsilon}} \left(\sssize_{I_0} \one_G  \right)^{\frac{1+\theta_2}{2}-\frac{1}{r_2}-\tilde{\epsilon}} \left(\sssize_{I_0} \one_{H'}  \right)^{\frac{1+\theta_3}{2}-\frac{1}{r'}-\tilde{\epsilon}}.
\]
The $\tilde{\epsilon}$ (which is a slight modification on the $\epsilon$ in the $\ii{P}(n)$ statement),  appears as an interpolation error; moreover, the conditions
\begin{equation}
\label{eq: constraint origin}
\frac{1+\theta_1}{2}-\frac{1}{r_1} >0, \quad \frac{1+\theta_2}{2}-\frac{1}{r_2} >0, \quad \frac{1+\theta_3}{2}-\frac{1}{r'} >0 
\end{equation}
are necessary, and they imply the constraint $C(R_1, R_2, R')$. This ends the proof of the induction step.

The same method applies in the case of paraproducts. The difference here is that the energies are $L^1$ quantities, and for that reason we don't have any extra assumptions; the range of the multiple vector-valued extensions is the same as that of the paraproducts. The model operator for paraproducts $\Pi$ corresponds to a ``rank $0$'' family of tri-tiles; that is, once we know the spatial interval $I_P$, there is no other degree of freedom and the frequency intervals are $\ds \left[1/{\lft I_P \rg}, 2/{\lft I_P \rg} \right]$ or  $\ds \left[0, 1/{\lft I_P \rg} \right]$. The exact definitions will be introduced in Section \ref{paraproduct results}.
\begin{main*}[Paraproducts case]
Under the same assumptions as in \eqref{thm: main BHT}, the localized trilinear form for paraproducts satisfies
 \begin{equation}
\tag*{$\ii{P}(n)$}
\label{main statement paraproducts}
\Big \vert \Lambda_{I_0}^n (f, g, h)  \Big \vert \lesssim \left( \sssize_{I_0} \one_F  \right)^{1-\epsilon}\cdot \left( \sssize_{I_0} \one_G  \right)^{1-\epsilon} \cdot \left( \sssize_{I_0} \one_{H'}  \right)^{1-\epsilon} \cdot |I_0|,
 \end{equation}
 provided
 \[
 \|f(x)\|_{L^{R_1}(\ii{W}, \mu)} \leq \one_F(x), \quad  \|g(x)\|_{L^{R_2}(\ii{W}, \mu)} \leq \one_G(x) \quad \text{ and} \quad  \|h(x)\|_{L^{R'}(\ii{W}, \mu)} \leq \one_{H'}(x).
 \]
\end{main*}

Finally, we want to point out that the helicoidal method applies equally in the case of (sub)linear operators. One last example is that of the Carleson operator 
\[
C_{\rr{R}}f(x)=\sup_N \big \vert \int_{\xi< N} \hat{f}(\xi)   e^{2 \pi i x \xi} d \xi \big \vert,
\]
for which $UMD-$valued extensions are already known from the work of Hyt\"{o}nen and Lacey \cite{Carleson_in_UMD_spaces}. 

In \cite{DemSilvaLight}, Demeter and Silva give an alternative proof for $\ell^2$- valued inequalities for the Carleson operator. In fact, they present a new principle, built around ideas from \cite{BatemanThiele2011}, for dealing with $\ell^2$- valued inequalities for sublinear operators which are not of Calder\'{o}n-Zygmund type.

We do not present all the details here, but the essential statement for proving multiple vector-valued inequalities for the Carleson operator, using the helicoidal method, is the following:
\begin{main*}[Carleson Operator]
Under the same assumptions as in \eqref{thm: main BHT}, the localized bilinear form for the discretized Carleson operator satisfies
 \begin{equation}
\tag*{$\ii{P}(n)$}
\label{main statement Carleson}
\Big \vert \Lambda_{\mathcal{C}\left(I_0\right)}^n (f, g)  \Big \vert \lesssim \left( \sssize_{I_0} \one_F  \right)^{1-\epsilon}\cdot \left( \sssize_{I_0} \one_G  \right)^{1-\epsilon} \cdot |I_0|,
 \end{equation}
 provided that
 \[
 \|f(x)\|_{L^{R_1}(\ii{W}, \mu)} \leq \one_F(x), \text{    and } \quad  \|g(x)\|_{L^{R_2}(\ii{W}, \mu)} \leq \one_G(x) .
 \]
\end{main*}

Comparing the main statements of the above three examples, we can see from the exponents of the sizes that the range of $L^p$ estimates for the vector-valued Carleson operator and for the vector-valued paraproduct $\Pi$ will coincide with the range of the scalar operator. However, for $BHT$ things are more complicated.

\section{Multiple vector-valued estimates for $BHT$}
\label{sec: multiple vector estimates for BHT}

In this section we describe the detailed proof of our Theorems \ref{vector valued BHT} and \ref{multiple vector valued BHT}.
\subsection{Estimates for Localized $BHT$}
\label{Estimates for Localized $BHT$}~\\

Here we assume that $F, G$ and $H'$ are fixed subsets of $\rr{R}$ of finite measure and $I_0$ is a fixed dyadic interval. We are interested in finding estimates for the bilinear operator
\[
BHT^{F, G, H'}_{I_0}(f,g)(x):=\sum_{P \in \rr{P}(I_0)} \frac{1 }{|I_P|^{1/2}} \langle f \cdot \one_F , \phi_{P_1}^1  \rangle \langle g \cdot \one_G , \phi_{P_2}^2 \rangle  \phi_{P_3}^3(x) \one_{H'}(x).
\] 

In doing so, we first study the associated trilinear form:
\[
\Lambda_{BHT; \rr{P}(I_0)}^{F, G, H'}(f, g, h):=\sum_{P \in \rr{P}(I_0)} \frac{1 }{|I_P|^{1/2}} \langle f \cdot \one_F , \phi_{P_1}^1  \rangle \langle g \cdot \one_G , \phi_{P_2}^2 \rangle  \langle h \cdot \one_{H'}, \phi_{P_3}^3 \rangle.
\]  
 
While this operator satisfies the same estimates as the bilinear Hilbert transform, the localization to the sets $F, G$ and $H'$, and  the restriction to the tiles in $\rr{P}(I_0)$ will bring some extra decay. First we prove a result in the``local $L^2$ case", when $\ds \frac{1}{r_1}, \frac{1}{r_2}, \frac{1}{r'} < \frac{1}{2}$. In this situation the proof is simpler, because we are employing ``energies", which are $L^2$ expressions, and they can easily be related to $L^{r_i}-$ averages when $r_i \geq 2$.

\begin{proposition}[The case $r_1, r_2, r' > 2$]
\label{Localization for local L^2 BHT}
Let $\rr{P}$ be a family of tri-tiles, $I_0$ a dyadic interval and $F, G, H' \subset \rr{R}$ sets of finite measure. Then one can find positive numbers $a_1, a_2$ and $a_3$ so that
\begin{align}
\label{localization local L^2}
\left|\Lambda_{BHT;\rr{P}(I_0)}^{F, G, H'}(f, g , h )\right| & \lesssim  \left( \ssize_{\rr{P}(I_0)} \one_F  \right)^{a_1} \left( \ssize_{\rr{P}(I_0)} \one_G  \right)^{a_2} \left( \ssize_{\rr{P}(I_0)} \one_{H'}  \right)^{a_3} \nonumber\\
& \qquad \| f \cdot \ci_{I_0} \|_{r_1} \| g \cdot \ci_{I_0} \|_{r_2} \| h  \cdot \ci_{I_0} \|_{r'}.
\end{align}
We can choose  $\ds a_j=1-\frac{2}{r_j}- \epsilon>0$, for a very small $\epsilon > 0$.
\begin{proof}
In this case we are proving restricted type estimates by applying directly Proposition \ref{BHT trilinear estimate}: let $E_1, E_2, E_3$ be sets of finite measure, and $|f| \leq \one_{E_1}, |g| \leq \one_{E_2}, |h| \leq \one_{E_3}. $ We have
\begin{align}
\label{eq: local L^2 BHT localization}
\Lambda_{BHT}(f \cdot \one_F, g \cdot \one_G, h \cdot \one_{H'}) &\lesssim \left( \ssize_{\rr{P}(I_0)} f \cdot \one_F  \right)^{\theta_1} \left( \ssize_{\rr{P}(I_0)} g \cdot \one_G  \right)^{\theta_2} \left( \ssize_{\rr{P}(I_0)} h \cdot \one_{H'} \right)^{\theta_3} \nonumber \\
& \cdot \left( \eenergy f \cdot \one_{F} \right)^{1-\theta_1} \left( \eenergy g \cdot \one_{G} \right)^{1-\theta_2} \left( \eenergy h \cdot \one_{H'} \right)^{1-\theta_3},
\end{align}
for any $0 \leq \theta_1, \theta_2, \theta_3<1$ such that $\theta_1+\theta_2+\theta_3=1$.
Recall that the sizes can be estimated by 
\[
\ssize_{\rr{P}(I_0)} f \cdot \one_{F} \lesssim \sup_{P \in \rr{P}(I_0)} \frac{1}{|I_P|} \int \one_{E_1} \cdot \one_{F} \cdot \ci_{I_P}^M dx
\]
where $M$ can be chosen as large as we wish. Then we observe that if $E_1$ is supported away from $I_0$, the sizes will decay fast, giving the desired $\ds \| f \cdot \ci_{I_0}  \|_{r_1}$ on the right hand side. Similarly for $E_2$ and $E_3$. For this reason, we can assume that the sets $E_1, E_2, E_3$ are supported on $5I_0$ and then we will need to show only that
\begin{align*}
|\Lambda_{BHT;\rr{P}(I_0)}(f \cdot \one_F, g \cdot \one_G, h \cdot \one_{H'} )| \lesssim &
 \left( \ssize_{\rr{P}(I_0)} \one_F  \right)^{a_1} \left( \ssize_{\rr{P}(I_0)} \one_G  \right)^{a_2} \left( \ssize_{\rr{P}(I_0)} \one_{H'}  \right)^{a_3}  \\
 &\quad \cdot \| f \|_{r_1} \| g \|_{r_2} \| h \|_{r'}.
\end{align*}
We are using the energies precisely for estimating the norms of $f, g$ and $h$, so the sizes are playing the role of a constant here. As we have seen in Lemma \ref{BHT energy estimate}, the energies are bounded by $L^2$ norms, so from \eqref{eq: local L^2 BHT localization}, we have
\begin{align*}
&\Lambda_{BHT; \rr P \left( I_0 \right)}^{F, G, H'}(f , g , h ) \lesssim \left( \ssize_{\rr{P}(I_0)}  \one_F  \right)^{\theta_1} \left( \ssize_{\rr{P}(I_0)} \one_G  \right)^{\theta_2} \left( \ssize_{\rr{P}(I_0)} \one_{H'} \right)^{\theta_3} |E_1|^\frac{1-\theta_1}{2} |E_2|^\frac{1-\theta_2}{2} |E_3|^{\frac{1-\theta_3}{2}}.
\end{align*}
By varying $\theta_1, \theta_2$ and $\theta_3$, we see that these restricted type estimates are true in a very small neighborhood of $\ds (\frac{1}{r_1}, \frac{1}{r_2}, \frac{1}{r'})$, and the interpolation Theorem \ref{interp thm restricted type} yields strong type estimates. Note that the constant in this case is 
\[
 \left( \ssize_{\rr{P}(I_0)}  \one_F  \right)^{\theta_1} \left( \ssize_{\rr{P}(I_0)} \one_G  \right)^{\theta_2} \left( \ssize_{\rr{P}(I_0)} \one_{H'} \right)^{\theta_3},
\]
which depends on the functions $\one_F, \one_G, \one_{H'}$, the fixed interval $I_0$, the values of $\theta_1, \theta_2,$ and $\theta_3$, but not on the functions $f, g, h$. 
\end{proof}
\end{proposition}

Now we deal with the general Banach triangle case, where $\ds  \left(\frac{1}{r_1}, \frac{1}{r_2}, \frac{1}{r'}\right)$ is an admissible tuple satisfying
\[
0 <  \frac{1}{r_1}, \frac{1}{r_2}, \frac{1}{r'} <1.
\]

The proof is going to be more complicated because we will need to use the sizes as well, for reconstructing the norms of $f, g, h$. In addition, we will also need to use the sizes of $\one_F, \one_G$ and $\one_{H'}$ later on. 

\begin{proposition}
\label{Localization for local L^1 BHT} Let $F, G$ and $H'$ be as above and let $\rr{P}(I_0)$ be a family of tri-tiles localized to the dyadic interval $I_0$. Then there exist positive numbers $a_1, a_2$ and $a_3$ so that
\begin{align}
\label{localization estimate}
 \quad  & \lft\Lambda_{BHT;\rr{P}(I_0)}^{F, G, H'}(f, g , h )\rg\\
&\lesssim \left( \widetilde{\ssize}_{\rr{P}(I_0)} \one_F  \right)^{a_1} \left( \widetilde{\ssize}_{\rr{P}(I_0)} \one_G  \right)^{a_2} \left( \widetilde{\ssize}_{\rr{P}(I_0)} \one_{H'}  \right)^{a_3} \| f \cdot \ci_{I_0} \|_{r_1} \| g \cdot \ci_{I_0} \|_{r_2} \| h \cdot \ci_{I_0} \|_{r'}, \nonumber
\end{align}
where $\displaystyle \frac{1}{r_1}+\frac{1}{r_2}+\frac{1}{r'}=1$. In fact, for $\epsilon >0$ small enough,
\begin{equation}
\label{localization a_j}
a_1=\frac{1+\theta_1}{2}-\frac{1}{r_1}-\epsilon, \quad a_2=\frac{1+\theta_2}{2}-\frac{1}{r_2}-\epsilon, \quad  a_3=\frac{1+\theta_3}{2}-\frac{1}{r'}-\epsilon,
\end{equation}
where $\theta_1, \theta_2, \theta_3$ are so that $0 \leq \theta_1, \theta_2, \theta_3<1$, $\theta_1+\theta_2+\theta_3=1$, and the expressions in \eqref{localization a_j} are positive.
\begin{proof}
In this case, we will use the interpolation Theorem \ref{interpolation for bad index}, and for this reason we cannot obtain directly the expression in the RHS of \eqref{localization estimate}, which represents \emph{localized} $L^p$ norms. However, as we will see soon, it will be enough to prove that $\Lambda_{BHT; \rr{P}(I_0)}$ is of generalized restricted type $\alpha=\left(\alpha_1, \alpha_2, \alpha_3 \right)$, for $\alpha$ in a small neighborhood of $\ds \left( \frac{1}{r_1}, \frac{1}{r_2}, \frac{1}{r'}  \right)$. Then the result in \eqref{localization estimate} will be a consequence of the fast decay of the wave packets away from $I_0$.

We start with $E_1, E_2, E_3$ sets of finite measure and define $\tilde{\Omega}$ to be the exceptional set:
\[
\tilde{\Omega}:= \left \lbrace x: \mathcal{M}(\one_{E_1}) > C \frac{|E_1|}{|E_3|}   \big \rbrace \cup \big \lbrace  x: \mathcal{M}(\one_{E_2}) > C \frac{|E_2|}{|E_3|  } \right \rbrace.
\]
Let $E_3':=E_3 \setminus \tilde{\Omega}$. We want to prove that \eqref{localization estimate} holds for any functions $f, g, h$ so that $|f| \leq \one_{E_1}, |g| \leq \one_{E_2}$, and $|h| \leq \one_{E_3'}$. For simplicity, we assume that $ \ds 1+ \frac{\dist (I_P, \tilde{\Omega}^c)}{|I_P|} \sim 2^d$ for every tile $P \in \rr{P}(I_0)$. Equivalently, we could decompose the collection of tiles into subcollections for which this property holds, for all $d \geq 0$. In the end however, the estimate \eqref{localization estimate} will be independent of such a decomposition.

With the above assumption, for every $P \in \rr{P}(I_0)$, we have 
\[
\frac{1}{|I_P|} \int_\rr{R} \one_{E_1} \cdot \one_F \cdot \ci_{I_P}^M dx \lesssim 2^d \frac{|E_1|}{|E_3|} \quad  \text{and } \quad \frac{1}{|I_P|} \int_\rr{R} \one_{E_2} \cdot \one_G \cdot \ci_{I_P}^M dx \lesssim 2^d \frac{|E_2|}{|E_3|}.
\]
This is important because now we can perform a stopping time which will allow us to estimate the `sizes' of the functions $\one_{E_j}$. For each of the functions $\one_F \cdot \one_{E_1}$,  $\one_G \cdot \one_{E_2}$ and $\one_{H'} \cdot \one_{E_3'}$, we will be looking for maximal dyadic intervals $J$ which are maximizers for
\begin{equation}
\label{def:new_size}
 \sup_{\substack{J \subseteq I_0 \\ \exists P \in \rr{P}(I_0) , I_P \subseteq J}} \frac{1}{\vert J \vert} \int_{\rr{R}} \one_{E_1} \cdot \one_{F} \cdot \ci_{J}^M dx.
\end{equation}
This is the reason we introduced the new size in Definition \ref{def:modified-size}.

The selection of the intervals and tiles is described in more detail in Section \ref{local L1}, so here we only sketch this process. 

We start with the largest possible value $2^{-l_1} \lesssim 2^d \dfrac{|E_1|}{|E_2|}$ and define $\ii{I}_{l_1}$ to be the collection of maximal dyadic intervals $I$ with the property that it contains some $I_P \in \rr{P}(I_0)$ which is not contained in any of the intervals previously selected, and $I$ also has the property that
\[
2^{-l_1-1} \leq \frac{1}{|I|} \int_{\rr{R}} \one_{E_1} \cdot \one_{F} \cdot \ci_{I}^M dx \leq 2^{-l_1}.
\]
Then for each $I \in \ii{I}_{l_1}$ we find the relevant tiles $P$ with $I_P \subseteq I$, and move them in $\rr{P}(I)$. Afterwards we restart the algorithm for the collection $\rr{P}(I_0) \setminus \cup_{I \in \ii{I}_{l_1}}\rr{P}(I)$. 

The algorithm continues by decreasing  $2^{-l_1}$ until all tiles in $\rr{P}(I_0)$ are exhausted. In this way, for any $l_1$ and any $I \in \ii{I}_{l_1}$ we have $\ds \sssize_{\rr{P}(I)}(\one_{E_1} \cdot \one_{F}) \sim 2^{-l_1}$. Similarly we define the collections of dyadic intervals $\ii{I}_{l_2}$ associated with the functions $\one_{E_2} \cdot \one_G$ as long as $2^{-l_2} \lesssim 2^d \dfrac{|E_2|}{|E_3|}$.

For the third component, the collections $\ii{I}_{l_3}$ are non-empty as long as $2^{-n_3} \lesssim 2^{-\tilde{M}d}$, and in that case, for any $I \in \ii{I}_{l_3}$, we have $\ds \sssize_{\rr{P}(I)}(\one_{H'} \cdot \one_{E'_3}) \sim 2^{-n_3}$. The extra decay is due to the fact that $E_3'$ is actually supported on $\tilde{\Omega}^c$.

Given $l_1, l_2, l_3$ as above, we denote $\ii{I}^{l_1, l_2, l_3}:= \ii{I}_{l_1} \cap \ii{I}_{l_2} \cap \ii{I}_{l_3}$. This is also going to be a collection of dyadic intervals, and any tile in $\rr P(I_0)$ will be contained in some $\rr P \left( I \right)$, with $I \in \ii I ^{l_1, l_2, l_3}$. In fact, these collections depend on the parameter $d$ as well, which controls the distance from the exceptional set. We have $\ds \rr{P}(I_0)= \bigcup_d\bigcup_{l_1, l_2, l_3} \bigcup_{I \in \ii{I}_d^{l_1, l_2, l_3}} \rr{P}(I)$, but we suppress the dependency on $d$ in the notation. Thus
\begin{equation}
\label{decomposition local BHT}
\Lambda^{F, G, H'}_{BHT; \rr{P}(I_0)}(f, g, h)=\sum_{l_1, l_2, l_3} \sum_{I \in \ii{I}^{l_1, l_2, l_3}} \Lambda^{F, G, H'}_{BHT; \rr{P}(I)}(f, g, h).
\end{equation}
Every $\ds \Lambda^{F, G, H'}_{BHT; \rr{P}(I)}(f, g, h)$ is going to be estimated by Lemma \ref{lemma:refined-trilinear-estimate}:
\begin{align*}
\Lambda^{F, G, H'}_{BHT; \rr{P}(I)}(f, g, h) &\lesssim \sssize_{\rr{P}(I)}(\one_{E_1} \cdot \one_{F})^{\theta_1} \sssize_{\rr{P}(I)}(\one_{E_2} \cdot \one_{G})^{\theta_2} \sssize_{\rr{P}(I)}(\one_{E'_3} \cdot \one_{H'})^{\theta_3}  \\
& \quad \quad \cdot \left\|\one_{E_1} \cdot \one_{F} \cdot \ci_{I}\right\|_2^{1 -\theta_1}\left\|\one_{E_2} \cdot \one_{G} \cdot \ci_{I}\right\|_2^{1 -\theta_2} \left\|\one_{E_3'} \cdot \one_{H'} \cdot \ci_{I}\right\|_2^{1 -\theta_3}.
\end{align*}

For the particular function $\one_{E_1} \cdot \one_F$ and an interval $I  \in \ii{I}^{l_1, l_2, l_3}$ we have
\[
\left(\int_{\rr{R}} \one_{E_1} \cdot \one_F \cdot \ci_I^M dx    \right)^{1/2} \lesssim 2^{-l_1/2} |I|^{1/2} \lesssim \left( \sssize_{\rr{P}(I)}(\one_{E_1} \cdot \one_F) \right)^{1/2} |I|^{1/2}.
\]
In this way, as long as
\begin{equation}
\label{condition}
\frac{1+\theta_1}{2}-\frac{1}{r_1} >0, \quad \frac{1+\theta_2}{2}-\frac{1}{r_2} >0, \quad \frac{1+\theta_3}{2}-\frac{1}{r'} >0,
\end{equation} we can estimate  $\Lambda^{F, G, H'}_{BHT; \rr{P}(I_0)}(f, g, h)$ as
\begin{align}
\label{local BHT sum}
&\Lambda^{F, G, H'}_{BHT; \rr{P}(I_0)}(f, g, h) \lesssim \sum_{l_1, l_2, l_3} \sum_{I \in \ii{I}^{l_1, l_2, l_3}} \sssize_{\rr{P}(I)}(\one_{E_1} \cdot \one_F)^{\theta_1} \sssize_{\rr{P}(I)}(\one_{E_2} \cdot \one_G)^{\theta_2} \sssize_{\rr{P}(I)}(\one_{E'_3} \cdot \one_{H'})^{\theta_3} \\
& \left(\frac{1}{|I|} \int_{\rr{R}} \one_{E_1} \cdot \one_{F} \cdot \ci_I^M dx   \right)^{\frac{1-\theta_1}{2}}  \left(\frac{1}{|I|} \int_{\rr{R}} \one_{E_2} \cdot \one_{G} \cdot \ci_I^M dx   \right)^{\frac{1-\theta_2}{2}}  \left(\frac{1}{|I|} \int_{\rr{R}} \one_{E'_3} \cdot \one_{H'} \cdot \ci_I^M dx   \right)^{\frac{1-\theta_3}{2}} \cdot |I| \nonumber \\
& \lesssim \sssize_{\rr{P}(I_0)}(\one_{F})^{\frac{1+\theta_1}{2} -\frac{1}{r_1}} \cdot \sssize_{\rr{P}(I_0)}(\one_{G})^{\frac{1+\theta_2}{2} -\frac{1}{r_2}} \cdot \sssize_{\rr{P}(I_0)}(\one_{H'})^{\frac{1+\theta_3}{2} -\frac{1}{r'}- \epsilon} \nonumber \\
& \cdot \sum_{l_1, l_2, l_3} \sum_{I \in \ii{I}^{l_1, l_2, l_3}} 2^{-\frac{l_1}{r_1}} 2^{-\frac{l_2}{r_2}} 2^{-l_3(\frac{1}{r'}+\epsilon)} |I|, \nonumber
\end{align}

The quantity 
\[
\ds \sssize_{\rr{P}(I_0)}(\one_{F})^{\frac{1+\theta_1}{2} -\frac{1}{r_1}} \sssize_{\rr{P}(I_0)}(\one_{G})^{\frac{1+\theta_2}{2} -\frac{1}{r_2}} \sssize_{\rr{P}(I_0)}(\one_{H'})^{\frac{1+\theta_3}{2} -\frac{1}{r'}- \epsilon}
\] 
is going to represent the operatorial norm $\ds \left\| \Lambda^{F, G, H'}_{BHT; \rr P \left(I_0 \right)} \right\|$ associated to the trilinear form $\Lambda_{BHT; \rr{P}(I_0)}^{F, G, H}$, as seen in \eqref{localization estimate}.

We are left with estimating $\ds \sum_{I \in \ii{I}^{l_1, l_2, l_3}} |I|$, which can be realized in three different ways; for example,
\[
 \sum_{I \in \ii{I}^{l_1, l_2, l_3}} |I| \leq \sum_{I \in \ii{I}_{l_1}} |I|=\| \sum_{I \in \ii{I}_{l_1}} \one_{I}  \|_{1, \infty} \lesssim \| \sum_{I \in \ii{I}_{l_1}} 2^{l_1} \mathcal{M} \one_{E_1} \cdot \one_I   \|_{1, \infty} \lesssim 2^{n_1} |E_1|. 
\]
For this reason, whenever $0 \leq \alpha_j \leq 1$, with $\alpha_1+\alpha_2+\alpha_3=1$, we have
\[
 \sum_{I \in \ii{I}^{l_1, l_2, l_3}} |I| \lesssim \left(  2^{l_1} |E_1|\right)^{\alpha_1} \left(  2^{l_2} |E_2|\right)^{\alpha_2} \left(  2^{l_3} |E'_3|\right)^{\alpha_3}.
\]
This yields
\begin{align*}
&\sum_{l_1, l_2, l_3} \sum_{I \in \ii{I}^{l_1, l_2, l_3}} 2^{-\frac{l_1}{r_1}} 2^{-\frac{l_2}{r_2}} 2^{-l_3(\frac{1}{r'}+\epsilon)} |I|  \\
& \lesssim \sum_{l_1, l_2, l_3} 2^{-l_1(\frac{1}{r_1}- \alpha_1)} 2^{-l_2(\frac{1}{r_2}- \alpha_2)} 2^{-l_3(\frac{1}{r'}+\epsilon- \alpha_1)} |E_1|^{\alpha_1} |E_2|^{\alpha_2}|E_3|^{\alpha_3} \\
& \lesssim \left( 2^d \frac{|E_1|}{|E_3|}   \right)^{\frac{1}{r_1}-\alpha_1} \left( 2^d \frac{|E_2|}{|E_3|}   \right)^{\frac{1}{r_2}-\alpha_2} \left(  2^{-\tilde{M d}} \right)^{(\frac{1}{r'}+\epsilon -\alpha_3)}  |E_1|^{\alpha_1} |E_2|^{\alpha_2}|E_3|^{\alpha_3}  \\
&\lesssim 2^{-100d} |E_1|^{1/{r_1}}  |E_2|^{1/{r_2}}  |E_3|^{1/{r'}}.
\end{align*}
Summing over $d$, this proves \eqref{localization estimate} in the particular case of characteristic functions. Upon interpolating, we lose an $\epsilon$ power of $\sssize_{\rr{P}(I_0)} \one_F$ and $\sssize_{\rr{P}(I_0)} \one_G$ respectively, to get
\begin{align*}
\lft\Lambda_{BHT;\rr{P}(I_0)}^{F, G, H'}(f , g , h )\rg 
&\lesssim \left( \widetilde{\ssize}_{\rr{P}(I_0)} \one_F  \right)^{a_1} \left( \widetilde{\ssize}_{\rr{P}(I_0)} \one_G  \right)^{a_2} \left( \widetilde{\ssize}_{\rr{P}(I_0)} \one_{H'}  \right)^{a_3} \\
&\qquad \cdot\| f \cdot \ci_{I_0} \|_{r_1} \| g \cdot \ci_{I_0} \|_{r_2} \| h \cdot \ci_{I_0} \|_{r'}.
\end{align*}

We note that the ``weights'' $\ci_{I_0}$ will not affect the interpolation process; once we have an inequality that holds for characteristic functions of finite sets, interpolation implies a similar result in full generality.

The exponents $a_1, a_2$ and $a_3$ can be described as
\[
a_1=\frac{1+\theta_1}{2}-\frac{1}{r_1}-\epsilon, \quad a_2=\frac{1+\theta_2}{2}-\frac{1}{r_2}-\epsilon, a_3=\frac{1+\theta_3}{2}-\frac{1}{r'}-\epsilon
\]
for some sufficiently small $\epsilon$, and for $0 \leq \theta_1, \theta_2, \theta_3 <1$, satisfying $\theta_1+\theta_2+\theta_3=1$ that will be chosen later. 
\end{proof}
\end{proposition}

\begin{corollary}[The case $r=1$]
\label{Localization for BHT with  target L^1 }Let $1< r_1, r_2 < \infty$ be so that  $\ds \frac{1}{r_1}+ \frac{1}{r_2}=1$, and $\theta_1, \theta_2$ satisfying $\ds \frac{1+\theta_1}{2} >  \frac{1}{r_1}$ and $\ds \frac{1+\theta_2}{2} >  \frac{1}{r_2}$. Then 
\begin{align*}
\|BHT_{\rr P \left( I_0 \right)}^{F, G, H'}(f,g)   \|_1 &\lesssim \left( \sssize_{\rr{P}(I_0)} \one_F  \right)^{\frac{1+\theta_1}{2}-\frac{1}{r_1}-\epsilon} \cdot \left( \sssize_{\rr{P}(I_0)} \one_G  \right)^{\frac{1+\theta_2}{2}-\frac{1}{r_2}-\epsilon} \cdot \left( \sssize_{\rr{P}(I_0)} \one_{H'} \right)^{\frac{1+\theta_3}{2}-\epsilon}  \\
&\qquad \cdot\| f  \cdot \ci_{I_0}  \|_{r_1} \| g \cdot \ci_{I_0}  \|_{r_2}.
\end{align*}
\begin{proof}
A careful examination of \eqref{local BHT sum} shows that one can choose any triple $(\beta_1, \beta_2, \beta_3)$ with $\beta_1+\beta_2+\beta_3=1$, even with $\beta_3 \leq 0$, in the place of $(1/r_1, 1/r_2, 1/{r'})$. In this case we get
\begin{align*}
\vert \Lambda_{BHT; \rr{P}(I_0)}^{F, G, H'}(f, g, h) \vert \lesssim &\left( \widetilde{\ssize}_{\rr{P}(I_0)} \one_F  \right)^{\frac{1+\theta_1}{2}-\beta_1} \left( \widetilde{\ssize}_{\rr{P}(I_0)} \one_G  \right)^{\frac{1+\theta_2}{2}-\beta_2} \left( \widetilde{\ssize}_{\rr{P}(I_0)} \one_{H'}  \right)^{\frac{1+\theta_3}{2} -\epsilon} \\
&\qquad \cdot |E_1|^{\beta_1} |E_2|^{\beta_2} |E_3|^{\beta_3}
\end{align*}
The restrictions are that $\ds \beta_j < \frac{1+\theta_j}{2}$, which works well for very small or negative values of $\beta_3$. Interpolating between tuples $(\beta_1, \beta_2, \beta_3)$ that lie in a small open neighborhood of $\ds \left( \frac{1}{r_1}, \frac{1}{r_2}, 0 \right)$, we get the conclusion. 
In this case, the interpolation is used for estimating the $L^1$ norm of the operator, and not the trilinear form $\Lambda_{BHT; \rr{P}(I_0)}^{F, G, H'}$.
\end{proof}
\end{corollary}

\subsection{Proof of Theorem \ref{vector valued BHT}}
\label{local L1}
Recall that the vector-valued $BHT$ is defined by:
\[
BHT(f,g)(x,w)=\int_{\rr{R}} f(x-t, w) g(x+t, w) \frac{dt}{t}=BHT(f_w, g_w)(x).
\]
Then the trilinear form associated with it is
\[
\Lambda_{\overrightarrow{BHT}}(f,g,h)=\int_{\rr{R}} \int_{\ii{W}} BHT(f,g)(x, w) h(x,w) d \mu(w) dx.
\]

First we prove generalized restricted type estimates for $\Lambda_{\vBHT}(f,g,h)$, and the general result will follow from the vector-valued  interpolation result presented in Proposition \ref{banach interpolation gen res}. Let $F, G$ and $H$ be sets of finite measure. In what follows, we will construct a major subset $H' \subseteq H$ and show
\begin{equation}
\label{gen restriction for linearization}
\vert \Lambda_{\vBHT; \rr{P}}(f, g, h)  \vert \lesssim |F|^{\alpha_1}  |G|^{\alpha_2} |H|^{\alpha_3} 
\end{equation} 
whenever $\ds \| f(x, \cdot)  \|_{L^{r_1}(\ii{W}, \mu)}  \leq \one_F(x),  \| g(x, \cdot)  \|_{L^{r_2}(\ii{W}, \mu)}\leq \one_G(x)$ and $ \| h(x, \cdot)  \|_{L^{r'}(\ii{W}, \mu)} \leq \one_{H'}(x)$. For simplicity, assume $|H|=1$. 
The exceptional set is defined as
\[
\Omega :=\lbrace x: \mathcal{M}(\one_F) > C |F| \rbrace \cup \lbrace x: \mathcal{M}(\one_G) > C |G|  \rbrace.
\]
Because of the $L^1 \to L^{1, \infty}$ boundedness of the maximal operator, for a constant $C$ large enough, we have $\vert\Omega \vert \ll 1$.

We partition the collection of tri-tiles according to the scaled distance from  the exceptional set
\[
\rr{P}^d= \lbrace P \in \rr{P} : 1 +\frac{ \dist (I_P, \Omega^c) }{|I_P|} \sim 2^d  \rbrace
\]
and we will prove estimates equivalent to \eqref{gen restriction for linearization} for the family $\rr{P}^d$, with an extra $2^{-10d}$ decay:
\begin{equation}
\label{tri-linear BHT 2^d}
\vert \Lambda_{\vBHT; \rr{P}^d}(f, g, h)\vert \lesssim 2^{-10d}|F|^{1/p} |G|^{1/q} |H|^{1/{s'}}.
\end{equation}
 We suppress the $d$-dependency for the moment, but all the subcollections $\ii I ^{n_j}_j$ and $\ii I ^{n_1, n_2, n_3}$ will actually depend on this parameter. At the very end we sum in $d$, and use interpolation, so that the final estimate depends only on the fixed interval $I_0$, and the fixed sets $F, G, H'$. 

Now we construct a collection $\lbrace \mathscr{I_1}^{n_1} \rbrace_{n_1 \geq \bar{n}_1}$ of relevant dyadic intervals,  according to the concentration of $\one_F$:
\begin{itemize}
\item[-] start with $\bar{n}_1$ so that $2^{- \bar{n}_1} \sim 2^d |F|$ and let $\rr{P}_{\bar{n}_1-1}'=\rr{P}$ ( here $\rr{P}_{n_1}'$ will play the role of \emph{Stock}, or collection of available tiles)
\item[-] define $\mathscr{I}_1^{\bar{n}_1}$ to be the collection of maximal dyadic intervals $I$ with the property that there exists at least one tile $P \in \rr{P}_{\bar{n}_1}'$ with $I_P \subseteq I$ and 
\begin{equation}
\frac{1}{|I|} \int \one_F \cdot \ci_{I}^M dx \sim 2^{- \bar{n}_1}.
\end{equation}
\item[-] for every such interval $I$, let $\rr{P}_{\bar{n}_1}(I)$ be the collection of tiles $P \in \rr{P}'_{\bar{n}_1}$ with the property that $I_P \subseteq I $
\item[-] set $ \ds \rr{P}_{\bar{n}_1}'=\rr{P} \setminus \bigcup _{I \in \mathscr{I}_1^{\bar{n}_1}} \rr{P}_{\bar{n}_1}(I)$
\item[-] repeat the procedure for all $n_1 \geq \bar{n}_1$; $\mathscr{I}_1^{n_1}$ will denote the collection of maximal dyadic intervals which contain a time interval $I_P$ for some $P \in \rr{P}'_{n_1+1}$ (which was not selected previously) and so that
\[
2^{-n_1-1} \leq \frac{1}{|I|} \int \one_F \cdot \ci_I^M dx < 2^{- n_1}. 
\]
\item[-] as before, $\ds \rr{P}_{n_1}(I):= \lbrace P \in \rr{P}'_{n_1} : I_P \subseteq I   \rbrace$
\item[-] set $\ds \rr{P}_{n_1}'=\rr{P}_{n_1+1} \setminus \bigcup _{I \in \mathscr{I}_1^{n_1}} \rr{P}_{n_1}(I)$ and notice that after a finite number of steps, $\ds \rr{P}_{n_1}'=\emptyset$.
\item[-]note that we always have $2^{-n_1} \lesssim 2^d |F|$.
\end{itemize}

For $d$ sufficiently large, the intervals $I_P$ for $P \in \rr{P}^d$ are going to be essentially disjoint and the intervals $I \in \mathscr{I}_1^{n_1}$ can be selected in an easier way; but this is not the case for example when $d=0$, which corresponds to $I_P \cap \Omega^c \neq \emptyset$. However, for every $n_1$, the intervals in $\mathscr{I}_1^{n_1}$ are going to be disjoint and this is going to be used later in the proof.

Similarly, $\mathscr{I}_2^{n_2}$ denotes the collection of maximal dyadic intervals $I$ containing at least some $I_P \subseteq I$ for some $P \in \rr{P}^d$, and 
\[
\frac{1}{|I|} \int \one_{G} \cdot \ci_I^M dx \sim 2^{- n_2} \lesssim 2^d |G|.
\]
For $\one_{H'}$, $\mathscr{I}_3^{n_3}=$ collection of maximal dyadic intervals $I$ containing at least some $I_P$ for some $P \in \rr{P}^d$ and so that 
\[
\frac{1}{|I|} \int \one_{H'} \cdot \ci_I^M dx \sim 2^{-n_3} \lesssim 2^{-Md}.
\] 
We denote $\ds \mathscr{I}^{n_1, n_2, n_3}:=\mathscr{I}_1^{n_1} \cap \mathscr{I}_2^{n_2} \cap \mathscr{I}_3^{n_3}$, and we further partition $\rr{P}^d$ as $\ds \rr{P}^d= \bigcup_{n_1, n_2, n_3} \bigcup_{ I \in \mathscr{I}^{n_1, n_2, n_3}} \rr{P}(I)$.

For $I \in \mathscr{I}_1^{n_1}$, we have $\ds \sssize_{\rr{P}_{n_1}(I)} \one_F \sim 2^{-n_1}$. When we consider the intersection $I'$ of different intervals in $\mathscr{I}_1^{n_1}, \mathscr{I}_2^{n_2}$ and $\mathscr{I}_3^{n_3}$ all we can say is that $\ds \sssize_{\rr{P}(I')} \one_F \lesssim 2^{-n_1}$. This fact is the technical obstruction in obtaining vector-valued $BHT$ estimates for any $p, q, s$ in the whole range of $BHT$.

In a similar way, the relation  $\ds \frac{1}{|I|} \int_{\rr{R}} \one_F \cdot \ci_I^M dx \sim 2^{-n_1}$ for $I \in \mathscr{I}_1^{n_1}$ becomes for an interval $I' \in \mathscr{I}_1^{n_1} \cap \mathscr{I}_2^{n_2} \cap \mathscr{I}_3^{n_3}$ an inequality: $\ds \frac{1}{|I'|} \int_{\rr{R}} \one_F \cdot \ci_{I'}^M dx \lesssim 2^{-n_1}.$

The trilinear form in \eqref{tri-linear BHT 2^d} becomes
\begin{align*}
&\sum_{n_1, n_2, n_3} \sum_{I \in \mathscr{I}^{n_1, n_2, n_3}} \Lambda_{\vBHT; \rr{P}(I)}(f, g, h) \\
& =\sum_{n_1, n_2, n_3} \sum_{I \in \mathscr{I}^{n_1, n_2, n_3}} \int_{\rr{R}}\int_{\ii{W}} BHT_{\rr{P}(I)}(f_w, g_w)(x) \cdot h_w(x) d\mu(w) dx\\
&= \int_{\ii{W}}\left(\sum_{n_1, n_2, n_3} \sum_{I \in \mathscr{I}^{n_1, n_2, n_3}} \int_{\rr{R}} BHT_{\rr{P}(I)} (f_w \cdot \one_F, g_w \cdot \one_G)(x) \cdot \one_{H'}(x) \cdot h_w(x) dx   \right) d\mu(w)
\end{align*}
Note that the functions $f_w$ are supported on $F$, the $g_w$ on $G$ and the $h_w$ on $H'$, for a.e. $w$. We can apply the localization Proposition \ref{Localization for local L^1 BHT} to get
\begin{align*}
|\Lambda_{BHT; \rr{P}(I)}^{F, G, H'}(f_w, g_w, h_w) | \lesssim & \left( \sssize_{\rr{P}(I)} \one_F   \right)^{a_1} \left( \sssize_{\rr{P}(I)} \one_G   \right)^{a_2} \left( \sssize_{\rr{P}(I)} \one_{H'}   \right)^{a_3}  \\
& \| f_w \cdot \ci_I   \|_{r_1} \| g_w \cdot \ci_I   \|_{r_2} \| h_w \cdot \ci_I   \|_{r'},
\end{align*}
where $\ds \frac{1}{r_1}+\frac{1}{r_2}+\frac{1}{r'}=1$.

Recall the expressions for $a_j$ from \eqref{localization a_j}:
\[
a_1=\frac{1+\theta_1}{2}-\frac{1}{r_1}-\epsilon, \quad a_2=\frac{1+\theta_2}{2}-\frac{1}{r_2}- \epsilon, \quad a_3=\frac{1+\theta_3}{2}-\frac{1}{r'}-\epsilon,
\]
where the only conditions we have on $\theta_1, \theta_2$ and $\theta_3$ are that $\theta_1+\theta_2+\theta_3=1$ and $a_j >0$. Using H\"{o}lder's inequality, the initial trilinear form can be estimated by
\begin{align*}
& \sum_{n_1, n_2, n_3} \sum_{I \in \mathscr{I}^{n_1, n_2, n_3}} \int_{\ii{W}} |\Lambda_{BHT; \rr{P}(I)}(f_w, g_w, h_w) | \\
&\lesssim \sum_{n_1, n_2, n_3} \sum_{I \in \mathscr{I}^{n_1, n_2, n_3}} \left( \sssize_{\rr{P}(I)} \one_F   \right)^{a_1} \left( \sssize_{\rr{P}(I)} \one_G   \right)^{a_2} \left( \sssize_{\rr{P}(I)} \one_{H'}   \right)^{a_3}   \\
&\qquad \left(\int_{\ii{W}} \| f_w \cdot \ci_{I}\|_{r_1}^{r_1} d \mu(w)\right)^{1/r_1} \cdot \left(\int_{\ii{W}} \| g_w \cdot \ci_{I}\|_{r_2}^{r_2} d \mu(w)\right)^{1/r_2} \cdot \left(\int_{\ii{W}} \| h_w\cdot\ci_{I}\|_{r'}^{r'} d \mu(w)\right)^{1/{r'}} \\
\end{align*}
\begin{align*}
&\lesssim \sum_{n_1, n_2, n_3} \sum_{I \in \mathscr{I}^{n_1, n_2, n_3}} \left( \sssize_{\rr{P}(I)} \one_F   \right)^{a_1} \left( \sssize_{\rr{P}(I)} \one_G   \right)^{a_2} \left( \sssize_{\rr{P}(I)} \one_{H'}   \right)^{a_3}  \\
&\qquad \frac{\| \one_F \cdot \ci_I   \|_{r_1}}{|I|^{1/{r_1}}}  \frac{\| \one_G \cdot \ci_I   \|_{r_2}}{|I|^{1/{r_2}}} \frac{\| \one_{H'} \cdot \ci_I   \|_{r'}}{|I|^{1/{r'}}} \cdot |I|  \\
& \lesssim \sum_{n_1, n_2, n_3} \sum_{I \in \mathscr{I}^{n_1, n_2, n_3}} 2^{- n_1/p} 2^{- n_2/q} 2^{-n_3 (a_3+\frac{1}{r'})}  \cdot |I|
\end{align*}
In the last inequality we need to assume $\ds \frac{1}{p} \leq a_1+\frac{1}{r_1}=\frac{1+ \theta_1}{2}$ and similarly $\ds \frac{1}{q} \leq \frac{1+ \theta_2}{2}$.
We will be summing $|I|$ when $I \in \mathscr{I}^{n_1, n_2, n_3}$. Note that
\begin{align*}
&\sum_{I \in \mathscr{I}^{n_1, n_2, n_3}} |I| \leq \sum_{I \in \mathscr{I}_1^{n_1}} |I|= \| \sum_{I \in \mathscr{I}^{n_1}_1} \one_{I}  \|_{1, \infty} \lesssim  \| \sum_{I \in \mathscr{I}_1^{n_1}} 2^{n_1} (\mathcal{M} \one_F) \cdot \one_I   \|_{1, \infty} \lesssim 2^{n_1} |F|.
\end{align*} 
Similarly, $\ds \sum_{I \in \mathscr{I}^{n_1, n_2, n_3}} |I| \lesssim 2^{n_2} |G|, 2^{n_3}|H| $ and interpolating these three inequalities we get 
$$ \sum_{I \in \mathscr{I}^{n_1, n_2, n_3}} |I| \lesssim (2^{n_1} |F|)^{\gamma_1} (2^{n_2} |G|)^{\gamma_2} (2^{n_3} |H|)^{\gamma_3},   $$
where $0 \leq \gamma_j \leq 1$ and $\gamma_1+\gamma_2+\gamma_3=1$. Finally, 
\begin{align*}
&|\sum_{n_1, n_2, n_3} \sum_{I \in \mathscr{I}^{n_1, n_2, n_3}} \Lambda_{\vBHT; \rr{P}(I)}(f, g, h)| \\
&\lesssim \sum_{n_1, n_2, n_3} 2^{-n_1/p} 2^{-n_2/q} 2^{-n_3 \frac{1 + \theta_3}{2}}  (2^{n_1} |F|)^{\gamma_1} (2^{n_2} |G|)^{\gamma_2} (2^{n_3} |H|)^{\gamma_3}  \\
&\lesssim \sum_{n_1, n_2, n_3} 2^{-n_1(1/p -\gamma_1)} 2^{-n_2(1/q -\gamma_2)} 2^{-n_3(\frac{1+\theta_3}{2}-\gamma_3)} |F|^{\gamma_1} |G|^{\gamma_2}
\end{align*}
The above series converges if we can pick $\gamma_j$ so that
\[
\frac{1}{p}> \gamma_1, \quad \frac{1}{q}>\gamma_2 \quad \text{and  } \frac{1+ \theta_3}{2}>\gamma_3.
\]
This will be possible as long as
\begin{equation}
\label{conditions r_1, r_2, r'}
\frac{1}{p}+\frac{1}{q}+\frac{1+\theta_3}{2}>1.
\end{equation}
If the above conditions are satisfied, we get generalized restricted type estimates
\[
| \Lambda_{\vBHT}(f, g, h)| \lesssim |F|^{1/p} |G|^{1/q}.
\]

There are four distinct cases:
\begin{enumerate}
\item[i)] $\ds \frac{1}{r_1}, \frac{1}{r_2}, \frac{1}{r'} \leq \frac{1}{2}$. In this case, if we pick $\theta_1=\theta_2 \sim 0$ and $\theta_3 \sim 1$, all the conditions hold and the range of $L^p$ estimates for $\vBHT_{\vec{r}}$ is going to be the convex hull of the points 
\[
(0, 0, 1), (1, 0, 0), \left( 1, \frac{1}{2}, -\frac{1}{2} \right),  \left( \frac{1}{2}, 1, -\frac{1}{2} \right), (0, 1, 0).
\]
That is, we get the same range as that of the $BHT$ operator: $p, q > 1$, $s> \dfrac{2}{3}$ and  $\ds \frac{1}{p}+\frac{1}{q}=\frac{1}{s}$.

\item[ii)] $\ds \frac{1}{r_2}, \frac{1}{r'} \leq \frac{1}{2}$ and $\ds \frac{1}{r_1} > \frac{1}{2}$.
For the condition $\ds \frac{1+\theta_1}{2}-\frac{1}{r_1} >0$ to hold, we have to choose $\theta_1 > \frac{2}{r_1}-1$ and this will imply that the range of the operator, described as a region in the hyperplane $\beta_1+\beta_2+\beta_3=1$,  is the convex hull of the points
\[
\left(0, 0, 1 \right), \left(1, 0, 0\right), \left(1, \frac{1}{2}, - \frac{1}{2} \right), \left( \frac{1}{r_1}, \frac{3}{2}-\frac{1}{r_1}, -\frac{1}{2}   \right), \left( 0, \frac{3}{2}-\frac{1}{r_1}, \frac{1}{r_1}-\frac{1}{2}  \right).
\]
\item[iii)] $\ds \frac{1}{r_1}, \frac{1}{r'} \leq \frac{1}{2}$, and $\ds \frac{1}{r_2} > \frac{1}{2}$. Similarly to the previous case, the range of the operator is the convex hull of 
\[
\left( 0, 0, 1  \right), \left( 0, 1, 0  \right), \left( 1, \frac{1}{2}, -\frac{1}{2}  \right), \left( \frac{3}{2}-\frac{1}{r_2}, \frac{1}{r_2}, -\frac{1}{2}  \right), \left( \frac{3}{2}-\frac{1}{r_2}, 0, \frac{1}{r_2}-\frac{1}{2}  \right). 
\]

\item[iv)] $\ds \frac{1}{r_1}, \frac{1}{r_2}\leq \frac{1}{2}$, $\ds \frac{1}{r'} > \frac{1}{2}$. The range is the convex hull of
\[
 (0, 0, 1), \left( \frac{1}{2}+\frac{1}{r}, 0, \frac{1}{2}-\frac{1}{r} \right), \left( \frac{1}{2}+\frac{1}{r}, \frac{1}{2}, -\frac{1}{r} \right), \left( \frac{1}{2}, \frac{1}{2}+\frac{1}{r}, -\frac{1}{r} \right), \left( 0, \frac{1}{2}+\frac{1}{r}, \frac{1}{2}-\frac{1}{r} \right).
\]

\end{enumerate}

\subsection{The cases $r=1$ or $r_i= \infty$}
\label{the case r=1}~\\
The proof is similar to the one in the previous Section  \ref{local L1}. We first consider the case $r=1$. Because the dual space of $L^1(\ii{W}, \mu)$ is $L^\infty(\ii{W}, \mu)$, the functions appearing in the trilinear form satisfy
\[
\|f(x, \cdot)\|_{L^{r_1}(\ii{W}, \mu)}\leq \one_F(x), \quad \|g(x, \cdot)\|_{L^{r_2}(\ii{W}, \mu)}\leq \one_G(x), \quad \|h(x, \cdot)\|_{L^{\infty}(\ii{W}, \mu)}\leq \one_{H'}.
\]
All the details are identical to the case $r>1$; the restrictions are given by only two inequalities:
\[
\frac{1+\theta_1}{2} > \frac{1}{r_1}, \quad \frac{1+\theta_2}{2} > \frac{1}{r_2}.
\]
In the case $r_1=r_2=2$, $r=1$, these are automatically satisfied and $\ds \ii{D}_{r_1, r_2, r}=Range(BHT)$.

When $r_1=\infty$, we use the fact that the adjoint $BHT^{\ast, 1}$ of $BHT$ is a bilinear operator of the same kind, which is bounded from $L^r \times L^{r'} \to L^1$; more precisely,
\[
\Lambda_{BHT}(f_w, g_w, h_w)=\int_{\rr{R}} BHT(f_w, g_w)(x) \cdot h_w(x) dx= \int_{\rr{R}} f_w(x) \cdot BHT^{*,1}(g_w, h_w)(x)dx.
\]

In proving the boundedness of vector-valued $BHT$ via interpolation, we assume 
\[\ds \|f(x, \cdot)\|_{L^{\infty}(\ii{W}, \mu)}\leq \one_F(x), \quad \|g(x, \cdot)\|_{L^{r}(\ii{W}, \mu)}\leq \one_G(x), \quad \|h(x, \cdot)\|_{L^{r'}(\ii{W}, \mu)}\leq \one_{H'}. \text{ Then}
\]
\begin{align*}
&\vert  \Lambda_{BHT; \rr{P}(I)}(f_w, g_w,  h_w) \vert \leq \| BHT_{\rr{P}(I)}^{*, 1}(g_w \cdot \one_G, h_w \cdot \one_{H'}) \cdot \one_F    \|_1 \\
& \lesssim \left( \sssize_{\rr{P}(I)} \one_F   \right)^{\frac{1+\theta_1}{2}-\epsilon} \left( \sssize_{\rr{P}(I)} \one_G   \right)^{\frac{1+\theta_2}{2}-\frac{1}{r}-\epsilon} \left( \sssize_{\rr{P}(I)} \one_{H'}   \right)^{\frac{1+\theta_3}{2}-\frac{1}{r'}-\epsilon} \cdot  \| g_w \cdot \ci_I   \|_{r} \| h_w \cdot \ci_I   \|_{r'}.
\end{align*}

The rest follows as before. Note that in the case $(\infty, 2, 2)$ we have no constraints on $p, q$, and $s$ except those coming from the original $BHT$ operator itself: indeed, for $\theta_2, \theta_3 >0$, we have 
\[
\frac{1+\theta_2}{2}-\frac{1}{2} > 0, \quad \frac{1+\theta_3}{2}-\frac{1}{2}>0.
\]

\subsection{Iterated $L^p(\ii{W}, \mu)$ spaces estimates for $BHT$}~\\
\label{iterated vv}

Previously, we proved that for any tuple $ (r_1, r_2, r)$ with $\ds \frac{1}{r_1}+\frac{1}{r_2}=\frac{1}{r}$,  $ 1 \leq r < \infty$, and $1 < r_1, r_2 \leq \infty$, we have
\[
BHT: L^p(\rr{R}; L^{r_1}(\ii{W}, \mu)) \times L^q(\rr{R}; L^{r_2}(\ii{W}, \mu)) \to L^s(\rr{R}; L^{r}(\ii{W}, \mu))
\]
whenever $p, q, r$ are in a certain range $\mathscr{D}_{r, r_1, r_2}$ which can be described in a precise manner.  The general ideas for proving multiple vector-valued estimates for $BHT$ (as presented in Theorem \ref{multiple vector valued BHT}) via the helicoidal method were described in the introduction. In this section, we present in more detail the proof in the case of two iterated spaces $\ell^s(\ell^r)$ in order to simplify the notation. First, we prove the following localized vector-valued result:

\begin{proposition}
\label{local vv-BHT}
\begin{align*}
&\| \left( \sum_{k=1}^N \vert BHT_{\rr{P}(I_0)}(f_k \cdot \one_F, g_k \cdot \one_G) \vert ^r   \right)^{1/r} \cdot \one_{H'}\|_s  \\
& \leq \tilde{C}\| \left( \sum_{k=1}^N \vert f_k \vert ^{r_1}  \right)^{1/r_1} \cdot \ci_{I_0} \|_{p}  \| \left( \sum_{k=1}^N \vert g_k \vert ^{r_2}  \right)^{1/{r_2}} \cdot \ci_{I_0} \|_{q}.
\end{align*}
where $\ds \tilde{C}= \left( \sssize_{\rr{P}(I_0)} \one_F  \right)^{\frac{1+\theta_1}{2}-\frac{1}{p}-\epsilon} \cdot \left( \sssize_{\rr{P}(I_0)} \one_G  \right)^{\frac{1+\theta_2}{2}-\frac{1}{q}-\epsilon} \cdot \left( \sssize_{\rr{P}(I_0)} \one_{H'}  \right)^{\frac{1+\theta_3}{2}-\frac{1}{s'}-\epsilon} $. 
\begin{proof}
This is going to be a refinement of the proof of Theorem \ref{vector valued BHT} from the previous section. In constructing the collection of intervals $\mathscr{I}_j^{n_j}$, we note that we only need to select intervals $I$ that are already contained in $I_0$, because all the tiles in $\rr{P}(I_0)$ are so that $I_P \subseteq I_0$.

As before, we prove generalized restricted type estimates, and we assume that the functions have the following properties:
\[
\left( \sum_k |f_k|^{r_1}\right)^{1/{r_1}} \leq \one_{E_1}, \quad \left( \sum_k |f_k|^{r_2}\right)^{1/{r_2}} \leq \one_{E_2}, \quad \left( \sum_k |h_k|^{r'}\right)^{1/{r'}} \leq \one_{E_3'}.
\]
The exceptional set is defined by $\ds \tilde{\Omega}=\lbrace  \mathcal{M} \one_{E_1} > C \frac{|E_1|}{|E_3|}  \rbrace \cup \lbrace \mathcal{M} \one_{E_2} > C \frac{|E_2|}{|E_3|}   \rbrace$, and we assume the tiles to be so that $\ds 1+ \frac{\dist (I_P, \tilde{\Omega}^c)}{|I_P|} \sim 2^d$.

For intervals $I \in \mathscr{I}_1^{n_1}$ we have 
\[
\frac{1}{|I|} \int_{\rr{R}} \one_{E_1} \cdot \one_F \cdot \ci_I^M dx \sim \sssize_{\rr{P}_{n_1}(I)} (\one_{E_1} \cdot \one_F) \sim 2^{-n_1} \leq 2^d \frac{|E_1|}{|E_3|}.
\]
When we consider intervals $I \in \mathscr{I}_1^{n_1} \cap \mathscr{I}_2^{n_2}\cap \mathscr{I}_3^{n_3}$, the above approximations become inequalities. We also need to point out that
\[
\sssize_{\rr{P}(I)} (\one_{E_1} \cdot \one_F) \quad \text{and } \quad \frac{1}{|I|} \int_{\rr{R}} \one_{E_1} \cdot \one_F \cdot \ci_I^M dx  \leq \sssize_{\rr{P}(I_0)} (\one_{E_1} \cdot \one_F).
\]

Now we add the trilinear forms in order to obtain generalized restricted type estimates:
\begin{align*}
&\sum_k \vert \Lambda_{BHT; \rr{P}(I_0)} (f_k \cdot \one_F, g_k \cdot \one_G, h_k \cdot \one_{H'} )  \vert \\
& \leq \sum_{n_1, n_2, n_3} \sum_{I \in \mathscr{I}^{n_1, n_2, n_3}} \sum_k \vert \Lambda_{BHT; \rr{P}(I_0 \cap I)} (f_k \cdot \one_F, g_k \cdot \one_G, h_k \cdot \one_{H'} )   \vert \\
&\lesssim \sum_{n_1, n_2, n_3} \sum_{I \in \mathscr{I}^{n_1, n_2, n_3}} \left( \sssize_{\rr{P}(I)} \one_{E_1} \one_F  \right)^{\frac{1+\theta_1}{2}-\frac{1}{r_1}-\epsilon} \cdot \left( \sssize_{\rr{P}(I)} \one_{E_2} \one_G  \right)^{\frac{1+\theta_2}{2}-\frac{1}{r_2}-\epsilon} \cdot \left( \sssize_{\rr{P}(I)} \one_{E_3'} \one_{H'} \right)^{\frac{1+\theta_3}{2}-\frac{1}{r'}-\epsilon} \\
& \quad \quad \quad \cdot \frac{\| \one_{E_1} \cdot \one_{F} \cdot \ci_{I}   \|_{r_1}}{|I|^{1/{r_1}}} \cdot \frac{\| \one_{E_2} \cdot \one_{G} \cdot \ci_{I}   \|_{r_2}}{|I|^{1/{r_2}}}  \cdot \frac{\| \one_{E_3'} \cdot \one_{H'} \cdot \ci_{I}   \|_{r'}}{|I|^{1/{r'}}} |I| 
\end{align*}
Using the \emph{modified} sizes from Definition \ref{def:modified-size} imply that 
\begin{align*}
&\sum_k \vert \Lambda_{BHT; \rr{P}(I_0)} (f_k \cdot \one_F, g_k \cdot \one_G, h_k \cdot \one_{H'} )  \vert \\
&\lesssim \left( \sssize_{\rr{P}(I_0)} \one_{E_1} \cdot \one_F  \right)^{\frac{1+\theta_1}{2}-\frac{1}{p}-\epsilon} \cdot \left( \sssize_{\rr{P}(I_0)} \one_{E_2} \one_G  \right)^{\frac{1+\theta_2}{2}-\frac{1}{q}-\epsilon} \cdot \left( \sssize_{\rr{P}(I_0)} \one_{E_3'} \one_{H'} \right)^{\frac{1+\theta_3}{2}-\frac{1}{s'} -\epsilon} \\
& \quad  \quad \quad \cdot \sum_{n_1, n_2, n_3} \sum_{I \in \mathscr{I}^{n_1, n_2, n_3}}  2^{- \frac{n_1}{p}} 2^{-\frac{n_2}{q}} 2^{-n_3(\frac{1}{s'}+\epsilon)} |I|
\end{align*}
The last part adds up to something $\ds \lesssim  2^{- \tilde{M}d}|E_1|^{\frac{1}{p}} |E_2|^{\frac{1}{q}} |E_3|^{\frac{1}{s'}}$, which is precisely what we were aiming in the beginning.

The cases when one of the $r_1, r_2$ or $r' = \infty$ follow in a similar manner.  
\end{proof}
\end{proposition}

The above proposition is an intermediate step in the proof of $L^p$ estimates for $\vBHT_{\vec{R}}$, in the case of two iterated vector spaces, which is presented below.

\begin{proposition}
\label{two iterated vv-BHT}
\begin{align*}
& \big \|  \left( \sum_{l}\left( \sum_{k} \big \vert BHT(f_{kl}, g_{kl}) \big \vert^{r}    \right)^{s/r}\right)^{1/s}   \big \|_t \\
&\leq C \big \|  \left( \sum_{l}\left( \sum_{k} \big \vert f_{kl}\big \vert^{r_1}    \right)^{{s_{1}}/{r_{1}}}  \right)^{1/{s_1}}   \big \|_p  \big \|  \left( \sum_{l}\left( \sum_{k} \big \vert g_{kl}\big \vert^{r_r}    \right)^{{s_{2}}/{r_{2}}}  \right)^{1/{s_2}}   \big \|_q
\end{align*}
\begin{proof}
Once again, we use generalized restricted type interpolation; $F, G, H$ are sets of finite measure, with $|H|=1$. The exceptional set is defined as usual, and $H'=H \setminus \Omega$. The sequences of functions will be so that 
\[
\left(\sum_l \left( \sum_k |f_{kl}|^{r_1}\right)^{\frac{s_1}{r_1}}\right)^{\frac{1}{s_1}} \leq \one_F, \left(\sum_l \left( \sum_k |g_{kl}|^{r_2}\right)^{\frac{s_2}{r_2}}\right)^{\frac{1}{s_2}} \leq \one_G, \left(\sum_l \left( \sum_k |h_{kl}|^{r'}\right)^{\frac{s'}{r'}}   \right)^{\frac{1}{s'}} \leq \one_{H'}.
\] 

The collections $\mathscr{I}_j^{n_j}$ are going to be chosen in the same way as in the proof of Theorem \ref{vector valued BHT}, depending on the sizes and averages of the characteristic functions $\one_F, \one_G, \one_{H'}$. Proposition \ref{local vv-BHT} yields the following:
{
\fontsize{10}{10}
\begin{align}
&\sum_k \vert \Lambda_{BHT; \rr{P}(I)}(f_{kl}, g_{kl}, h_{kl})  \vert  \lesssim \left( \sssize_{\rr{P}(I)} \one_F  \right)^{\frac{1+\theta_1}{2}-\frac{1}{s_1}-\epsilon} \cdot \left( \sssize_{\rr{P}(I)} \one_G  \right)^{\frac{1+\theta_2}{2}-\frac{1}{s_2}-\epsilon} \cdot \left( \sssize_{\rr{P}(I)} \one_{H'} \right)^{\frac{1+\theta_3}{2}-\frac{1}{s'}-\epsilon} \\
& \cdot \|  \left( \sum_k |f_{kl}|^{r_1}   \right)^{\frac{1}{r_1}}  \cdot \ci_I\|_{s_1}  \cdot \|  \left( \sum_k |g_{kl}|^{r_2}   \right)^{\frac{1}{r_2}}  \cdot \ci_I\|_{s_2} \cdot  \|  \left( \sum_k |h_{kl}|^{r'}   \right)^{\frac{1}{r'}}  \cdot \ci_I\|_{s'} \label{will use holder}
\end{align}
}
Then we sum \eqref{will use holder} over $l$ as well, and apply H\"{o}lder for the triple $s_1, s_2, s'$. In this way, we recover $\| \one_{F} \cdot \ci_I  \|_{s_1}$, and the corresponding quantities for the second and third entries. We have
\begin{align*}
&\vert \sum_{k,l} \Lambda_{BHT}(f_{kl}, g_{kl}, h_{kl})  \vert \\
&\lesssim \sum_{n_1, n_2, n_3} \sum_{\mathscr{I}^{n_1, n_2, n_3}} \left( \sssize_{\rr{P}(I)} \one_F  \right)^{\frac{1+\theta_1}{2}-\frac{1}{s_1}-\epsilon} \cdot \left( \sssize_{\rr{P}(I)} \one_G  \right)^{\frac{1+\theta_2}{2}-\frac{1}{s_2}-\epsilon} \cdot \left( \sssize_{\rr{P}(I)} \one_{H'} \right)^{\frac{1+\theta_3}{2}-\frac{1}{s'}-\epsilon} \\
&\quad \quad \quad  \cdot \frac{\| \one_F \cdot \ci_I    \|_{s_1}}{|I|^{1/{s_1}}} \cdot \frac{\| \one_G \cdot \ci_I    \|_{s_2}}{|I|^{1/{s_2}}} \cdot \frac{\| \one_{H'} \cdot \ci_I    \|_{s'}}{|I|^{1/{s'}}} \cdot |I| \\
& \lesssim \sum_{n_1, n_2, n_3} \sum_{\mathscr{I}^{n_1, n_2, n_3}} \left( ``\sssize"_{\rr{P}(I)} \one_F  \right)^{\frac{1+\theta_1}{2}-\frac{1}{p}-\epsilon} \cdot \left( ``\sssize"_{\rr{P}(I)} \one_G  \right)^{\frac{1+\theta_2}{2}-\frac{1}{q}-\epsilon} \cdot \left( ``\sssize"_{\rr{P}(I)} \one_{H'} \right)^{\frac{1+\theta_3}{2}-\frac{1}{t'} -\epsilon} \\
&\quad \quad \quad  \cdot  2^{- \frac{n_1}{p}} 2^{-\frac{n_2}{q}} 2^{-n_3 (\frac{-1}{t'}+ \epsilon) }  \cdot |I|.
\end{align*}

\begin{remark}
The ``sizes" appearing in the line above are not exactly the ones from Definition \ref{def size}, but the modified ones from Definition \ref{def:modified-size} . Note that
\[
\max \left( \sssize_{\rr{P}(I)} \one_F,  \frac{1}{|I|} \int_{\rr{R}} \one_{F} \cdot \ci_I^M dx\right) \leq ``\sssize"_{\rr{P}(I)} \one_F.
\]
This is the step where we can prove also the localized version of the statement in Proposition \ref{two iterated vv-BHT}. Assuming all the tiles are sitting above an interval $I_0$, we can obtain the same result with operatorial norm
\[
\left( \sssize_{\rr{P}(I_0)} \one_F  \right)^{\frac{1+\theta_1}{2}-\frac{1}{p}-\epsilon} \cdot \left( \sssize_{\rr{P}(I_0)} \one_G  \right)^{\frac{1+\theta_2}{2}-\frac{1}{q}-\epsilon} \cdot \left( \sssize_{\rr{P}(I_0)} \one_{H'} \right)^{\frac{1+\theta_3}{2}-\frac{1}{t'}-\epsilon}.
\]
\end{remark}
The rest of the proof is identical to the simpler vector case of Theorem \ref{vector valued BHT}; the quantities on $LHS$ add up to $|F|^{1/p} |G|^{1/q}$, provided 
\[
\frac{1+\theta_1}{2}>\frac{1}{p}, \quad \frac{1+\theta_2}{2}>\frac{1}{q}, \quad \frac{1+\theta_3}{2}>\frac{1}{s'}.
\] 
\end{proof}
\end{proposition}

\section{Similar Results for Paraproducts : proof of Theorem \ref{multiple vector valued paraproducts}}
\label{paraproduct results}
The paraproduct case is similar to $BHT$, even though the bilinear Hilbert transform is a much more complicated object. The extra difficulties are hidden in Propositions \ref{BHT trilinear estimate}, but we will see from the proof of the vector-valued extensions that the complexity of the paraproduct case is comparable to the ``local $L^2$" case for  $BHT$. In both situations, we recover the maximal range for vector-valued estimates.

We will be working with the \emph{discrete paraproduct} of the functions $f$ and $g$, which is defined by
\[
\Pi(f, g)(x)=\sum_{I \in \ii{I}} \frac{1}{|I|^{1/2}} \langle f, \phi_I^1  \rangle \langle g, \phi_I^2 \rangle \phi_I^3(x).
\]
Here $\ii{I}$ is a family of dyadic intervals, and the wave packets $\lbrace \phi_I^j\rbrace_{I \in \ii{I}}$ are so that two of the families are \emph{lacunary} ($\phi_I^j$ is a wave packet on $\ I \times [\frac{1}{|I|}, \frac{2}{|I|}]$ ), and the third one is \emph{non-lacunary} ($\phi_I^{j_0}$ is a wave packet on $\ I \times [0,\frac{1}{|I|}]$ ). Again, we present the case of $\ell^p$ spaces, for simplicity. The operator we are interested in is
\[
\vec{\Pi}_{r}(f, g):=\left( \sum_{k=1}^N |\Pi(f_k, g_k)|^r   \right)^{1/r}.
\]
\begin{remark}
We could alternatively look at operators of the form
\[
(f, g) \mapsto \left( \sum_{k=1}^N |\Pi_k(f_k, g_k)|^r   \right)^{1/r},
\]
where each paraproduct $\Pi_k$ is associated to a family $\ii{I}_k$ of dyadic intervals. The $\Pi_k$s don't need to be precisely the same, but they display a similar behavior. Similarly, for $\vBHT$ we could have a ``perturbation" $BHT_w$ for each $w \in \ii{W}$, and the method of the proof applies in that case as well.
\end{remark}

\subsection{A few results about Paraproducts}
The concepts of \emph{sizes} and \emph{energies} are similar to the corresponding ones for the bilinear Hilbert transform; we don't need to organize the tiles into trees because the family of tiles is of rank $0$. We recall some definitions bellow.
\begin{definition}
Let $\mathscr{I}$ be a family of dyadic intervals. For any $1 \leq j \leq 3$, we define
\[
\ssize_{\mathscr{I}}\left( \langle f, \phi_I^j  \rangle_{I \in \mathscr{I}}   \right)=\sup_{I \in \ii{I}} \frac{|\langle f, \phi_I^j \rangle|}{|I|^{1/2}}, \quad \text{ if $(\phi_I^j)_I$ is non-lacunary and}
\]
\[
\ssize_{\mathscr{I}}\left( \langle f, \phi_I^j  \rangle_{I \in \mathscr{I}}   \right)=\sup_{I_0 \in \ii{I}}  \frac{1}{|I_0|^{1/2}} \| ( \sum_{\substack{I \subseteq I_0\\ I \in \ii{I}}} \frac{|\langle f, \phi_I^j \rangle|^2}{|I|}  \cdot \one_I )^{1/2}    \|_{1, \infty}, \quad \text{ if $(\phi_I^j)_I$ is lacunary.}
\]
Similarly to the $BHT$ case, energy is defined as
\[
\eenergy^j_{\ii{I}} \left( \langle f, \phi_I^j   \rangle_{I \in \ii{I}}   \right):=\sup_{n \in \rr{Z}} 2^n \sup_{\rr{D}} (\sum_{I \in \rr{D}}|I|)
\]
where $\rr{D}$ ranges over all collections of disjoint intervals $I_0$ with the property that 
\begin{align*}
& \frac{|\langle f, \phi_{I_0}^j \rangle|}{|I_0|^{1/2}} \geq 2^n, \quad \text{ if $(\phi_{I}^j)_I$ is non-lacunary and respectively} \\
&\frac{1}{|I_0|^{1/2}} \| ( \sum_{\substack{I \subseteq I_0\\ I \in \ii{I}}} \frac{|\langle f, \phi_I^j \rangle|^2}{|I|}  \cdot \one_I )^{1/2}    \|_{1, \infty} \geq 2^n, \quad \text{ if $(\phi_I^j)_I$ is lacunary.}
\end{align*}

\end{definition}

We have estimates similar to Lemma \ref{BHT size estimate} and Lemma \ref{BHT energy estimate}. However, because we don't need to use orthogonality of trees, the energy becomes an $L^1$ quantity.

\begin{lemma}[Lemma 2.13 of \cite{multilinear_harmonic}] If $F$ is an $L^1$ function and $1 \leq j \leq 3$, then
\[
\ssize_{\ii{I}}^j(\langle F, \phi_I^j  \rangle_{I \in \ii{I}}) \lesssim \sup_{I \in \ii{I}}\frac{1}{|I|} \int_{\rr{R}} |F| \ci_{I}^M dx
\]
for $M>0$, with implicit constants depending on $M$.
\end{lemma}

\begin{lemma}[Lemma 2.14 of \cite{multilinear_harmonic}]
If $F$ is an $L^1$ function and $1 \leq j \leq 3$, then 
$$\ds \eenergy_{\ii{I}}^j(\langle F, \phi_I^j  \rangle_{I \in \ii{I}}) \lesssim \| F \|_1.$$
\end{lemma}

\begin{proposition}[Proposition 2.12 of \cite{multilinear_harmonic}]
\label{paraproduct estimates}
Given a paraproduct $\Pi$ associated with a family $\ii{I}$ of intervals,
\begin{align*}
& \lft\Lambda_{\Pi}(f_1, f_2, f_3) |=| \sum_{I \in \ii{I}} \frac{1}{|I|^{1/2}} \langle f_1, \phi_I^1 \rangle \langle f_2, \phi_I^2 \rangle \langle f_3, \phi_I^3 \rangle  \rg\\
&\lesssim \prod_{j=1}^3 \left( \ssize_{\ii{I}}^{(j)} ( \langle f_j, \phi_I^j  \rangle_{I \in \ii{I}} )  \right)^{1-{\theta_j}} \left( \eenergy_{\ii{I}}^{(j)} (\langle f_j, \phi_I^j  \rangle_{I \in \ii{I}}   )  \right)^{\theta_j},
\end{align*}
for any $0 \leq \theta_1, \theta_2, \theta_3 <1$ such that $\theta_1+\theta_2+\theta_3=1$, where the implicit constant depends on $\theta_1, \theta_2, \theta_3$ only.
\end{proposition}

While the above proposition is the main ingredient, we need ``localized" estimates. If $I_0$ is some fixed dyadic interval, then we define
\[
\Pi(I_0)(f, g)(x)=\sum_{\substack{I \in \ii{I} \\ I \subseteq I_0}} \frac{1}{|I|^{1/2}} \langle f, \phi_I^1 \rangle \langle g, \phi_I^2 \rangle \phi_I^3(x). 
\]
Here again we need some localization results which play the role of Proposition \ref{Localization for local L^1 BHT} and Corollary \ref{Localization for BHT with target L^1 } from the $BHT$ case. 

The trilinear form associated to the localized paraproduct is given by
\[
\Lambda_{\Pi(I_0)}^{F, G, H'}(f, g, h):=\Lambda_{\Pi(I_0)}(f \cdot \one_F, g \cdot \one_G, h \cdot \one_{H'}).
\]

\begin{proposition}
\label{Localization for local L^1 paraproducts}
Let $I_0$ be a fixed dyadic interval and $F, G, H' \subset \rr{R}$ sets of finite measure. Then there exist some positive numbers $0 \leq a_1, a_2, a_3 < 1 $ so that 
\[
|\Lambda_{\Pi(I_0)}^{F, G, H'}(f, g, h)| \lesssim \sssize_{\ii{I}(I_0)}(\one_{F})^{a_1} \cdot \sssize_{\ii{I}(I_0)}(\one_{G})^{a_2} \cdot  \sssize_{\ii{I}(I_0)}(\one_{H'})^{a_3} \cdot \| f \cdot \ci_{I_0}\|_{r_1}   \| g \cdot \ci_{I_0} \|_{r_2} \| h \cdot \ci_{I_0} \|_{r'} 
\] 
whenever $\ds \frac{1}{r_1}+\frac{1}{r_2}+\frac{1}{r'}=1$, and $ 1< r_1, r_2, r' < \infty$. Here $\ds a_j =1-\frac{1}{r_j}-\epsilon$.
\begin{proof}
The idea of the proof is very similar to that of Proposition \ref{Localization for local L^2 BHT}. Restricted type estimates are proved by performing a triple stopping time and then the result follows by interpolation. We leave the routine details to the reader.
\end{proof}
\end{proposition}

The case $r=1$ is obtained through interpolation of restricted type estimates only. This comes in contrast with the $r=1$ case for $BHT$, where generalized restricted type interpolation is necessary. More exactly, for the $BHT$ operator, in order to conclude estimates for $\ds (\frac{1}{r_1}, \frac{1}{r_2}, 0)$, one needs to interpolate between good ($\beta_i>0$) and bad ($\beta_3 <0$) tuples $\beta=(\beta_1, \beta_2, \beta_3)$.

\begin{proposition}
\label{Localization for paraproducts with target L^1}
If $H'$ is a fixed set of finite measure, 
\begin{equation}
\label{paraproduct multilinear form when target space is L^1}
\big \vert \Lambda_{\Pi(I_0)}(f, g, \one_{H'}) \big \vert  \lesssim \sssize_{\ii{I}(I_0)}(\one_{H'}) \| f \cdot \ci_{I_0}\|_p   \| g \cdot \ci_{I_0} \|_q,
\end{equation}
whenever $\ds \frac{1}{p}+\frac{1}{q}=1$, and  $1< p, q < \infty$.
\begin{proof}
In this case $\Lambda_{\Pi(I_0)}(f, g, \one_{H'})$ becomes a bilinear form with respect to the first two entries. Because of the decay of $\ci_{I_0}$, it will be sufficient to prove the proposition in the case $\text{supp }f, g \subseteq 5 I_0$. By Theorem \ref{interp thm restricted type}, it will be enough to show restricted type estimates for the bilinear form
\[
(f,g) \mapsto \Lambda_{\Pi(I_0)}(f, g, \one_{H'}).
\]
Let $F$ and $G$ be sets of finite measure and $|f| \leq \one_F $ and $|g| \leq \one_G$. Using Proposition \ref{paraproduct estimates} with $\theta_3=0$ and estimating $\ds \sssize_{\ii{I}(I_0)}(f) \lesssim 1, \sssize_{\ii{I}(I_0)}(g) \lesssim 1$, we get
\[
|\Lambda_{\Pi(I_0)}(f, g, \one_{H'})| \lesssim \sssize_{\ii{I}(I_0)} (\one_{H'}) |F|^{\theta_1} |G|^{\theta_2},
\] 
where $\theta_1+\theta_2=1$ and $0< \theta_1, \theta_2<1.$ This proves restricted type estimates in a small neighborhood of $\ds \left(  \frac{1}{p}, \frac{1}{q}\right)$.
\end{proof}
\end{proposition}

\subsection{Proof of Theorem \ref{multiple vector valued BHT}: a particular case }~\\

We will be using vector-valued interpolation theorems, as usual. Hence, we fix $F, G$ and $H$ sets of finite measure and we assume $|H|=1$. Let $f=\lbrace f_k\rbrace_k, g=\lbrace g_k\rbrace_k$, with $ (\sum_k |f_k|^{r_1})^{1/{r_1}} \leq \one_F$, $(\sum_k |g_k|^{r_2})^{1/{r_2}} \leq \one_G$.

The exceptional set will be
\[
\tilde{\Omega}:=\lbrace x: \mathcal{M}(\one_F)(x) > C |F|  \rbrace \cup \lbrace x: \mathcal{M}(\one_G)(x) > C |G|    \rbrace
\]
and $H'=H \setminus \tilde{\Omega}$. We have a sequence of functions $\lbrace h_k \rbrace_k$ with $\ds (\sum_k |h_k|^{r'})^{1/{r'}} \leq \one_{H'}$.

 For every $d \geq 0$
\[
\mathscr{I}^d:=\lbrace I \in \ii{I}: 1 +\frac{\dist(I, \Omega^c)}{|I|} \sim 2^d \rbrace
\]
When estimating paraproducts associated to the collection $\ii{I}^d$, we get an extra $2^{-10d}$ decay and thus the $d$-dependency of the paraproducts can be assumed to be implicit. As before, for each of the sets $F, G$ and $H'$ we define collections of disjoint maximal intervals $\mathcal{J}_1^{n_1}, \mathcal{J}_2^{n_2} $ and $\mathcal{J}_3^{n_3}$ respectively. For example, if $I \in \mathcal{J}_1^{n_1}$, then
\[
2^{-n_1-1} \leq \frac{1}{|I|} \int_{\rr{R}} \one_F \cdot \ci_{I} dx \leq 2^{-n_1} \lesssim |F|.
\]
Returning to the operator $\vec{\Pi}_{r}$, we have for the associated multilinear form
\begin{align*}
&|\sum_k \Lambda_{\Pi}(f_k, g_k, h_k)|\leq \sum_{n_1, n_2, n_3} \sum_{I \in \mathcal{J}^{n_1, n_2, n_3}} \sum_k|\Lambda_{\Pi(I_0)}(f_k, g_k, h_k)|.
\end{align*}

Now we use the localization results of Proposition \ref{Localization for local L^1 paraproducts} to estimate the above expression by
\begin{align*}
& \sum_{n_1, n_2, n_3} \sum_{I_0 \in \mathcal{J}^{n_1, n_2, n_3}} \sum_{k=1}^n  \sssize_{\ii{I}(I_0)}(\one_{F})^{b_1} \cdot \sssize_{\ii{I}(I_0)}(\one_{G})^{b_2} \cdot  \sssize_{\ii{I}(I_0)}(\one_{H'})^{b_3} \\
& \quad \quad \cdot \|  f_k \cdot \ci_{I_0}  \|_{r_1} \|  g_k \cdot \ci_{I_0}  \|_{r_2} \|  h_k \cdot \ci_{I_0}  \|_{r'}  \\
&\lesssim \sum_{n_1, n_2, n_3} \sum_{I_0 \in \mathcal{J}^{n_1, n_2, n_3}}  \sssize_{\ii{I}(I_0)}(\one_{F})^{b_1} \cdot \sssize_{\ii{I}(I_0)}(\one_{G})^{b_2} \cdot  \sssize_{\ii{I}(I_0)}(\one_{H'})^{b_3} \\
&\quad \quad  \cdot \frac{\|  \one_F \cdot \ci_{I_0}  \|_{r_1}}{|I_0|^{1/{r_1}}} \frac{\|  \one_G \cdot \ci_{I_0}  \|_{r_2}}{|I_0|^{1/{r_2}}} \frac{\|  \one_{H'} \cdot \ci_{I_0}  \|_{r'}}{|I_0|^{1/{r'}}} |I_0|.
\end{align*}
Here we choose some $0 \leq b_j \leq a_j$, which we can do because the sizes are subunitary.
Whenever $0 \leq\gamma_j \leq 1$ are so that $\gamma_1+\gamma_2+\gamma_3=1$, 
\[
\sum_{I_0 \in \mathcal{J}^{n_1, n_2, n_3}}|I_0| \lesssim (2^{n_1}|F|)^{\gamma_1} (2^{n_2}|G|)^{\gamma_2} (2^{n_3}|H|)^{\gamma_3}. 
\]
Adding all the pieces together we have
\begin{align*}
&|\sum_k \Lambda_{\Pi}(f_k, g_k, h_k)|\lesssim \sum_{n_1, n_2, n_3} 2^{-n_1(b_1+\frac{1}{\tilde{p}} -\gamma_1)} 2^{-n_2(b_2+\frac{1}{\tilde{q}} -\gamma_2)} 2^{-n_3(b_3+\frac{1}{r'} -\gamma_3)}  |F|^{\gamma_1} |G|^{\gamma_2} \lesssim |F|^{\frac{1}{\tilde{p}}} |G|^{\frac{1}{\tilde{q}}}.
\end{align*}
Of course, the last inequality is true provided we can choose $\gamma_1, \gamma_2, \gamma_3$  so that the series converges.
Choosing the $\theta_j$s and $\alpha_j$s carefully, one can prove that the restricted weak type estimates hold arbitrary close to the points
\[
(0, 0, 1), (1, 0, 0), (0, 1, 0), \text{   and  } (1, 1, -1).
\]
Then the general result follows by interpolation.
\begin{remark}
With a few adjustments, the proof is valid in the case $r=1$ as well.
\end{remark}
\bigskip

\section{Tensor products $BHT \otimes \Pi^{\otimes^n}$}
\label{tensor products}
In this section, we will prove the boundedness of the tensor product 
\[
BHT \otimes \Pi^{\otimes n}=BHT \otimes \Pi \otimes \ldots \otimes \Pi : L^p(\rr{R}^{n+1})\times L^q(\rr{R}^{n+1}) \to L^r(\rr{R}^{n+1})
\]
whenever $\ds \frac{1}{r}=\frac{1}{p}+\frac{1}{q}$, with $\frac{2}{3}<r< \infty$, $1\leq p, q < \infty$. 

If $T_1: L^p(\rr{R}^{n_1})\times L^q(\rr{R}^{n_1}) \to L^r(\rr{R}^{n_1})$ and $T_2: L^p(\rr{R}^{n_2})\times L^q(\rr{R}^{n_2}) \to L^r(\rr{R}^{n_2})$ are two bilinear operators, then the tensor product
\[
T_1\otimes T_2 : L^p(\rr{R}^{n_1+n_2})\times L^q(\rr{R}^{n_1+n_2}) \to L^r(\rr{R}^{n_1+n_2})
\]
will act as $T_1$ in the first variable and as $T_2$ in the second variable. In our case, the operators are given by singular multipliers, and in this situation we can give a characterization of the tensor product. Assume
\[
T_1(f, g)(x)=\int_{\rr{R}^{2n_1}} \hat{f}(\xi_1) \hat{g}(\xi_2) m_1(\xi_1, \xi_2) e^{2\pi i x \cdot (\xi_1+\xi_2)} d\xi_1 d\xi_2
\]
and similarly
\[
T_2(f, g)(y)=\int_{\rr{R}^{2n_2}} \hat{f}(\eta_1) \hat{g}(\eta_2) m_2(\eta_1, \eta_2) e^{2\pi i y \cdot (\eta_1+\eta_2)} d\eta_1 d\eta_2.
\]
Then the multiplier of the tensor product is precisely $m_1(\xi_1, \xi_2) \cdot m_2(\eta_1, \eta_2)$:
{
\fontsize{10}{10}
\[
T_1\otimes T_2(f, g)(x,y)=\int \hat{f}(\xi_1, \eta_1) \hat{g}(\xi_2,\eta_2) m_1(\xi_1, \xi_2) m_2(\eta_1, \eta_2) e^{2 \pi i x \cdot(\xi_1+\xi_2)} e^{2\pi i y \cdot (\eta_1+\eta_2)}d \xi_1 d \xi_2 d\eta_1 d\eta_2.
\]
}
The multiplier associated with $BHT$ is $sgn(\xi_1-\xi_2)$, while the multiplier of a paraproduct of two functions on the real line is a classical Marcinkiewicz-Mikhlin-H\"{o}rmander multiplier $m(\xi_1, \xi_2)$, smooth away from the origin, satisfying the condition $\ds |\partial ^\alpha m (\xi)| \lesssim |\xi|^{-|\alpha|}$ for sufficiently many multi-indices $\alpha$. The decay in $m$ and a Fourier series decomposition allows one to approximate the multiplier by a finite number of sums of the form 
\[
\sum_{k} \hat{\varphi}_k(\xi_1) \hat{\psi}_k(\xi_2) \hat{\psi}_k(\xi_1+\xi_2), \quad \sum_{k} \hat{\psi}_k(\xi_1) \hat{\varphi}_k(\xi_2) \hat{\psi}_k(\xi_1+\xi_2) \quad \text{or   } \sum_{k} \hat{\psi}_k(\xi_1) \hat{\psi}_k(\xi_2) \hat{\varphi}_k(\xi_1+\xi_2).
\] 

Recall that $Q_k$ is the Littlewood-Paley projection onto $\lbrace |\xi| \sim 2^k   \rbrace$(which is really the convolution with $\psi_k(\cdot)$), and $P_k$ is the projection onto $\lbrace |\xi| \leq 2^k \rbrace$, corresponding to the convolution with $\varphi_k$.  Then we can regard paraproducts as being expressions of the form
\begin{equation}
\label{eq:def paraprod}
\sum_k Q_k (P_k f \cdot Q_k g )(x,y)  \quad, \sum_k Q_k (Q_k f \cdot P_k g )(x,y)  \quad \text{or  } \sum_k P_k (Q_k f \cdot Q_k g )(x,y).
\end{equation}
It is important in the following proofs that the outer-most function $\ds \hat{\varphi}_k(\xi_1+\xi_2)$ and $\ds \hat{\psi}_k(\xi_1+\xi_2)$ are identically equal to $1$ on the supports of $\hat{\psi}_k(\xi_1)\cdot \hat{\psi}_k(\xi_2)$ and $\hat{\psi}_k(\xi_1) \cdot \hat{\varphi}_k(\xi_2) $ respectively. This can always be achieved with the price of an extra decomposition.

\begin{proposition}
\label{characterization of tensor with paraproduct}
Let $T_m: L^p(\rr{R}^n) \times L^q(\rr{R}^n) \to L^r(\rr{R}^n)$ be a bilinear operator with smooth symbol $m$, and $\Pi: L^p(\rr{R}) \times L^q(\rr{R}) \to L^r(\rr{R})$ a paraproduct as described above.
\begin{enumerate}
\item \label{a type paraproduct}
If $\Pi$ is given by $\sum_k Q_k (P_k f \cdot Q_k g )(x,y)$, then 
\[
(T_m \otimes \Pi )(f,g)(x,y)=\sum_k Q_k^2 \left(T_m( P_k^y f, Q_k^y g)\right)(x)=\sum_k T_m( P_k^y f, Q_k^y g)(x).
\]
\item If $\Pi$ is given by $\sum_k P_k (Q_k f \cdot Q_k g )(x,y)$, then
\[
(T_m \otimes \Pi)(f,g)(x,y)=\sum_k P_k^2 \left(T_m(Q_k^y f, Q_k^y g)\right)(x)=\sum_k  T_m(Q_k^y f, Q_k^y g)(x).
\]
\end{enumerate}
Here we need to explain the notation: $Q_k^2$ denotes the projection onto $|\xi_2| \sim 2^k$ in the second variable, and $P_k^yf $ is a function of $x$ only, with the variable $y$ fixed. The exact formulas are
\[
P_k^yf(x)=\int_{\rr{R}} \varphi_k(s) f(x,y-s) ds, \quad P_k^2f(x,y)=\int_{\rr{R}} \varphi_k(s) f(x,y-s) ds,
\] 
\[
Q_k^yf(x)=\int_{\rr{R}} \psi_k(s) f(x,y-s) ds, \quad Q_k^2f(x,y)=\int_{\rr{R}} \psi_k(s) f(x,y-s) ds.
\]
\begin{proof}
The proof is a series of direct computations, and we only present the case \eqref{a type paraproduct}:
{
\fontsize{10}{10}
\begin{align*}
& (T_m \otimes \Pi)(f, g)(x,y)\\
&= \int_{\rr{R}^{2n+2}} \hat{f}(\xi_1, \eta_1) \hat{g}(\xi_2, \eta_2) m(\xi_1, \xi_2) \left( \sum_k \hat{\varphi}_k(\eta_1) \hat{\psi}_k(\eta_2) \hat{\psi}_k(\eta_1+\eta_2)   \right) e^{2 \pi i x \cdot (\xi_1+\xi_2)} e^{2 \pi i y(\eta_1+\eta_2)} d\xi d \eta \\
&=\sum_k \int_{\rr{R}^{2n+2}} \hat{f}(\xi_1, \eta_1) \hat{g}(\xi_2, \eta_2) m(\xi_1, \xi_2) \hat{\varphi}_k(\eta_1) \hat{\psi}_k(\eta_2) \left( \int_{\rr{R}} \psi_k(s) e^{-2 \pi i s(\eta_1+\eta_2)} ds  \right) \\
& \quad \quad\cdot  e^{2 \pi i x \cdot (\xi_1+\xi_2)} e^{2 \pi i y(\eta_1+\eta_2)} d\xi d \eta=\sum_k \int_{\rr{R}} \psi_k(s) (T_m(P_k^{y-s}f, Q_k^{y-s}g)(x)) ds \\
&=\sum_k Q_k^2 T_m(P_k^yf, Q_k^yg)(x).
\end{align*}
}
\end{proof}
\end{proposition}

A final ingredient that we will need in the proof of Theorem \ref{tensor product BHT d-paraproducts_intro} is the following lemma, which appears in \cite{multi-parameter_hardy_spaces}:
\begin{lemma}
\label{norm bounded by square function}
Let $f \in \mathcal{S}(\rr{R}^n) $, and $1 \leq l \leq n$, and $\lbrace i_1, \ldots i_l   \rbrace \subset \lbrace 1, \ldots ,n  \rbrace$. Then 
\[
\|f\|_{L^p} \lesssim \|    \left(  \sum_{k_1, \ldots, k_l} | Q_{k_1}^{i_1} \ldots Q_{k_l}^{i_l} f  |^2   \right)^{1/2} \|_{L^p}
\]
for any $0<p<\infty$.
\end{lemma}

Lemma \ref{norm bounded by square function} above states that the $L^p$ norm of $f$ is bounded by the $L^p$ norm of a square function associated with the variables $x_{i_1}, \ldots, x_{i_l}$, even when $0< p \leq 1$. In the case $p>1$, it is well known that the two norms are equivalent. When $p<1$, the proof makes use of multi-parameter Hardy spaces.

\subsection{Proof of Theorem \ref{tensor product BHT d-paraproducts_intro}:}
\label{section proof tensor products}
\begin{proof}
We start with the proof in the case $BHT \otimes \Pi$, in order to make the presentation clear.

(a) Assume that $\Pi(f,g)=\sum_k Q_k\left( P_k f \cdot Q_k g \right)$. Then Proposition \ref{characterization of tensor with paraproduct} implies that $\ds  BHT \otimes \Pi(f, g)(x,y)=\sum_k Q_k^2 BHT \left(P_k^y f, Q_k^y g   \right)(x)$. Lemma \ref{norm bounded by square function} yields
\[
\big \| BHT \otimes \Pi  \big \|_{L^s\left( \rr{R}^2 \right)} \lesssim \big \| \left( \sum_k  \vert Q_k^2 BHT \left( P_k^yf, Q_k^y g   \right)      \vert ^2 \right)^{1/2}  \big \|_{L^s\left( \rr{R}^2 \right)}.
\]
For the paraproducts that we are considering, $\ds Q_k\left( P_k f \cdot Q_k g \right)(y)= P_k f(y) \cdot Q_k g(y) $, so we need to estimate 
\[
\| \left( \sum_k  \vert BHT ( P_k^y f, Q_k^y g) \vert ^2 \right)^{1/2}  \big \|_{L^s\left( \rr{R}^2 \right)}.
\] 
We first estimate the $L^s$ norm of $\ds x \mapsto \left( \sum_k  \vert BHT \left( P_k^y f, Q_k^y g   \right)(x) \vert ^2 \right)^{1/2}$, and Fubini will imply the desired result for $BHT \otimes \Pi$. Here we use the vector-valued extension for the bilinear Hilbert transform
\[
BHT: L^p\left( \ell^\infty \right)\times L^q\left( \ell^2\right) \to L^s\left( \ell^2\right),
\]
which holds whenever $(p, q, s) \in Range(BHT)$. More exactly, 
\begin{align*}
& \big \| BHT \otimes \Pi  \big \|_{L^s\left( \rr{R}^2 \right)} \lesssim \big \|  \big \|  \left( \sum_k  \vert BHT \left( P_k^y f, Q_k^y g   \right)(x) \vert ^2 \right)^{1/2}  \big\|_{L^s_x}   \big \|_{L^s_y}  \\
&\lesssim \big \|   \big \| \sup_k \vert  P_k^y f  \vert \|_{L^p_x}  \cdot \big \|  \left(  \sum_k \vert  Q_k^y g  \vert ^2\right)^{1/2}  \big \|_{L^q_x}  \big \|_{L^s_y}  \\
& \lesssim \big \|   \big \| \sup_k \vert  P_k^y f  \vert \|_{L^p_x}\big \|_{L_y^p}  \cdot \big \| \big \|  \left(  \sum_k \vert  Q_k^y g  \vert ^2\right)^{1/2}  \big \|_{L^q_x}  \big \|_{L^q_y}  \lesssim \| f \|_p \cdot \|g\|_q
\end{align*}
To get the conclusion, we are using Fubini again, and the boundedness of the maximal and square function operators.

(b) The case $\Pi(f, g)=\sum_k P_k \left( Q_k f, Q_k g \right)$ is more direct, but the ideas are similar. The functions $\varphi$ in the paraproduct definition are so that $\Pi(f, g)=\sum_k \left( Q_k f \cdot Q_k g \right)$, so we have
\[
BHT \otimes \Pi(f, g)(x,y)= \sum_k BHT(Q_k^y f, Q_k^y g)(x).
\]
Now we use the vector-valued extension $\ds BHT: L^p\left( \ell^2 \right) \times L^q\left( \ell^2 \right) \to L^s\left( \ell^1\right)$(which is well defined for any $(p, q, s) \in Range(BHT)$), together with Fubini and the boundedness of the square function to get 
\begin{align*}
& \big \| BHT \otimes \Pi  \big \|_{L^s\left( \rr{R}^2 \right)} \lesssim \big \|  \big \|  \sum_k  \vert BHT \left( Q_k^y f, Q_k^y g   \right)(x) \vert  \big\|_{L^s_x}   \big \|_{L^s_y} \lesssim \\
&\lesssim \big \|   \big \| \left(  \sum_k \vert  Q_k^y f  \vert ^2\right)^{1/2} \|_{L^p_x}  \cdot \big \|  \left(  \sum_k \vert  Q_k^y g  \vert ^2\right)^{1/2}  \big \|_{L^q_x}  \big \|_{L^s_y} \lesssim \|f \|_p \cdot \|g\|_q.
\end{align*}

The general case of Theorem \ref{tensor product BHT d-paraproducts_intro} is similar, but slightly more technical. We present it below for completeness. The paraproducts can be of three types, as seen in \eqref{eq:def paraprod}. This generates a partition of $\lbrace 1, \ldots, n \rbrace$ into three subsets of indices $\mathcal{I}_1$, $\mathcal{I}_2$ and $\mathcal{I}_3$ so that if $k \in \mathcal{I}_1$, then $\ds \Pi(f,g)(y)=\sum_k Q_k(P_k f \cdot Q_k g)(y) $, and similarly for $\mathcal{I}_2$ and $\mathcal{I}_3$.

Because the projections on different coordinates commute, i.e. $Q_{k}^i P_l^j=P_l^j Q_k^i$ and $Q_{k}^i Q_l^j=Q_l^j Q_k^i$, we can assume 
$$\mathcal{I}_1=\lbrace 1, \ldots, l  \rbrace, \quad \mathcal{I}_2=\lbrace l+1, \ldots, l+d  \rbrace, \quad \mathcal{I}_3=\lbrace l+d+1, \ldots, n \rbrace.$$
Of course, we allow the possibility that one or even two of these sets of indices are empty.
With this assumption, Proposition \ref{characterization of tensor with paraproduct} applied iteratively yields
\begin{align*}
& BHT\otimes \Pi \otimes \ldots \otimes\Pi(f, g)(x, y_1, \ldots, y_n)\\
&=\sum_{k_1, \ldots, k_n} Q_{k_1}^1\ldots Q_{k_l}^l Q_{k_{l+1}^{l+1}} \ldots Q_{k_{l+d}}^{l+d} P_{k_{l+d+1}}^{l+d+1} \ldots P_{k_n}^n \circ \\
& BHT(
P_{k_1}^{y_1}\ldots P_{k_l}^{y_l} Q_{k_{l+1}}^{y_{l+1}}\ldots Q_{k_n}^{y_n} f, Q_{k_1}^{y_1} \ldots Q_{k_l}^{y_l} P_{k_{l+1}} \ldots P_{k_{l+d}}^{y_{l+d}} Q_{k_{l+d+1}^{y_{l+d+1}}} \ldots Q_{k_n}^{y_n}g)(x).
\end{align*}

The outer-most expressions $\ds Q_{k_1}^1\ldots Q_{k_l}^l Q_{k_{l+1}}^{l+1} \ldots Q_{k_{l+d}}^{l+d} P_{k_{l+d+1}}^{l+d+1} \ldots P_{k_n}^n$ are extremely important. Expressions of the type $P_k$ will be associated with $\ell^1$ norms, and the $Q_k$s with $\ell^2$ norms and square functions. Here we want to apply Proposition \ref{norm bounded by square function}, so we need to deal with the $Q_k$ functions first. Once we do this, we can estimate the $L^r$ norm of $\ds BHT\otimes \Pi \otimes \ldots \Pi(f, g)$ by
{
\fontsize{9}{10}
\begin{align*}
&\| \left( \sum_{k_1,\ldots, k_{l+d} } \big\vert \sum_{k_{l+d+1}, \ldots, k_n} P_{k_{l+d+1}}^{l+d+1}   \ldots P_{k_{n}}^n BHT(P_{k_1}^{y_1}\ldots Q_{k_{l+1}}^{y_{l+1}}\ldots  f, Q_{k_1}^{y_1} \ldots P_{k_{l+1}}^{y_{l+1}} \ldots Q_{k_{l+d+1}^{y_{l+d+1}}} \ldots g)   \big\vert^2     \right)^{1/2} \|_r \\
& \lesssim \| \left( \sum_{k_1,\ldots, k_{l+d} } \big \vert \sum_{k_{l+d+1}, \ldots, k_n} | BHT(P_{k_1}^{y_1}\ldots Q_{k_{l+1}}^{y_{l+1}}\ldots f, Q_{k_1}^{y_1} \ldots P_{k_{l+1}}^{y_{l+1}} \ldots Q_{k_{l+d+1}^{y_{l+d+1}}} \ldots g)|   \big \vert^2  \right)^{1/2} \|_r\\
& \lesssim \|f\|_{p} \|g\|_q
\end{align*}
}
For the last part we used the following vector-valued estimates for the $BHT$:
\begin{align*}
&L^p(\underbrace{\ell^{\infty}(\ldots (\ell^{\infty}}_l(\underbrace{\ell^2( \ldots (\ell^2}_{d} (\underbrace{\ell^2( \ldots (\ell^2}_{n-l-d} ))\ldots)
 \times L^q(\underbrace{\ell^{2}(\ldots (\ell^{2}}_l(\underbrace{\ell^\infty( \ldots (\ell^\infty}_{d} (\underbrace{\ell^2( \ldots (\ell^2}_{n-l-d} ))\ldots)\\
 & \mapsto L^s(\underbrace{\ell^{2}(\ldots (\ell^{2}}_l(\underbrace{\ell^2( \ldots (\ell^2}_{d} (\underbrace{\ell^1( \ldots (\ell^1}_{n-l-d} ))\ldots)
\end{align*}
together with the boundedness of the maximal operator and square function.

Similarly, we can obtain estimates for $\Pi^{\otimes^{d_1} } \otimes BHT \otimes \Pi^{\otimes ^{d_2}}$ within the same range as that of $BHT$. Some partial results in mixed norm $L^p$ spaces can be obtained too, but the general case, for arbitrary values of $d_1$ and $d_2$ remains open. We present a few particular cases that illustrates the main ideas, without being too technical. 
\begin{itemize}
\item[i)] Here, we prove mixed norm $L^p$ estimates for $\ds \Pi_1 \otimes BHT \otimes \Pi_3$, where $\ds \Pi_1= \sum_k Q_k^1(P_k^1 \cdot Q_k^1)$, $\ds \Pi_3= \sum_l Q_l^3(Q_l^3 \cdot P_l^3)$, and the exponents $p_j, q_j \in \left[ 2, \infty  \right)$. We note that 
\[
\Pi_1 \otimes BHT \otimes \Pi_3(f, g)(x,y,z)=\sum_{k,l} Q_k^1 Q_l^3 BHT(P_k^x Q_l^z f, Q_k^x P_l^z g)(y),
\]
and we want to estimate the above expression in the space $\ds \| \cdot \|_{L_x^{s_1}L^{s_2}_yL_z^{s_3}}$. The key observation is that whenever $ 1< s_2, s_3<\infty$,
\begin{equation}
\label{eq:vv-HardyIneq}
\left \|  \sum_{k, l} Q_k^1 Q_l^3 F(x, y, z)   \right\|_{L_x^{s_1}L^{s_2}_yL_z^{s_3}} \lesssim \left\| \left( \sum_{k, l} \vert Q_k^1 Q_l^3 F(x, y, z)    \vert ^2 \right)^{1/2}   \right \|_{L_x^{s_1}L^{s_2}_yL_z^{s_3}},
\end{equation}
which is a Banach-valued equivalent of Lemma \ref{norm bounded by square function}. This result, for $s_1>1$, can be found in \cite{vv_singular_integrals_product_kernel} and \cite{vv_CZ}, and it follows from the boundedness of Calder\'{o}n-Zygmund operators (the dual of the square function is such an operator) on $L^p$ spaces with mixed norms. The proof in the case $s_1 \leq 1$ is a Banach space adaptation of the proof of Lemma \ref{norm bounded by square function}. Given the special properties of the $Q_k^1$ and $Q_l^3$ operators, we obtain
\[
\big \| \Pi_1 \otimes BHT \otimes \Pi_3(f, g)\big \|_{L_x^{s_1}L^{s_2}_yL_z^{s_3}} \lesssim \| \left( \sum_{k, l} \vert BHT(P_k^x Q_l^z f, Q_k^x P_l^z g)(y)\vert ^2 \right)^{1/2}\big \|_{L_x^{s_1}L^{s_2}_yL_z^{s_3}}.
\]

The multiple vector-valued estimates $\ds BHT: L^{p_2}_y(L^{p_3}_z(\ell^\infty(\ell^2)))) \times L^{q_2}_y(L^{q_3}_z(\ell^2(\ell^\infty)))) \to L^{s_2}_y(L^{s_3}_z(\ell^2(\ell^2))))$, which exist in the local $L^2$ case at least, together with H\"{o}lder's inequality imply
\begin{align*}
&\big \| \Pi_1 \otimes BHT \otimes \Pi_3(f, g)\big \|_{L_x^{s_1}L^{s_2}_yL_z^{s_3}} \\
& \quad \lesssim \| \sup_k \left( \sum_l \vert P_k^x Q_l^z f (y) \vert^2 \right)^{1/2}  \|_{L_x^{p_1}L^{p_2}_yL_z^{p_3}}  \cdot \| \left( \sum_k  \vert \sup_l \vert Q_k^x P_l^z g (y)\vert \vert^2 \right)^{1/2}  \|_{L_x^{q_1}L^{q_2}_yL_z^{q_3}} \\
&\quad \lesssim \|f\|_{L_x^{p_1}L^{p_2}_yL_z^{p_3}} \cdot \|g\|_{L_x^{q_1}L^{q_2}_yL_z^{q_3}}.
\end{align*}
The last inequality follows again from Banach-valued extensions of convolution operators. Since our proof makes use of multiple vector-valued estimates for $BHT$, we cannot obtain mixed norm $L^p$ estimates for all the exponents in the Banach range. From the above example, one can see that besides the constraints imposed by the square functions and maximal operators, we also need $(p_3, q_3, s_3) \in \ii{D}_{p_2, q_2, s_2}$.
\item[ii)] If $d_1=0$ and $d_2=1$, we have
\[
 BHT \otimes \Pi: L_x^{p_1}L_y^{p_2} \times L_x^{q_1}L_y^{q_2} \to L_x^{s_1}L_y^{s_2},
\]
whenever $1<p_2, q_2, s_2<\infty$, $1<p_1, q_1 \leq \infty, \frac{2}{3} < s_1 <\infty$ and $\left(p_2, q_2, s_2\right) \in \ii D_{p_1, q_1, s_1}$.
\item[iii)]If $d_1=1$ and $d_2=0$, we have
\[
\Pi\otimes BHT : L_x^{p_1}L_y^{p_2} \times L_x^{q_1}L_y^{q_2} \to L_x^{s_1}L_y^{s_2},
\]
whenever $1<p_2, q_2, s_2<\infty$, $1<p_1, q_1 \leq \infty, \frac{1}{2} < s_1 <\infty$.
Since the ``target'' spaces (that is, inner spaces in the mixed norms) are strictly between $1$ and $\infty$, the outer $L^\infty$ cases (that is, $p_1=\infty$ or $q_1=\infty$) follow easily from similar estimates on the adjoints.
\end{itemize}

We note that mixed norm estimates for $\Pi \otimes BHT$ appear also in \cite{francesco_UMDparaproducts}, where all the inner spaces involved are $L^p$ spaces with $1<p< \infty$ (in our notation, that means $1<p_2, q_2, s_2<\infty$).
\end{proof}

\section{Leibniz rules: Theorem \ref{Leibniz rule}}
\label{sec: Leibniz rule}
Now we present some ideas behind the proof of Theorem \ref{Leibniz rule}. Littlewood-Paley projections play an important role when dealing with derivatives.
\begin{align*}
D_1^\alpha D_2^\beta (f \cdot g)(x,y)&= \sum_{k,l} \left[ (f \ast \varphi_k \otimes \varphi_l) \cdot (g \ast  \psi_k \otimes \psi_l   )\right] \ast \left(  D_1^\alpha \psi_k \otimes D_2^\beta \psi_l \right)(x,y) \\
&=\sum_{k,l}\left[ \left( f \ast \varphi_k \otimes \varphi_l  \right) \cdot \left( g \ast \psi_k \otimes \psi_l  \right)  \right] \ast
\left( 2^{k \alpha} \tilde{\psi}_k \otimes 2^{l \beta} \tilde{\psi}_l  \right)(x,y),
\end{align*}
where $\ds \quad  \widehat{\tilde{\psi}}_k(\xi) =\frac{|\xi|^{\alpha}}{2^{k \alpha}} \hat{\psi}_k(\xi) \quad $ and $\ds \quad \widehat{\tilde{\psi}}_l(\eta) =\frac{|\eta|^{\beta}}{2^{l \beta}} \hat{\psi}_l(\eta).$ Then one can move the $2^{k \alpha}$ inside, and couple it with the $\psi_k$s because $\ds 2^{k \alpha} \psi_k(x)=D^\alpha \tilde{\tilde{\psi}}_k(x)$. Here $\ds \widehat{\tilde{\tilde{\psi}}_k}(\xi)=\frac{2^{k \alpha}}{\vert \xi \vert^\alpha} \hat{\psi}_k(\xi)$.

In this way, we obtain $\ds D_1^\alpha D_2^\beta(f \cdot g)= \tilde{\Pi} \otimes \tilde{\Pi} (f, D_1^\alpha D_2^\beta g )+$ eight other similar terms. We can estimate $\ds \Pi \otimes \Pi$ in $L^p$ spaces with mixed norms, as long as the ``outside" functions $\hat{\psi}_k$ and $\hat{\varphi}_k$ are constantly equal to $1$ on $2^{k-2} \leq \vert  \xi \vert \leq 2^{k+2}$, and $\vert \xi \vert \leq 2^{k+2}$, respectively. The operators $\tilde{\Pi}$ are slightly different, but using Fourier series we can write $\tilde{\Pi}(F, G)$ as
\begin{align*}
(F,  G) \mapsto \sum_{n \in \rr{Z}} c_n \sum_{k,l} \left[  F \ast ( \varphi_k\otimes \varphi_l) \cdot G \ast (\tilde{\tilde{\psi}}_k \otimes \tilde{\tilde{\psi}}_l) \right] \ast \psi_k \otimes \tilde{\psi}_{l,n}(x,y).
\end{align*}
Here the coefficients $|c_n| \lesssim n^{-M}$, and $\ds \psi_{k,n}(x)=\psi_k(x+2^{-k}n)$. Now notice that the RHS above becomes
\begin{align*}
\sum_n c_n \sum_l Q_l^2 \tilde{\Pi}(P^y_{l,n} F, \tilde{\tilde{Q}}^y_{l,n} G)(x),
\end{align*}
which is a superposition of $\Pi \otimes \Pi$ operators.

The proof of the Leibniz rule follows from 
\begin{itemize}
\item[1)](multiple) vector-valued estimates for the paraproduct
\[
\tilde{\Pi}(f,g)=\sum_l \left[ (f \ast \varphi_l)\cdot (g \ast \tilde{\tilde{\psi}}_l)   \right] \ast \tilde{\psi}_l
\]
\item[2)] the boundedness of the shifted maximal and square functions:
\[
\|  \sup_l \vert f \ast \varphi_{l,n}\vert \| _{p} \lesssim \log \langle n\rangle \|f\|_{p}, \quad \| \left(  \sum_l  |f \ast \tilde{\tilde{\psi}}_{l,n}|^2  \right)^{1/2}\|_p \lesssim \log \langle n\rangle \|f\|_p.
\] 
\end{itemize}

Returning to the Leibniz rules, we have for $s_1, s_2 \geq 1$
\begin{align*}
&\|  \|  D_1 ^\alpha D_2^\beta(f,g)\|_{L_y^{s_2}}\|_{L_x^{s_1}}\leq \sum_n |c_n| \|   \|  \sum_l Q_l^2 \tilde{\Pi}(P^y_{l,n} F,\tilde{\tilde{ Q}}_{l,n}^y G)     \|_{L_y^{s_2}}\|_{L_x^{s_1}}  \\
\lesssim & \sum_n |c_n|  \|  \| \left( \sum_l  \vert \tilde{\Pi}^{\beta_1, \beta_2}(P^y_{l,n} F, \tilde{\tilde{Q}}_{l,n}^y G)  \vert^2  \right)^{1/2}\|_{L_y^{s_2}}\|_{L_x^{s_1}} \\
\lesssim & \sum_n |c_n|  \| \| \sup_l \vert P_{l,n}^y F  \vert \|_{L_y^{p_2}}\|_{L_x^{p_1}} \|  \|  \left( \sum_l \vert \tilde{\tilde{Q}}_{l,n}^yG  \vert^2 \right)^{1/2} \|_{L_y^{q_2}}\|_{L_x^{q_1}}\\
\lesssim & \|f  \|_{L_x^{p_1}L_y^{p_2}}    \|D_1^{\alpha}D_2^{\beta}g \|_{L_x^{q_1}L_y^{q_2}}.
\end{align*}
Here we used the vector-valued estimates 
\[
\tilde{\Pi}: L^{p_1}_x\left( L^{p_2}_y \left( \ell^\infty \right)   \right)\times  L^{q_1}_x\left( L^{q_2}_y \left( \ell^2 \right)   \right) \to  L^{s_1}_x\left( L^{s_2}_y \left( \ell^2\right)   \right),
\]
as well as the boundedness of the square function and maximal operator. We note that the square function is in the $y$ variable, and for that reason at first we cannot allow $p_2=\infty$ or $q_2=\infty$. However, this obstruction can be removed by using duality.

The same proof works in the case $\frac{1}{2}<s_1<1$, if $1< p_2, q_2<\infty$. In this case, we use the subadditivity of $\| \cdot \|_{s_1}^{s_1}$. The case $\frac{1}{2}<s_1<1$ and $p_2=\infty$ requires a slightly different reasoning, and can be deduced from the corresponding mixed norm estimates for $\Pi \otimes \Pi$. This will be presented at the end of this section.

A slightly more difficult case of the Leibniz rule is when one of the last components is a $\varphi-$type function:
\begin{align*}
D_1^\alpha D_2^\beta (f \cdot g)(x,y)& \sim \sum_{k,l} \left[ (f \ast \psi_k \otimes \varphi_l) \cdot (g \ast  \psi_l \otimes \psi_l   )\right] \ast \left(  D_1^\alpha \varphi_k \otimes D_2^\beta \psi_l \right)(x,y) \\
&=\sum_{k,l}\left[ \left( f \ast \psi_k \otimes \varphi_l  \right) \cdot \left( g \ast \psi_k \otimes \psi_l  \right)  \right] \ast
\left( 2^{k \alpha} \tilde{\varphi}_k \otimes 2^{l \beta} \tilde{\psi}_l  \right)(x,y).
\end{align*}
In this case $\ds \widehat{\tilde{\varphi}}_k(\xi)=\frac{|\xi|^\alpha}{2^{k \alpha}} \hat{\varphi}_k(\xi)$, but $\tilde{\varphi}$ doesn't behave as nicely as $\tilde{\psi}$; it is not smooth at the origin, and for that reason its decay is much slower:
\[
\big \vert \tilde{\varphi}(x)  \big \vert \leq \frac{1}{ \left( 1+|x| \right)^{1+\alpha}}.
\]
We use a Fourier series decomposition of $\widehat{\tilde{\varphi}}_k$ on its support
\[
\widehat{\tilde{\varphi}}_k(\xi)=\sum_{n \in \rr{Z}} c_n e^{\frac{2 \pi i n \xi}{2^k}} \cdot \hat{\varphi}_k(\xi), \quad \text{where } \quad c_n=\frac{1}{2^k} \int_{\rr{R}} \widehat{\tilde{\varphi}}_k(\xi) e^{-\frac{2 \pi i n \xi}{2^k}} d \xi.
\]
In this case we only have $\ds \vert  c_n  \vert \leq \frac{1}{n^{1+\alpha}}$, but this is enough for the coefficients to sum up, if $s_1>\dfrac{1}{1+\alpha}$. Since $s_2 \geq 1$, we will not have a similar issue when doing the decomposition in the second variable.

Following the same line of ideas, the problem reduces to estimating 
\[
\sum_n c_n \sum_k P_k^1 \tilde{\Pi}(\tilde{\tilde{Q}}^x_{k,n} F, Q^x_{k,n}G)(y),
\]
and it would imply ``mixed square functions" estimates of the form 
\[
\big \| \left( \sum_n \vert Q^x_{k,n} G  \vert^2 \right)^{1/2}   \big \|_{L_x^{q_1}L_y^{q_2}}.
 \]
This is bounded as long as $1<q_1, q_2< \infty$, and in order to recover the case $p_i=\infty$ or $q_i=\infty$ we want to make sure that the square functions are in the inner-most variable, which is $y$. So we need a decomposition of $\tilde{\psi}_l$, as before. Also, we will need vector-valued estimates for the ``generalized paraproduct" 
\[
(f, g) \mapsto \sum_k \left( f \ast \psi_k \cdot g \ast \psi_k    \right) \ast \tilde{\varphi_k},
\]
where the last component $\tilde{\varphi}$ has slow decay. The vector spaces involved are $(\ell^2, \ell^\infty, \ell^2)$ or $(\ell^2, \ell^2, \ell^1)$, and such estimates can be proved using ideas similar to those in Section \ref{paraproduct results}, modulo standard technical difficulties, as discussed in \cite{multilinear_harmonic}. 

We now present the proof of the mixed norm estimates for the biparameter paraproducts:

\begin{proof}[Proof of  Theorem \ref{bi-parameter paraproducts in L^p spaces with mixed norms}]
 
 Since the other case are very similar, we can assume that $\Pi_y$, the paraproduct acting on the variable $y$ is of the form 
\[
\Pi_y\left( \cdot , \cdot \right)= \sum_k Q_k \left( P_k \left( \cdot  \right) , Q_k \left( \cdot \right) \right).
\]  
Then we can write $\Pi \otimes \Pi$ as $\Pi \otimes \Pi(f, g)(x, y)=\sum_k Q_k^2 \Pi \left( P_k^y, Q_k^y \right)(x)$. Then we have
\begin{align*}
&\left\|    \left\|    \sum_k Q_k^2 \Pi \left( P_k^y, Q_k^y \right)(x)    \right\|_{L_y^{s_2}}   \right\|_{L_x^{s_1}} \lesssim 
\left\|    \left\|  \left(  \sum_k  \lft \Pi \left( P_k^y, Q_k^y \right)(x)\rg^2 \right)^{1/2} \right\|_{L_y^{s_2}}   \right\|_{L_x^{s_1}}\\
& \lesssim  \left\| \left\|  \sup_k \lft  P_k^y f(x) \rg \right\|_{L_y^{p_2}} \right\|_{L_x^{p_1}} \cdot  \left\|    \left\|    \left( \sum_k \lft Q_k^y g(x)  \rg^2  \right)^{1/2}    \right\|_{L_y^{q_2}}   \right\|_{L_x^{q_1}}.
\end{align*}

In the above inequality we used the multiple vector-valued estimate
\[
\Pi_x : L_x^{p_1}\left(  L_y^{p_2} \left(  \ell^\infty \right) \right) \times  L_x^{q_1}\left(  L_y^{q_2} \left(  \ell^2 \right) \right) \to L_x^{s_1}\left(  L_y^{s_2} \left(  \ell^2 \right) \right),
\]
which is a consequence of Theorem \ref{multiple vector valued paraproducts}.

Now we focus on the case $p_2=\infty, 1<q_2=q<\infty$, since $q_2=\infty$ is symmetric. We want to prove that 
\[
\Pi \otimes \Pi : L_x^{p_1}L_y^{\infty} \times L_x^{q_1}L_y^{q} \to L_x^{s_1}L_y^q,
\]
by using Banach-valued restricted type interpolation. That is, for any $F, G, H$ sets of finite measure, we can find a major subset $H' \subseteq H$, and we will prove that 
\begin{equation}
\label{eq:ineq-paraprod-bidisc-restricted} 
\lft \int_{\rr R ^2} \Pi \otimes \Pi \left( f, g \right)(x,y) h(x,y) dx dy \rg \lesssim \lft F \rg^{\alpha_1} \lft G \rg^{\alpha_2} \lft H  \rg^{\alpha_3}
\end{equation}
for any functions $f, g$ and $h$ satisfying 
\[
\left\| f\left(x, \cdot\right)  \right\|_{L_y^\infty} \leq \one_{F}(x), \qquad \left\| g\left(x, \cdot\right)  \right\|_{L_y^q} \leq \one_{G}(x), \qquad \left\| h\left(x, \cdot\right)  \right\|_{L_y^{q'}} \leq \one_{H'}(x),
\]
and $\left( \alpha_1, \alpha_2, \alpha_3  \right)$ any tuple satisfying $\alpha_1+\alpha_2+\alpha_3=1$, situated in the neighborhood of $\left( \frac{1}{p_1}, \frac{1}{p_2}, \frac{1}{p'}\right)$.
 
A triple stopping time similar to the one appearing in the proof of Theorem \ref{vector valued BHT} will allow us to recover any exterior $L^{p_j}_x$ norms, while the interior norms are fixed: $L^\infty_y, L^q_y, L^q_y$.

We will consider \emph{localizations} of the paraproduct acting on the $x$ variable. More exactly, the following estimate, the proof of which is a combination of Proposition \ref{Localization for local L^1 paraproducts} and $L^p$ estimates for $\Pi \otimes \Pi$, is key:

If $I_0$is a fixed dyadic interval, then $\ds \Pi_{I_0}^{F, G, H'} \otimes \Pi : L_x^{\infty}L_y^\infty \times L_x^q L_y^q \to L_x^q L_y^q$ with operatorial norm
\[
\left\| \Pi_{I_0}^{F, G, H'} \otimes \Pi \right\|_{L_x^\infty L_y^\infty \times L_x^q L_y^q \to L_x^q L_y^q} = \left\| \left(\Pi_{I_0}^{F, G, H'} \otimes \Pi \right)^{*, 1}\right\|_{L_x^{q'} L_y^{q'}\times L_x^q L_y^q \to L_x^1 L_y^1}.
\]

The latter is bounded above by 
\[
 \left\| \left(\Pi_{I_0}^{F, G, H'} \otimes \Pi \right)^{*, 1}\right\|_{L_x^{q'} L_y^{q'}\times L_x^q L_y^q \to L_x^1 L_y^1} \lesssim \ \left( \sssize_{I_0} \one_{H'}\right)^{\frac{1}{q}-\epsilon}\left( \sssize_{I_0} \one_{G}\right)^{\frac{1}{q'}-\epsilon} \left( \sssize_{I_0} \one_{F}\right)^{1-\epsilon},
\]
which is a consequence of the localized multiple vector-valued estimates that always appear in the iterative step of the helicoidal method.

More exactly, we have
\begin{align*}
\lft \Pi_{I_0}^{F, G, H} \otimes \Pi (f, g)(x,y) h(x,y) dx dy\rg &\lesssim  
\left( \sssize_{I_0} \one_{H'}\right)^{\frac{1}{q}-\epsilon}\left( \sssize_{I_0} \one_{G}\right)^{\frac{1}{q'}-\epsilon} \left( \sssize_{I_0} \one_{F}\right)^{1-\epsilon} \\
& \left\| \left\| h(x, \cdot) \right\|_{L_y^{q'}} \cdot \ci_{I_0} \right\|_{L_x^{q'}} \cdot \left\| \left\| g(x, \cdot) \right\|_{L_y^{q}} \cdot \ci_{I_0} \right\|_{L_x^{q}} \left\| f(\cdot, \cdot) \right\|_{L_x^\infty L_y^\infty}.
\end{align*}
This implies, after performing the usual stopping times that
\begin{align*}
& \lft \int_{\rr R ^2} \left( \Pi \otimes \Pi  \right)(f, g)(x,y) h(x, y) dx dy \rg \lesssim \sum_{n_1, n_2, n_3} \sum_{I_0} \lft \int_{\rr R ^2} \left( \Pi_{I_0}^{F, G, H'} \otimes \Pi  \right)(f, g)(x,y) h(x, y) dx dy \rg\\
&\lesssim \sum_{n_1, n_2, n_3} \sum_{I_0} \left( \sssize_{I_0} \one_{F}\right)^{1-\epsilon} \left( \sssize_{I_0} \one_{G}\right)^{1-\epsilon} \left( \sssize_{I_0} \one_{H'}\right)^{1-\epsilon} \cdot \lft I_0 \rg.
\end{align*}
From here, the desired $L^p$ estimates follow almost immediately. 

\end{proof}

\section{Rubio de Francia Theorem for Iterated Fourier Integrals}
\label{Rubio de Francia Theorem for Iterated Fourier Integrals}

We end by answering the initial question that motivated the study of vector-valued $BHT$. More exactly, we prove Theorem \ref{main theorem}, which is a consequence of Theorem \ref{vector valued BHT}, with $r_1, r_2$ chosen carefully so that $\ds \frac{1}{r_1}+\frac{1}{r_2}=\frac{1}{r}$.

\begin{proof}[Proof of Theorem \ref{main theorem}]
We start with the case $r\geq 2$; this follows from Theorem \ref{vector valued BHT}:
\begin{equation}
\label{zzz}
\| \left( \sum_k |BHT(P_{I_k}f, P_{I_k}g)(x)|^2  \right)^{1/2}\|_s \lesssim \|\left(\sum_k |P_{I_k} f|^{r_1} \right)^{1/{r_1}}\|_p \|\left(\sum_k |P_{I_k} g|^{r_2} \right)^{1/{r_2}}\|_q, 
\end{equation}
for any $1<p, q< \infty, \frac{2}{3}<s < \infty$.

This is implied by Rubio de Francia's theorem, if one can find $r_1$ and $r_2$ with $\ds \frac{1}{r_1}+\frac{1}{r_2}=\frac{1}{2}$ and
\[
\frac{1}{p}< \frac{1}{r_1'}, \quad \frac{1}{q}<\frac{1}{r_2'}.
\]
This is possible as long as $\ds\frac{1}{s} =\frac{1}{p}+\frac{1}{q}< \frac{1}{r_1'}+\frac{1}{r_2'}=\frac{3}{2}$, which coincides with the condition that we have for the range of $BHT$.

The case $1\leq r <2$ is similar; for $p, q, $ and $s$ as above, one needs to find $r_1$ and $r_2 \geq 2$ so that
\[
2- \frac{1}{r}=\frac{1}{r_1'}+\frac{1}{r_2'}> \frac{1}{p}+\frac{1}{q}.
\]
Note that  $ \ds \frac{1}{p}< \frac{1}{r_1'} = 1-\frac{1}{r}+\frac{1}{r_2} \leq \frac{1}{r'}+\frac{1}{2}$, and similarly for $q$. Because of this restriction, the operator $T_r$ is bounded as long as admissible triple $\ds (\frac{1}{p}, \frac{1}{q}, \frac{1}{s'})$ is in the convex hull of the points
\[
\left( 0, 0, 1 \right), \left( \frac{1}{2}+ \frac{1}{r'}, \frac{1}{2}, -\frac{1}{r'} \right), \left( \frac{1}{2}, \frac{1}{2}+\frac{1}{r'}, -\frac{1}{r'} \right), \left( \frac{1}{2}+ \frac{1}{r'}, 0, \frac{1}{2} -\frac{1}{r'} \right), \left( 0, \frac{1}{2}+\frac{1}{r'},  \frac{1}{2}-\frac{1}{r'} \right).
\]

\end{proof}

\begin{remark}
An alternative way of proving the boundedness of $T_r$ within the range mentioned in Theorem \ref{main theorem} is by interpolating between
\begin{align}
&L^{p_1}\times L^{q_1} \to L^{s_1}(\ell^2) \quad \text{with } p_1, q_1, s_1 \text{ in the range of the $BHT$ operator  and} \label{l^2$ estimate}\\
&L^{p_2}\times L^{q_2} \to L^{s_2}(\ell^1) \quad \text{with } p_2, q_2>1,  s_2 \geq 1. \label{l^1 estimate}
\end{align}
\end{remark}

\subsection{Boundedness of operators $M_1$ and $M_2$}

In what follows we prove the boundedness of operators $M_1$ and $M_2$ presented in \eqref{breaking op 1} and \eqref{breaking op 2}:
\[
M_1(f_1, f_2, g)(\xi)=\sum_{\omega} \int_{\substack{ x_1< x_2, x_1, x_2 \in \omega_L \\ x_3 \in \omega_R }} \hat{f}_1(x_1) \hat{f}_2(x_2) g(x_3) e^{2 \pi i \xi (x_1+x_2+x_3)} d x_1 d x_2 d x_3 
\]
and 
\[
M_2(f_1, f_2, g)(\xi)=\sum_{\omega} \int_{\substack{x_1< L(\omega_L), x_2 \in \omega_L \\ x_3 \in \omega_R}} \hat{f}_1(x_1) \hat{f}_2(x_2) g(x_3) e^{2 \pi i \xi (x_1+x_2+x_3)} d x_1 d x_2 d x_3
\]

For both operators, we are going to use the triangle inequality in $L^r$, the target space for operators $M_1$ and $M_2$. However, if $r<1$ this inequality is not available anymore for the quasi-norm $\| \cdot \|_r$ and instead we use the triangle inequality  for $\| \cdot \|_r^r$. This is the only difference between the Banach and quasi-Banach case, and for simplicity we assume $r \geq 1$. Also, as previously stated, we assume $\ds \|g\|_p=1$.

\begin{proposition}
\label{boundedness of M_1}
Let $1<p<2$ and $\ds \frac{1}{r}=\frac{1}{s}+\frac{1}{p'}=\frac{1}{p_1}+\frac{1}{p_2}+\frac{1}{p'}$. Then
\[
\| M_1(f_1, f_2, g) \|_r \lesssim \|f_1\|_{p_1} \|f_2\|_{p_2} \|g\|_{p}.
\]
\begin{proof}
Recall that $\omega \in \mathcal{D}$ is the mesh of dyadic intervals contained in $[0,1]$, and we identify them with their preimage: $\omega \sim \varphi^{-1}(\omega)$. We rewrite $M_1$ as
\[
M_1(f_1, f_2, g)(\xi)=\sum_{\omega} BHT(P_{\omega_L} f_1, P_{\omega_L} f_2)(\xi) \cdot \widehat{g \cdot \one_{\omega_R}}(\xi).
\]
 Then
\begin{align*}
\| M_1(f_1, f_2, g)  \|_r &\lesssim \sum_{k \geq 0} \| \sum_{|\omega|=2^{-k}} BHT(P_{\omega_L} f_1, P_{\omega_L} f_2) \cdot \widehat{g \cdot \one_{\omega_R}} \|_r  \\
& \lesssim \sum_{k \geq 0} \| \left(  \sum_{|\omega|=2^{-k}} \vert BHT(P_{\omega_L} f_1, P_{\omega_L} f_2)   \vert ^{p}  \right)^{1/p} \cdot \left( \sum_{|\omega|=2^{-k}} \vert  \widehat{g \cdot \one_{\omega_R}}  \vert^{p'}  \right)^{1/{p'}}     \|_r  \\
&\lesssim \sum_{k \geq 0} \| \left(  \sum_{|\omega|=2^{-k}} \vert BHT(P_{\omega_L} f_1, P_{\omega_L} f_2)   \vert ^{p}  \right)^{1/p}  \|_s   \cdot \left( \sum_{|\omega|=2^{-k}} \|  \widehat{g \cdot \one_{\omega_R}} \|_{p'}^{p'}   \right)^{1/{p'}}
\end{align*}
We estimate $\ds \| \widehat{g \cdot \one_{\omega_R}}  \|_{p'} \lesssim \|  g \cdot \one_{\omega_R} \|_p=2^{- \frac{k}{p}}$ using the Hausdorff-Young theorem. Also, there are $2^k$ dyadic intervals of length $2^{-k}$ in $[0,1]$ and because of this
\[
\|M_1(f_1, f_2, g)\|_r \lesssim \sum_{k \geq 0} 2^{-k\left( \frac{1}{p}-\frac{1}{p'}   \right)} \|  \left(  \sum_{|\omega|=2^{-k}} \vert BHT(P_{\omega_L} f_1, P_{\omega_L} f_2)   \vert ^{p}  \right)^{1/p} \|_s.
\]
If we estimate the last term using the operator $T_p$ directly, we will not obtain the full range stated above, as there will appear extra constraints of the type 
\[
\frac{1}{p_1}+\frac{1}{p} < \frac{3}{2}, \quad \frac{1}{p_2}+\frac{1}{p} < \frac{3}{2}.
\]
Instead, using H\"{o}lder and the fact that $1<p<2$, we have $$\ds \|BHT(P_{\omega_L}f_1, P_{\omega_L} f_2) \|_{\ell^p(\omega)} \leq \|BHT(P_{\omega_L}f_1, P_{\omega_L} f_2) \|_{\ell^2(\omega)} \cdot 2^{k \left( \frac{1}{p}-\frac{1}{2}  \right)}.$$

Using the boundedness of $T_2$, we have $\ds \|M_1(f_1, f_2, g)\|_r \lesssim \sum_{k \geq 0} 2^{-k\left( \frac{1}{2}-\frac{1}{p'} \right)} \|f_1\|_{p_1} \|f_2\|_{p_2}$.
\end{proof}
\end{proposition}

\begin{proposition}
\label{boundedness of M_2}
Let $1<p<2$ and $\ds \frac{1}{r}=\frac{1}{s}+\frac{1}{p'}=\frac{1}{p_1}+\frac{1}{p_2}+\frac{1}{p'}$. Then
\[
\| M_2(f_1, f_2, g) \|_r \lesssim \|f_1\|_{p_1} \|f_2\|_{p_2} \|g\|_{p},
\]
provided $\ds \frac{1}{p_2}+\frac{1}{p'}<1$.
\begin{proof}
First, we remark that
\[
|M_2(f_1, f_2, g)(\xi)| \leq \sum_\omega |C f_1(\xi)| \cdot |P_{\omega_L} f_2(\xi)| \cdot |\widehat{g \cdot \omega_R}(\xi)|,
\] 
where $C$ is the Carleson operator, bounded on $L^p$ whenever $1< p < \infty$. From here on the estimates are similar to those in Proposition \ref{boundedness of M_1}, but instead of the bilinear operator $T_r(f,g)$ we will have to use the more restrictive Rubio de Francia operator $RF_\nu$:

\begin{align*}
&\|M_2(f_1, f_2, g)\|_r \leq \sum_{k \geq 0} \| C f_1 \cdot \left( \sum_{|\omega|=2^{-k}} \vert P_{\omega_L} f_2  \vert^p    \right)^{1/p} \cdot \left( \sum_{|\omega|=2^{-k}} \vert \widehat{g \cdot \one_{\omega_R}}  \vert^{p'}  \right)^{1/{p'}}  \|_r  \\
&\leq \sum_{k \geq 0} \|  C f_1 \|_{p_1} \cdot \|  \left( \sum_{|\omega|=2^{-k}} \vert P_{\omega_L} f_2  \vert^p    \right)^{1/p} \|_{p_2} \cdot \left( \sum_{|\omega|=2^{-k}} \| \widehat{g \cdot \one_{\omega_R}}  \|_{p'}^{p'}   \right)^{1/{p'}} \\
&\leq \sum_{k \geq 0}  2^{k \left( \frac{1}{p}-\frac{1}{\nu}  \right)}\|  C f_1 \|_{p_1} \cdot \|  \left( \sum_{|\omega|=2^{-k}} \vert P_{\omega_L} f_2  \vert^\nu    \right)^{1/\nu} \|_{p_2} \cdot \left( \sum_{|\omega|=2^{-k}} \| \widehat{g \cdot \one_{\omega_R}}  \|_{p'}^{p'}   \right)^{1/{p'}} \\
&\leq \sum_{k \geq 0} 2^{-k \left( \frac{1}{\nu}-\frac{1}{p'}  \right)} \|f_1 \|_{p_1} \| RF_{\nu}( f_2) \|_{p_2}.
\end{align*}

If $p_2 \geq 2$,we can take $\nu =2$ and there are no other restrictions. In the case $p_2< 2$, Rubio de Francia requires $\ds \frac{1}{\nu}+\frac{1}{p_2} < 1$. This and the condition $\ds \frac{1}{\nu}-\frac{1}{p'} >0$ (so that the geometric series above is finite) can be summarized as $\ds \frac{1}{p_2}+\frac{1}{p'}<1$.

\end{proof}
\end{proposition}

\

\bibliographystyle{alpha}
\bibliography{../../../../work/harmonic.bib}

\end{document}